\documentclass[a4paper,11pt]{article}
\usepackage{makeidx}
\usepackage[mathscr]{eucal}
\usepackage[dvips]{color}
\usepackage{amssymb, amsfonts, amsmath, amsthm, graphicx, cases} 
\usepackage{here}
\usepackage{tasks}
\newtheorem{proposition}{Proposition}[section]
\newtheorem{theorem}{Theorem}[section]

\newtheorem{lemma}[proposition]{Lemma}
\newtheorem{remark}{Remark}[section]
\newtheorem{corollary}[theorem]{Corollary}

\numberwithin{equation}{section}

\numberwithin{equation}{section}

\author{}
\date{}
\title{Non-existence of ground states 
and gap of variational problems for 
combined power-type 
nonlinear scalar field equations 
involving the Sobolev critical exponent 
in three space dimensions}

\author{Takafumi Akahori, Slim Ibrahim, 
Hiroaki Kikuchi, Hayato Nawa}
\date{}
\begin{document}
\maketitle

\begin{abstract}
In this paper, we consider minimization problems related to the 
combined power-type nonlinear scalar field equations 
involving the Sobolev critical exponent in three space dimensions. In four and higher space dimensions, it is known 
that for any frequency and any power of the subcritical nonlinearity, 
there exists a ground state. In contrast to those cases, 
when the space dimension is three and the subcritical power is three or less, we can 
show that there exists a threshold frequency, 
above which no ground state exists, 
and below which the ground state exists 
(see Theorems \ref{thm-0} and \ref{ex-critical-o}). 
Furthermore, we prove the difference between two typical variational problems used to characterize the ground states (see Theorem \ref{thm-2mini-1}). 
\end{abstract}

\section{Introduction}
In this paper, we consider the existence of 
minimizer for the following minimization problem:
\begin{equation}\label{m-value}
m^{\mathcal{N}}_{\omega} 
:= 
\inf\big\{\mathcal{S}_{\omega}(u) \colon 
u \in H^{1}(\mathbb{R}^{d}) 
\setminus \{0\}, \; \mathcal{N}_{\omega} (u) = 0 
\big\}, 
\end{equation} 
where $\mathcal{S}_{\omega}$ 
and $\mathcal{N}_{\omega}$ are the functionals 
on $H^{1}(\mathbb{R}^{d})$ defined by 
\begin{align}
& \mathcal{S}_{\omega}(u) := \frac{1}{2} \|\nabla u\|_{L^{2}}^{2} 
+ \frac{\omega}{2} \|u\|_{L^{2}}^{2} - \frac{1}{p+1} \|u\|_{L^{p+1}}^{p+1}
- \frac{d-2}{2d} \|u\|_{L^{\frac{2d}{d-2}}}^{\frac{2d}{d-2}}, 
\label{eq-action}\\
& \mathcal{N}_{\omega}(u) := \|\nabla u\|_{L^{2}}^{2} 
+ \omega \|u\|_{L^{2}}^{2} - \|u\|_{L^{p+1}}^{p+1}
- \|u\|_{L^{\frac{2d}{d-2}}}^{\frac{2d}{d-2}}, 
\end{align}
with $1<p<\frac{d+2}{d-2}$. We refer $\mathcal{S}_{\omega}$ and $\mathcal{N}_{\omega}$ as action and Nehari's functional, respectively. 
\par
The main reason why we consider the minimization problem \eqref{m-value} is that when a minimizer exists, it solves the following Sobolev-critical scalar field equation: 
\begin{equation}\label{sp}
- \Delta u + \omega u - |u|^{p-1}u - 
|u|^{\frac{4}{d-2}}u = 0 
\quad 
\mbox{in $H^{-1}(\mathbb{R}^{d})$}. 
\end{equation}
We note that 
the equation \eqref{sp} is related to the following 
nonlinear Schr\"{o}dinger equation 
\begin{equation} \label{nls}
i \frac{\partial \psi}{\partial t} + \Delta \psi + |\psi|^{p-1} \psi 
+ |\psi|^{\frac{4}{d-2}} \psi = 0 \qquad \mbox{
in $\mathbb{R}^{d} \times \mathbb{R}$}. 
\end{equation}
When we look for solutions to \eqref{nls} of the 
form $\psi(t, x) = e^{i \omega t} u(x) \; (\omega > 0)$, 
we see that $u$ satisfies \eqref{sp}. 

In this paper, we call a minimizer for 
$m_{\omega}^{\mathcal{N}}$ a 
\lq\lq ground state of $m_{\omega}^{\mathcal{N}}$"; 
The definition of ground state seems 
to vary depending on papers. 
Indeed, in our previous papers \cite{AIKN3, AIIKN}, 
a ground state is defined as a minimizer 
for the following problem: 
\begin{equation} \label{mini-sol}
m_{\omega}^{S}
:= \inf \big\{ \mathcal{S}_\omega(v) : 
\mbox{$v \in H^{1}(\mathbb{R}^d)$ 
is a nontrivial (complex-valued) 
solution to \eqref{sp}} 
\big\}.
\end{equation} 
Clearly, if $u \in H^{1}(\mathbb{R}^d)$ 
is a solution to \eqref{sp}, then 
$u \in H^{1}(\mathbb{R}^d)$ satisfies 
$\mathcal{N}_{\omega}(u) = 0$, so that 
$m_{\omega}^{\mathcal{N}} \le m_{\omega}^{S}$. 
Furthermore, as mentioned, a ground state (minimizer for $m_{\omega}^{\mathcal{N}}$) satisfies \eqref{sp}, 
so that 
$m_{\omega}^{\mathcal{N}} = m_{\omega}^{S}$ 
as long as a ground state exists. 

The aim of this paper is to study the existence and non-existence 
of ground state (minimizer for $m_{\omega}^{\mathcal{N}}$) in $\mathbb{R}^{3}$. 
To this end, we introduce a supplementary variational value $m_{\infty}$ as 
\begin{equation}\label{m-infty}
m_{\infty}
:=
\inf\big\{ \mathcal{H}^{\ddagger}(u) \colon 
u \in \dot{H}^{1}(\mathbb{R}^{d})\setminus \{0\},\ 
\mathcal{N}^{\ddagger}(u)=0 
\big\}
,
\end{equation}
where 
\begin{align} 
\label{h-infty}
\mathcal{H}^{\ddagger}(u)
&:=
\frac{1}{2}\|\nabla u\|_{L^{2}}^{2}
-
\frac{d-2}{2d}\|u \|_{L^{\frac{2d}{d-2}}}^{\frac{2d}{d-2}},
\\[6pt]
\label{n-infty}
\mathcal{N}^{\ddagger}(u)
&:=
\|\nabla u\|_{L^{2}}^{2}
-
\|u \|_{L^{\frac{2d}{d-2}}}^{\frac{2d}{d-2}}.
\end{align}
Putting here 
\begin{equation}\label{sob-bc}
\sigma := 
\inf\big\{\|\nabla u\|_{L^{2}}^{2} \colon 
\mbox{$u \in \dot{H}^{1}(\mathbb{R}^{d})$ with 
$\|u\|_{L^{\frac{2d}{d-2}}} = 1$}\big\},
\end{equation}
we can verify the following (see Section \ref{app-A} for the proof): 
\begin{equation}\label{sob-bc-rela}
m_{\infty}=\frac{1}{d}\sigma^{\frac{d}{2}}.
\end{equation}
Notice that the value $\sigma$ 
gives us the best constant of Sobolev's 
inequality in $\mathbb{R}^{d}$ (see \cite{Aubin, 
Talenti}): 
\begin{equation}\label{19/9/7/15:31}
\sigma \|u\|_{L^{\frac{2d}{d-2}}}^{2} \le 
\|\nabla u\|_{L^{2}}^{2}
\end{equation}
with equality if and only if 
\begin{equation}\label{scale-W}
u(x) = z \lambda_{0}^{-1} 
W( \lambda_{0}^{-\frac{2}{d-2}}x+x_{0})
\quad 
\mbox{for some $z \in \mathbb{C}$, 
$\lambda_{0}>0$ and 
$x_{0}\in \mathbb{R}^{d}$}, 
\end{equation}
where $W$ is the 
Talenti function on $\mathbb{R}^{d}$ 
with $W(0) = 1$, namely, 
\begin{equation}\label{F-Talenti-1}
W(x) := \Big(1 + \frac{|x|^{2}}{d(d-2)}\Big)^{- 
\frac{d-2}{2}}. 
\end{equation}
It is well known that for any $d \geq 3$, 
we have 
\begin{equation}\label{19/9/18/11:15}
\|\nabla W\|_{L^{2}}^{2}=\|W\|_{L^{\frac{2d}{d-2}}}^{\frac{2d}{d-2}}=\sigma^{\frac{d}{2}}.
\end{equation}
Moreover, the values $m_{\infty}$ is related to a sufficient condition for the existence of ground state (cf, \cite{Brezis-Nirenberg}). Indeed, we have the following Brezis-Nirenberg type proposition: 
\begin{proposition}\label{thm-0-1}
Assume $d \ge 3$ and 
$1<p<\frac{d+2}{d-2}$. Then, for any $\omega>0$ with $m_{\omega}^{\mathcal{N}} < m_{\infty}$, a ground state (minimizer for $m_{\omega}^{\mathcal{N}}$) exists. 
\end{proposition} 

\begin{remark}\label{remark-1}
It is known (see 
\cite{GNN} and Proposition 2.1 of \cite{AIKN3}) 
that for any ground state $Q_{\omega}$, 
there exist $\theta \in [0,2\pi)$, $y \in \mathbb{R}^{d}$ and a positive radial ground state $\Phi_\omega$ 
such that $Q_{\omega}=e^{i\theta }\Phi_{\omega}(\cdot-y)$. Thus, we may assume that ground states are positive radial solutions 
to \eqref{sp}. 
\end{remark}

We give a proof of Proposition \ref{thm-0-1} in Appendix \ref{exisistence-groundstate}. 
We can also find that $m_{\omega}^{\mathcal{N}}$ satisfies 
the following: 
\begin{proposition}\label{thm4-1}
Assume $d \geq 3$ and 
$1<p< \frac{d+2}{d-2}$. Then, the value $m_{\omega}^{\mathcal{N}}$ is nondecreasing with respect to $\omega$ on $(0,\infty)$. 
\end{proposition}

We will prove Proposition \ref{thm4-1} in Section \ref{section-4}. Observe from Proposition \ref{thm4-1} that the set $\{ \omega > 0 
\colon m_{\omega}^{\mathcal{N}} < m_{\infty} 
\}$ is connected. This fact together with Proposition \ref{thm-0-1} motivates us to introduce the following value: 
\begin{equation}\label{threshold}
\omega_{c}^{\mathcal{N}} := 
\sup\bigm\{ \omega > 0 
\colon m_{\omega}^{\mathcal{N}} < m_{\infty} 
\bigm\}
. 
\end{equation}
\begin{remark}\label{remark-0}
\noindent
{\rm (i)}~It follows from a result in \cite{Zhang-Zou} that $m_{\omega}^{\mathcal{N}} < m_{\infty}$ for any $3 < p < 5$ and any $\omega > 0$. Hence, when $3<p<5$, $\omega_{c}^{\mathcal{N}} = \infty$ (see also \cite{AIIKN, CG}). 
\\
\noindent
{\rm (ii)}~The authors of the present paper showed in \cite{AIKN3} that for any $1 <p \le 3 $, there exists $\omega_{0}>0$ such that if $\omega \in (0, \omega_{0})$, then $m_{\omega}^{\mathcal{N}} < m_{\infty}$. Hence, when $1<p\le 3$, $\omega_{c}^{\mathcal{N}}>0$. 
\end{remark}

Now, we state main results of this paper (recall that by a ground state, we 
mean a minimizer for $m_{\omega}^{\mathcal{N}}$): 
\begin{theorem} \label{thm-0}
Assume $d=3$ and $1 < p \le 3$. 
Then, the following holds: 
\begin{enumerate}
\item[\rm (i)]
$0 < \omega_{c}^{\mathcal{N}} < \infty$. 
\item[\rm (ii)] 
If $0< \omega < \omega_{c}^{\mathcal{N}}$, 
then a ground state exists; and 
if $\omega > \omega_{c}^{\mathcal{N}}$, 
then there is no ground state.
\end{enumerate}
\end{theorem}

\begin{theorem} \label{ex-critical-o}
Assume $d=3$ and $1 < p < 3$. 
If $\omega = \omega_{c}^{\mathcal{N}}$, then 
a ground state exists. 
\end{theorem}

\begin{remark}\label{rem-3}
When $p = 3$ and 
$\omega = \omega_{c}^{\mathcal{N}}$, 
we have not obtained any information 
about the existence of ground state. 
\end{remark}

Here, we would like to emphasize that 
the significance of the study for 
the existence and non-existence of 
ground state in $\mathbb{R}^{3}$. 
When $d\ge 4$, it is known that 
for any $\omega>0$ and 
any $1<p<\frac{d+2}{d-2}$, a ground state exists 
(see Proposition 1.1 of \cite{AIIKN}). 
On the other hand, 
Theorem \ref{thm-0} tells us 
that when $d=3$, 
we cannot always find a solution to \eqref{sp} 
(or a standing wave of \eqref{nls}) 
through the minimization problem for $m_{\omega}^{\mathcal{N}}$. 

Note here that, as mentioned before, 
we have
$m_{\omega}^{\mathcal{N}} \le m_{\omega}^{S}$ for any $\omega > 0$, 
in particular, 
$m_{\omega}^{\mathcal{N}} 
= m_{\omega}^{S}$ 
if a ground state exists. 
It would be an interesting problem to reveal 
whether or not we always have $m_{\omega}^{\mathcal{N}} = m_{\omega}^{S}$. Concerning this, we can obtain the 
following result:
\begin{theorem}\label{thm-2mini-1}
Assume $d=3$. If $1<p<3$, then $m_{\omega}^{\mathcal{N}}< m_{\omega}^{S}$ for any 
$\omega > \omega_{c}^{\mathcal{N}}$; 
and if $p=3$, then there exists $\omega_{*}\ge \omega_{c}^{\mathcal{N}}$ such that $m_{\omega}^{\mathcal{N}}<m_{\omega}^{S}$ for any $\omega>\omega_{*}$.
\end{theorem}
\begin{remark}\label{remark-11}
By Proposition \ref{thm-0-1} and 
the definition of $\omega_{c}^{\mathcal{N}}$, 
we see that if $d\ge 3$, $1<p<\frac{d+2}{d-2}$ 
and $\omega<\omega_{c}^{\mathcal{N}}$,
then we have $m_{\omega}^{\mathcal{N}} = 
m_{\omega}^{\mathcal{S}}$.
\end{remark}

We will prove Theorem \ref{thm-2mini-1} 
in Section \ref{section-gap}.

We make a comment on a result in \cite{AIKN}. The same authors of the present paper considered the following 
minimization problem in \cite{AIKN}: 
\begin{equation}\label{mini-k}
m_{\omega}^{\mathcal{K}} := 
\inf\left\{\mathcal{S}_{\omega}(u) \colon 
u \in H^{1}(\mathbb{R}^{d}) \setminus \{0\}, 
\; \mathcal{K} (u) = 0 \right\}, 
\end{equation}
where 
\begin{equation}\label{func-k}
\mathcal{K}(u) := \|\nabla u\|_{L^{2}}^{2} 
- \frac{d(p-1)}{2(p+1)} \|u\|_{L^{p+1}}^{p+1} 
- \|u\|_{L^{\frac{2d}{d-2}}}^{\frac{2d}{d-2}}. 
\end{equation}
In Theorem 1.2 of \cite{AIKN}, the 
authors claimed that if $d=3$, $1+
\frac{4}{3} < p < 5$ and $\omega$ is 
sufficiently large, 
then there is no minimizer for 
$m_{\omega}^{\mathcal{K}}$. 
However, the claim is false; indeed, we 
can show that for any $3 < p < 5$ and 
any $\omega>0$, there exists a 
minimizer for $m_{\omega}
^{\mathcal{K}}$ in a way similar to 
Proposition 1.1 of \cite{AIIKN}. In the 
present paper, we correct the mistake. 
To this end, we introduce the following 
value $\omega_{c}^{\mathcal{K}}$, 
as well as $\omega_{c}^{\mathcal{N}}
$: 
\begin{equation}\label{threshold-2}
\omega_{c}^{\mathcal{K}} := 
\sup\bigm\{ \omega > 0 
\colon m_{\omega}^{\mathcal{K}} < 
m_{\infty}
\bigm\}. 
\end{equation}
Then, we can obtain 
Theorem \ref{thm-1} below, in which 
the third claim {\rm (iii)} gives us the 
correct 
range of the exponent $p$ 
concerning the existence of minimizer 
for $m_{\omega}^{\mathcal{K}}$:
\begin{theorem} \label{thm-1}
Assume $d=3$ and $\frac{7}{3} < p < 5$. 
Then, all of the following holds: 
\begin{enumerate}
\item[\rm (i)]
$\omega_{c}^{\mathcal{N}} = 
\omega_{c}^{\mathcal{K}}$.
\item[\rm (ii)]
$m_{\omega}^{\mathcal{N}} = 
m_{\omega}^{\mathcal{K}}$ for any 
$\omega >0$. 
\item[\rm (iii)]
Assume $1+\frac{4}{3}<p\le 3$ 
($p=3$ is excluded). 
If $0< \omega < \omega_{c}
^{\mathcal{N}} (=\omega_{c}
^{\mathcal{K}})$, then a minimizer for 
$m_{\omega}^{\mathcal{K}}$ exists; 
and 
if $\omega > \omega_{c}
^{\mathcal{N}} (=\omega_{c}
^{\mathcal{K}})$, then there is no 
minimizer for $m_{\omega}
^{\mathcal{K}}$. 
\item[\rm (iv)]
Assume $1+\frac{4}{3}<p < 3$. 
If $\omega = \omega_{c}
^{\mathcal{N}} (=\omega_{c}
^{\mathcal{K}})$, then a minimizer for 
$m_{\omega}^{\mathcal{K}}$ exists. 
\end{enumerate}
\end{theorem} 
\begin{remark}
We can verify that any ground state (minimizer for $m_{\omega}^{\mathcal{N}}$) becomes a minimizer for $m_{\omega}^{\mathcal{K}}$, and vice versa. 
\end{remark}
We give a proof of Theorem \ref{thm-1} in Section \ref{sec-7}. 

In what follows, we 
will always fix $d = 3$ 
unless otherwise noted. 
We introduce the notation used in this paper: 
\\ 
\noindent 
{\bf Notation and auxiliary results}.
\\
\noindent 
{\rm (i)}~We define the operators $L_{+}$ and $L_{-}$ by 
\begin{align}
\label{eq:1.18}
L_{+}&:=-\Delta + V_{+}
\quad 
\mbox{with}
\quad 
V_{+}:=-5W^{4}
,
\\[6pt]
\label{19/9/20/8:20}
L_{-}&:=-\Delta +V_{-}
\quad 
\mbox{with}
\quad 
V_{-}:=-W^{4}
.
\end{align}
The operators $L_{+}$ and $L_{-}$ 
are self-adjoint on $L^{2}(\mathbb{R}^{3})$ with the domain $H^{2}(\mathbb{R}^{3})$. For the information about the spectrum of $L_{+}$, see Lemma \ref{18/11/14/10:43} below. 

Furthermore, we define the operator $H$ by 
\begin{equation}\label{op-H}
H u := L_{+}\Re[u] + i L_{-}\Im[u]
.
\end{equation}
We refer $H$ as the linearized operator around the Talenti function $W$.

\noindent 
{\rm (ii)}~
\begin{equation}\label{eq-J}
\mathcal{J}(u):= 
\Big(\frac{1}{2} - \frac{1}{p+1}\Big) \|u\|
_{L^{p+1}}^{p+1} 
+ 
\frac{1}{3}\|u\|_{L^{6}}^{6}.
\end{equation} 
\noindent 
{\rm (iii)}~For $d\ge 3$, we define the operator $\Lambda \colon \dot{H}^{1}(\mathbb{R}^{3})\to \dot{H}^{1}(\mathbb{R}^{3})$ by 
\begin{equation}\label{18/06/24/17:08}
\Lambda u := \frac{1}{2} u + 
x \cdot \nabla u
.
\end{equation}
Notice that 
$\Lambda W \not\in L^{2}(\mathbb{R}^{3})$, and $\Lambda W$ is a zero energy resonance of $L_{+}$. 
By direct computations, we can verify that 
\begin{equation}\label{19/10/9/10:22}
V_{+}\partial_{j}W = \Delta \partial_{j}W \quad \mbox{for $1\le j\le 3$}, 
\qquad 
V_{+} \Lambda W = \Delta \Lambda W,
\qquad 
V_{-}W =\Delta W.
\end{equation}

\noindent 
{\rm (iv)}~We use the symbol $\langle \cdot, \cdot \rangle$ 
to denote the scalar product 
in $L^{2}(\mathbb{R}^{3})$, namely, 
if $f, g \in L^{2}(\mathbb{R}^{3})$, then 
\begin{equation}\label{19/9/18/1:1}
\langle f, \, g \rangle = 
\Re\int_{\mathbb{R}^{3}} f(x) \overline{g(x)} \,dx.
\end{equation} 
We also use the same symbol to denote the pairing between $H^{1}(\mathbb{R}^{3})$ and $H^{-1}(\mathbb{R}^{3})$, namely if $f \in H^{1}(\mathbb{R}^{3})$ and $g \in H^{-1}(\mathbb{R}^{3})$, then 
\begin{equation}\label{19/9/18/1:2}
\langle f, \, g \rangle 
= 
\Re\int_{\mathbb{R}^{3}} \langle \nabla \rangle f(x) \overline{\langle \nabla \rangle^{-1} g(x)} \,dx.
\end{equation} 

\noindent 
{\rm (v)}~For $\nu>0$, we define the scaling operator $T_{\nu}$ to be that for any function $u$ on $\mathbb{R}^{3}$, 
\begin{equation}\label{18/06/29/12:35}
T_{\nu} u (x):=\nu^{-1}u(\nu^{-\frac{2}{3}}x). 
\end{equation}
Notice that $\|\nabla T_{\nu}u\|_{L^{2}}= \|\nabla u\|_{L^{2}}$ and 
$\|T_{\nu}u\|_{L^{6}}= \|\nabla u\|_{L^{6}}$.

\noindent 
{\rm (vi)}~For any $(y,\nu,\theta)\in 
\mathbb{R}^{3}\times (0,\infty) \times [-\pi, \pi)$, 
we define the transformation $T_{(y,\nu,\theta)}$ 
by 
\begin{equation}\label{19/9/27/12:07}
T_{(y,\nu,\theta)}[u](x):=\nu^{-1} e^{i\theta} u 
(\nu^{-\frac{2}{3}}x+y)
.
\end{equation}

\noindent 
{\rm (vii)}~
We put 
\[
\mathbf{b}_{0} = (0, 1, 0) \in \mathbb{R}^3 
\times (0, \infty) 
\times [-\pi, \pi). 
\]
Moreover, for $\delta>0$ and $r>0$, we define 
\begin{align*}
B_{\dot{H}^{1}}(W,\delta)
&:=
\bigm\{u \in \dot{H}^{1}(\mathbb{R}^{3}) \colon 
\|u-W\|_{\dot{H}^{1}} 
\le \delta
\bigm\}, 
\\[6pt]
B(\mathbf{b}_{0},r)&:=
\bigm\{ (y,\nu,\theta) \in \mathbb{R}^{3} \times 
\mathbb{R} \times [-\pi, \pi)
\colon |y|+|\nu-1|+|\theta| < r 
\bigm\}
.
\end{align*}


\section{Unrealizable sequence}\label{sec-2}

We will employ the contradiction argument in order to prove Theorems \ref{thm-0} though \ref{thm-2mini-1}. 
The following theorem plays a crucial role 
to prove these theorems: 
\begin{theorem} 
\label{non-seq}
Assume $1 < p \le 3$. 
Then, there is no sequence 
$\{(Q_{n}, \omega_{n})\}$ in 
$H^{1}(\mathbb{R}^{3}) \times (0,\infty)$ 
satisfying all of the following conditions:
\\
\noindent 
{\rm (i)}~$Q_{n}$ is a solution to \eqref{sp} with $\omega=\omega_{n}$.
\\
\noindent 
{\rm (ii)}~$\liminf\limits_{n\to \infty}
\omega_{n}>0$ if $1<p<3$, 
and $\limsup\limits_{n\to \infty}
\omega_{n}=\infty$ if $p=3$.
\\
\noindent 
{\rm (iii)}~$\lim\limits_{n\to \infty}
\|Q_{n}\|_{L^{\infty}}=\infty$, and 
$\lim\limits_{n\to \infty} \omega_{n}
\|Q_{n}\|_{L^{\infty}}^{-4}=0$. 
\\
\noindent 
{\rm (iv)}~
Define $\widetilde{Q}_{n}$ by
$\widetilde{Q}_{n}(x):=\|Q_{n}\|_{L^{\infty}}^{-1}
Q_{n}(\|Q_{n}\|_{L^{\infty}}^{-2}x)$. 
Then, $\lim\limits_{n\to \infty}
\widetilde{Q}_{n}=W$ strongly in $\dot{H}^{1}(\mathbb{R}^{3})$. 
\end{theorem}
We will give the proof of Theorem \ref{non-seq} in Section \ref{19/9/9/14:50} below.

\subsection{Preliminaries}\label{sec2-1}

In order to prove Theorem 
\ref{non-seq}, 
we need some preparations.

It is well known (see, e.g., Theorem 
6.23 of \cite{Lieb-Loss}) that for any 
$\alpha >0$ and any function $u$ on $
\mathbb{R}^{3}$, 
\begin{equation} \label{explicit-1}
(-\Delta + \alpha)^{-1} u(x) 
= \int_{\mathbb{R}^{3}} \frac{e^{- 
\sqrt{\alpha} |x-y|}}{4\pi |x - y|} u(y) 
dy.
\end{equation}
Moreover, 
\begin{equation}
\label{delta0-inv}
(-\Delta)^{-1}f(x)
:=
\int_{\mathbb{R}^{3}}\frac{\Gamma(\frac{3}{2})}{1 2 \pi^{\frac{3}{2}}}
|x-y|^{-1} f(y)\,dy 
,
\end{equation}
where $\Gamma$ denotes the Gamma function, namely 
\begin{equation*}
\Gamma(\nu)=\int_{0}^{\infty}e^{-t}t^{\nu-1}\,dt 
\quad \mbox{for $\nu>0$}. 
\end{equation*}
For our convenience, we record the Hardy-
Littlewood-Sobolev inequality: 
For $0<s<3$, and 
$1<q<r<\infty$ with $s-\frac{3}{q}=-\frac{3}{r}$, 
\begin{equation}\label{HLS-ineq}
\big\| 
\int_{\mathbb{R}^{3}}\frac{f(y)}{|x-y|^{3-s}}\,dy 
\big\|_{L^{r}} 
\lesssim 
\|f\|_{L^{q}}.
\end{equation}
See, e.g. Theorem 2.6 of \cite{Linares-Ponce}.

The following lemma 
holds from 
the Young and 
the weak Young inequalities 
(see, e.g., (9) in Section 4.3 of \cite{Lieb-Loss})
\begin{lemma}\label{Delta-ineq}
Let $\alpha >0$, and let 
$1\le s \le q \le \infty$ and 
$3(\frac{1}{s}-\frac{1}{q})<2$. Then, we have 
\begin{equation}\label{res-est1}
\|(-\Delta+\alpha)^{-1} \|_{L^{s}(\mathbb{R}^{3}) \to L^{q}(\mathbb{R}^{3})}
\lesssim 
\alpha^{\frac{3}{2}(\frac{1}{s}-\frac{1}{q})-1}
,
\end{equation}
where the implicit constant depends only on $d$, $s$ and $q$. Furthermore, if 
$1< s \le q < \infty$ and $3(\frac{1}{s}-\frac{1}{q})<2$, then, 
\begin{equation}\label{res-est2}
\|(-\Delta+\alpha)^{-1} \|_{L_{\rm weak}^{s}(\mathbb{R}^{3}) \to L^{q}(\mathbb{R}^{3})}
\lesssim 
\alpha^{\frac{3}{2}(\frac{1}{s}-\frac{1}{q})-1}
,
\end{equation}
where the implicit constant depends only on $d$, $s$ and $q$.
\end{lemma}

Lemma \ref{Delta-ineq} does not deal with the case where $3(\frac{1}{s}-\frac{1}{q})=2$. In that case, we use the following:
\begin{lemma}\label{18/11/23/17:17}
Let $\alpha >0$, and 
let $3 < q < \infty$. 
Then, we have 
\begin{equation}\label{res-est3}
\|(-\Delta+\alpha)^{-1} \|
_{L^{\frac{3q}{3+2q}}(\mathbb{R}
^{3}) \to L^{q}(\mathbb{R}^{3})}
\lesssim 
1,
\end{equation}
where the implicit constant depends only on $s$ and $q$. 
\end{lemma}
\begin{proof}[Proof of Lemma \ref{18/11/23/17:17}]

It follows from the first resolvent 
equation that for any $\alpha>0$ and $
\alpha_{0}\ge 0$, 
\begin{equation}\label{first-res}
\begin{split}
(-\Delta+\alpha)^{-1}-(-\Delta+\alpha_{0})^{-1}
= -(\alpha-\alpha_{0}) (-\Delta+\alpha)^{-1} (-
\Delta+\alpha_{0})^{-1}.
\end{split}
\end{equation}
We see from \eqref{first-res} 
with $\alpha_{0} = 0$, 
Lemma \ref{Delta-ineq} and the Hardy-Littlewood-Sobolev inequality 
\eqref{HLS-ineq} 
(or Sobolev's embedding) that 
\begin{equation}\label{18/11/23/17:18}
\begin{split}
\|(-\Delta+\alpha)^{-1}f \|_{L^{q}}
&\le 
\|\big\{(-\Delta+\alpha)^{-1} -(-\Delta)^{-1} \big\} f \|_{L^{q}}
+
\|(-\Delta)^{-1} f \|_{L^{q}}
\\[6pt]
&\le
\alpha \| (-\Delta+\alpha)^{-1} (-\Delta)^{-1} f \|_{L^{q}}
+
\|(-\Delta)^{-1} f \|_{L^{q}}
\\[6pt]
&\lesssim 
\|(-\Delta)^{-1} f \|_{L^{q}}
\lesssim 
\|f\|_{L^{\frac{3q}{3+2q}}}.
\end{split}
\end{equation}
Note that if $3<q$, 
then $\frac{3q }{3+2q}>1$. Hence, the desired estimate \eqref{res-est1} holds. 
\end{proof}

In addition to the estimates above, we will use the following result given in Lemma 2.7 of \cite{CG}: 
\begin{lemma}\label{thm-6}
Let $\alpha>0$. Then, we have 
\begin{equation}\label{19/9/12/16:54}
\alpha^{\frac{1}{2}}
\bigm\langle W, \, (-\Delta + \alpha)^{-1} V_{+} \Lambda W \bigm\rangle
= 
6\pi 
+ 
O(\alpha^{\frac{1}{2}})
.
\end{equation}
Moreover, for any $2< p < 5$ and any $0< \delta_{1}<\min\{1, p-2\}$, we have 
\begin{equation}\label{19/9/12/16:55}
\bigm\langle W^{p}, \, (-\Delta + \alpha)^{-1} V_{+}\Lambda W \bigm\rangle 
= 
-\frac{5-p}{10(p+1)}
\|W\|_{L^{p+1}}^{p+1} 
+ 
O(\alpha^{\frac{\delta_{1}}{2}})
.
\end{equation}
\end{lemma}
The restriction about $p$ in 
Lemma \ref{thm-6} comes from the fact 
that $W \not\in L^{p+1}(\mathbb{R}^{3})$ 
for $1 \leq p \leq 2$. 
When $1<p\le 2$, we use the following estimate, 
instead of \eqref{19/9/12/16:55} in Lemma \ref{thm-6} : 
\begin{lemma} \label{thm-6-1}
Let $\alpha>0$. 
Then, for any $1 < p \le 2$ and 
any $0< \delta <\frac{p}{2}$, we have 
\begin{equation}
\big| \bigm\langle W^{p},\, 
(-\Delta +\alpha)^{-1} V_{+} \Lambda W \bigm\rangle 
\big|
\lesssim 
\alpha^{-\frac{2-p}{2}- \delta}, 
\end{equation}
where the implicit constant depends only on $p$ and $\delta$. 
\end{lemma}
\begin{proof}[Proof of Lemma \ref{thm-6-1}] 
Let $1<p\le 2$ and $0< \delta <\frac{p}{2}$. Furthermore, put $\delta^{\prime}:=\frac{6\delta}{p-2\delta}$ and $q':=\frac{3 + \delta^{\prime}}{p}$. Notice that $\frac{3}{2}< q' <\infty$. Hence, it follows from \eqref{res-est1} in Lemma \ref{Delta-ineq} with $s=q'$ and $q = \infty$, and $W \in L^{3+\delta^{\prime}}(\mathbb{R}^{3})$ that 
\begin{equation}
\begin{split}
&\big| 
\langle (-\Delta + \alpha)^{-1} V_{+} \Lambda W, 
\ W^{p} \rangle 
\big| 
= \big| 
\langle V_{+} \Lambda W, \ (-\Delta + \alpha )^{-1} W^{p} \rangle 
\big| 
\\[6pt]
& \leq 
\|V_{+} \Lambda W\|_{L^{1}} 
\|(-\Delta+ \alpha)^{-1} W^{p} \|_{L^{\infty}} 
\lesssim 
\alpha^{\frac{1}{2}( \frac{3}{q'}-2)}
\|W\|_{L^{3+\delta^{\prime}}}^{p}
\lesssim 
\alpha^{-\frac{2-p}{2}-\delta},
\end{split}
\end{equation}
where the implicit constants depend only on $p$ and $\delta$. 
Hence, we have proved the lemma. 
\end{proof}
For $\alpha>0$, we define the operators $G(\alpha)$ and $G(\alpha)^{*}$ by 
\begin{align}
\label{def-G}
G(\alpha)f
&:=
\big\{ 1+(-\Delta+\alpha)^{-1}V_{+}\big\} \Re[f]
+i
\big\{
1+ (-\Delta+\alpha)^{-1}V_{-}\big\} \Im[f]
,
\\[6pt]
\label{def-G*}
G(\alpha)^{*}f
&:=
\big\{ 1+V_{+}(-\Delta+\alpha)^{-1} \big\} \Re[f]
-i
\big\{
1+ V_{-}(-\Delta+\alpha)^{-1}\big\} \Im[f]
.
\end{align}
Notice that 
\begin{equation}\label{eq-HG}
H+\alpha =
(-\Delta + \alpha) G(\alpha). 
\end{equation}

\begin{theorem}\label{thm-ufe} 
Assume $6 < r <\infty$. 
Then, there exists $\alpha_{*}>0$ 
depending only on $r$ 
such that for any $0<\alpha < \alpha_{*}$
the inverse of $G(\alpha)$ exists as 
bounded operator from 
$L^{r}(\mathbb{R}^{3})$ to itself, and for any $f \in L^{r}(\mathbb{R}^{3})$,
\begin{equation}\label{res-est-G}
\|G(\alpha)^{-1}f\|_{L^{r}}
\lesssim 
\alpha^{-\frac{1}{2}} 
\| f \|_{L^{r}},
\end{equation} 
where the implicit constant depends only on $d$ and $r$. Furthermore, if $f \in L^{r}(\mathbb{R}^{3})$ satisfies 
\begin{equation}\label{res-est-G1}
\langle f,\, V_{+} \nabla W \rangle 
=0,
\quad 
\langle f,\, V_{+} \Lambda W \rangle
=
0, 
\quad 
\langle f,\, i V_{-} \Lambda W \rangle =0,
\end{equation}
then 
\begin{equation}\label{res-est-G2}
\|G(\alpha)^{-1}f \|_{L^{r}}
\lesssim 
\| f \|_{L^{r}},
\end{equation} 
where the implicit constant depends only on $r$.
\end{theorem}
We give a proof of Theorem \ref{thm-ufe} in Section \ref{sec7-reso}. 

In order to ensure the orthogonalities in \eqref{res-est-G1}, we need the following geometric decomposition:
\begin{theorem}\label{thm-imp}
There exist 
$0< \delta_{\rm geo}<1$ and $0<r_{\rm geo}<\frac{1}{2}$ such that for any $0< \alpha <1$, there exists a unique continuous mapping $\mathbf{b}_{\alpha}=(y_{\alpha}, \nu_{\alpha}, \theta_{\alpha}) \colon B_{\dot{H}^{1}}(W,\delta_{\rm geo}) \to 
B(\mathbf{b}_{0}, r_{\rm geo})$ with the following properties: 
\begin{equation}\label{eq-origin}
\mathbf{b}_{\alpha}(W)
=
(y_{\alpha}(W), \nu_{\alpha}(W), \theta_{\alpha}(W))=(0,1,0) = \mathbf{b}_{0}, 
\end{equation}
and if $u \in B_{\dot{H}^{1}}(W,\delta_{\rm geo})$, then 
\begin{align}
\label{ortho-1}
&\bigm\langle 
T_{\mathbf{b}_{\alpha}(u)}[u]-W,
\ 
G(\alpha)^{*} V_{+}\partial_{j}W 
\bigm\rangle
=0 \quad 
\mbox{for $1\le j\le 3$}, 
\\[6pt]
\label{ortho-2}
&\bigm\langle 
T_{\mathbf{b}_{\alpha}(u)}[u]-W,
\ G(\alpha)^{*}V_{+}\Lambda W 
\bigm\rangle
=0,
\\[6pt]
\label{ortho-3}
&\bigm\langle 
T_{\mathbf{b}_{\alpha}(u)}[u]-W,
\ i G(\alpha)^{*}V_{-} W 
\bigm\rangle
=0
.
\end{align}
\end{theorem}
\begin{remark}\label{19/9/27/9:25}
Since the conditions \eqref{ortho-1} through \eqref{ortho-3} contain the parameter $\alpha$, Proposition \ref{thm-imp} does not follow directly from the implicit function theorem; 
The point is that $\delta_{\rm geo}$ is independent of $\alpha$. 
\end{remark}

We will prove Theorem 
\ref{thm-imp} in Section \ref{19/10/9/11:28}.


\subsection{Proof of Theorem \ref{non-seq}}\label{19/9/9/14:50}

We give a proof of Theorem \ref{non-seq}:
\begin{proof}[Proof of Theorem \ref{non-seq}]
Suppose for contradiction 
that there exists a sequence 
$\{(Q_{n}, \omega_{n}) \}$ in 
$H^{1}(\mathbb{R}^{3}) \times (0, \infty)$ 
satisfying the conditions {\rm (i)} through {\rm (iv)} 
in Theorem \ref{non-seq}: 
hence $Q_{n}$ is a solution to \eqref{sp} with $\omega=\omega_{n}$; and 
\begin{align}
\label{eq2-1-1}
\liminf_{n\to \infty} \omega_{n}&>0 \quad \mbox{if $1<p<3$,}
\quad 
\mbox{and}
\quad 
\limsup_{n\to \infty} \omega_{n}=\infty \quad \mbox{if $p=3$}, 
\\[6pt]
\label{eq2-1-2}
\lim_{n\to \infty}M_{n}&=\infty,
\qquad 
\lim_{n\to \infty}\alpha_{n}=0,
\\[6pt]
\label{eq2-1-3}
\lim_{n\to \infty}\widetilde{Q}_{n}&=W
\quad 
\mbox{strongly in $\dot{H}^{1}(\mathbb{R}^{3})$}
,
\end{align}
where 
\begin{equation}\label{eq2-1-4}
M_{n}:=\|Q_{n}\|_{L^{\infty}}, 
\quad 
\alpha_{n}:=\omega_{n}M_{n}^{-4}, 
\quad 
\widetilde{Q}_{n}:=T_{M_{n}} Q_{n}=M_{n}^{-1}Q_{n}(M_{n}^{-2}\cdot).
\end{equation} 
We also introduce $\beta_{n}:=M_{n}^{p-5}$. 
Note that $\widetilde{Q}_{n}$ is the normalization of $Q_{n}$ in $L^{\infty}(\mathbb{R}^{3})$, namely 
\begin{equation}\label{eq2-1-5}
\|\widetilde{Q}_{n}\|_{L^{\infty}}=1
.
\end{equation}

We shall derive a contradiction. 
Let $\delta_{\rm geo}$ and $r_{\rm geo}$ 
be the constants given in Theorem \ref{thm-ufe}.
By \eqref{eq2-1-2} and \eqref{eq2-1-3}, passing to a subsequence, we may assume that $0< \alpha_{n}<1$ and $\widetilde{Q}_{n}\in B_{\dot{H}^{1}}(W,\delta_{\rm geo})$ for all $n$. Hence, for any $n$, we can apply Theorem 
\ref{thm-imp} to $\widetilde{Q}_{n}$, and obtain the sequence $\{ (y_{n}, \nu_{n},\theta_{n})\}$ in $B(\mathbf{b}_{0}, r_{\rm geo})$ such that: 
\begin{align}
\label{conv-0}
&\lim_{n\to \infty}|y_{n}|=0, 
\quad 
\lim_{n\to \infty}\nu_{n}=1,
\quad 
\lim_{n\to \infty}\theta_{n}=0,
\\[6pt]
\label{orthogo-1}
&\langle 
T_{(y_{n}, \nu_{n},\theta_{n})}[\widetilde{Q}_{n}]
-W, \ 
G(\alpha_{n})^{*} V_{+} \partial_{j} W 
\rangle = 0, \quad 
1\le j\le 3, 
\\[6pt]
\label{orthogo-2}
&\langle 
T_{(y_{n},\nu_{n},\theta_{n})}[\widetilde{Q}_{n}] - W, \ 
G(\alpha_{n})^{*}V_{+} \Lambda W 
\rangle
=0, 
\\[6pt]
\label{orthogo-3}
&
\langle 
T_{(y_{n}, \nu_{n}, \theta_{n})}[\widetilde{Q}_{n}] - W, 
\ 
G(\alpha_{n})^{*}V_{-} (iW) 
\rangle
=0,
\\[6pt]
\label{conv-Ta}
&
\lim_{n\to \infty}
\|
T_{\nu_{n}}[\widetilde{Q}_{n}(\cdot+y_{n})]-W 
\|_{\dot{H}^{1}} 
= 0.
\end{align}
Then, we define 
\begin{align}\label{eq2-1-6}
u_{[n]}&:= T_{(y_{n}, \nu_{n},\theta_{n})}[\widetilde{Q}_{n}],
\\[6pt]
\label{eq2-1-7}
\eta_{[n]}&:=u_{[n]} - W 
= T_{(y_{n}, \nu_{n},\theta_{n})}[\widetilde{Q}_{n}] - W. 
\end{align}
Observe that $u_{[n]}$ satisfies 
\begin{equation}\label{eq2-1-8}
- \Delta u_{[n]}
+ 
\alpha_{[n]} u_{[n]}
- \beta_{[n]} 
| u_{[n]}|^{p-1} u_{[n]}
-|u_{[n]}|^{4} u_{[n]} = 0, 
\end{equation}
\begin{align}
& 
\alpha_{[n]} := \alpha_{n} \nu_{n}^{-4},
\\[6pt]
&
\beta_{[n]} := \beta_{n} \nu_{n}^{p-5}.
\end{align} 
Note that 
\begin{equation} \label{re-a-b-1}
\beta_{[n]} = \omega_{n}^{- \frac{5-p}{4}} 
\alpha_{[n]}^{\frac{5-p}{4}}
\end{equation}
Furthermore, plugging \eqref{eq2-1-7} into \eqref{eq2-1-8}, one can see that $\eta_{[n]}$ satisfies 
\begin{equation}\label{non-term}
\begin{split}
(H+ \alpha_{[n]}) \eta_{[n]} 
= 
F_{[n]}
:=
-\alpha_{[n]} W 
+ \beta_{[n]} W^{p} 
+ N_{[n]}, 
\end{split}
\end{equation}
where $H$ is the linearized operator around 
$W$ (see \eqref{op-H}), 
and 
\begin{equation}\label{term-N}
\begin{split}
N_{[n]}
&:= 
|W + \eta_{[n]}|^{4} (W + \eta_{[n]})
- W^{5} - 5 W^{4} \Re[\eta_{[n]}] 
-i W^{4}\Im[\eta_{n,\nu}]
\\[6pt]
&\qquad + 
\beta_{[n]} \bigm\{
|W + \eta_{[n]}|^{p-1}(W + \eta_{[n]}) 
- W^{p} 
\bigm\}
. 
\end{split} 
\end{equation}
Using the factorization \eqref{eq-HG}, 
we can rewrite the equation \eqref{non-term} as 
\begin{equation}
\label{eq-eta}
G(\alpha_{[n]}) \eta_{[n]} 
= 
(-\Delta + \alpha_{[n]})^{-1} 
F_{[n]}.
\end{equation}
By the orthogonalities \eqref{orthogo-1} 
through \eqref{orthogo-3}, and \eqref{eq-eta}, 
one can see that 
\begin{align}
\label{orthogo-11}
\langle (-\Delta + \alpha_{[n]})^{-1} F_{[n]}, \ 
V_{+}\partial_{j} W \rangle
= \langle \eta_{[n]}, \ 
G(\alpha_{[n]})^{*}V_{+}\partial_{j} W \rangle
&=0 
\quad 
\mbox{for $1\le j \le 3$},
\\[6pt]
\label{orthogo-21}
\langle (-\Delta + \alpha_{[n]})^{-1} F_{[n]}, \ 
V_{+}\Lambda W \rangle
= \langle \eta_{[n]}, G(\alpha_{[n]})^{*}
V_{+}\Lambda W \rangle 
&=0,
\\[6pt]
\label{orthogo-31} 
\langle (-\Delta + \alpha_{[n]})^{-1} F_{[n]}, \ 
i V_{-}W \rangle
= \langle \eta_{[n]}, 
G(\alpha_{n})^{*} V_{-} (iW) \rangle 
&=0
.
\end{align}
Here, we claim the following: 
\begin{lemma}\label{error-estimate-gap-1}
For any $1 < p \le 3$, any $0< \varepsilon <\frac{p-1}{20}$, 
and any sufficiently large $n$, the following estimate holds:
\begin{equation}\label{eq2-1-10}
\|\eta_{[n]}\|_{L^{\frac{1}{\varepsilon}}} 
\lesssim 
\alpha_{[n]}^{\frac{1}{2} - 2 \varepsilon} 
,
\end{equation}
where the implicit constant depends only on $p$, $\varepsilon$ and the quantity $\liminf_{n\to \infty}\omega_{n}$ 
(if finite). 
\end{lemma}
\begin{proof}[Proof of Lemma \ref{error-estimate-gap-1}]
We shall divide the proof into 2 steps. 

\textbf{Step 1.}~ Let $0< \varepsilon <\frac{p-1}{20}$. 
Then, by \eqref{eq-eta}, \eqref{non-term}, 
Theorem \ref{thm-ufe}, and \eqref{orthogo-11} through \eqref{orthogo-31}, one can see that 
\begin{equation} \label{eq2-1-11}
\begin{split}
\|\eta_{[n]}\|_{L^{\frac{1}{\varepsilon}}} 
\lesssim 
\|F_{[n]}\|_{L^{\frac{1}{\varepsilon}}}
&\leq 
\alpha_{[n]} 
\|(-\Delta + \alpha_{[n]})^{-1} W\|_{L^{\frac{1}{\varepsilon}}} 
\hspace{-3pt} + 
\beta_{[n]} 
\|(-\Delta + \alpha_{[n]})^{-1} W^{p}\|_{L^{\frac{1}{\varepsilon}}} 
\\[6pt]
&\quad + 
\|(-\Delta + \alpha_{[n]})^{-1} N_{[n]}\|_{L^{\frac{1}{\varepsilon}}} 
.
\end{split} 
\end{equation}
We consider the first two terms on the right-hand side of \eqref{eq2-1-11}. 
It follows from 
Lemma \ref{Delta-ineq} 
and \eqref{re-a-b-1} that 
\begin{equation}\label{eq2-1-12}
\begin{split}
&
\alpha_{[n]} 
\|(-\Delta + \alpha_{[n]})^{-1} W\|_{L^{\frac{1}{\varepsilon}}} 
+ 
\beta_{[n]} 
\|(-\Delta + \alpha_{[n]})^{-1} W^{p}\|_{L^{\frac{1}{\varepsilon}}}
\\[6pt]
&\lesssim 
\alpha_{[n]}\alpha_{[n]}^{-\frac{1}{2}-2 \varepsilon}
\|W\|_{L^{\frac{3}{1- \varepsilon}}}
+ 
(\omega_{n}^{-\frac{5-p}{4}} \alpha_{[n]}^{\frac{5-p}{4}})
\alpha_{[n]}
^{\frac{p-2}{2}- \frac{p+3}{2} \varepsilon} 
\|W^{p}\|_{L^{\frac{3}{p(1-\varepsilon)}}}
\\[6pt]
& \lesssim 
\alpha_{[n]}^{\frac{1}{2} - 2 \varepsilon} 
+ 
\omega_{n}^{-\frac{5-p}{4}}
\alpha_{[n]}
^{\frac{p +1}{4} - \frac{p+3}{2}\varepsilon} 
\lesssim
\alpha_{[n]}^{\frac{1}{2}- 2 \varepsilon} 
,
\end{split}
\end{equation}
where the implicit constants 
depend only on $p$, $\varepsilon$ and 
the quantity $\liminf_{n\to \infty}\omega_{n}$ 
(if finite). 

Next, we consider the last term on the right-hand side of \eqref{eq2-1-11}. Our aim here is to show that 
\begin{equation}\label{estN-e}
\|(-\Delta + \alpha_{[n]})^{-1} N_{[n]}\|_{L^{\frac{1}{\varepsilon}}} 
\le
o_{n}(1) \|\eta_{[n]}\|_{L^{\frac{1}{\varepsilon}}}
.
\end{equation}
Notice that by \eqref{estN-e}, the last term on the right-hand side of \eqref{eq2-1-11} can be absorbed into the left-hand side. 
Hence, the claim \eqref{eq2-1-10} follows 
from \eqref{eq2-1-11}, \eqref{eq2-1-12} and \eqref{estN-e}.

\textbf{Step 2.}~
It remains to prove \eqref{estN-e}. 
It follows from 
\eqref{term-N} and 
the triangle inequality 
that 
\begin{equation} \label{est-Ne0}
\begin{split}
&
\|(-\Delta + \alpha_{[n]})^{-1} N_{[n]}\|_{L^{\frac{1}{\varepsilon}}} 
\\[6pt]
&\le 
\|(-\Delta + \alpha_{[n]})^{-1} 
\bigm\{ 
|W + \eta_{[n]}|^{4}(W + \eta_{[n]}) 
- W^{5} - 5 W^{4} \Re[\eta_{[n]}]
-iW^{4} \Im[\eta_{[n]}]
\bigm\}
\|_{L^{\frac{1}{\varepsilon}}} 
\\[6pt]
&\quad + 
\beta_{[n]} \|(-\Delta + \alpha_{[n]})^{-1} 
\big\{ |W + \eta_{[n]}|^{p-1}(W+\eta_{[n]}) 
- W^{p}
\big\} 
\|_{L^{\frac{1}{\varepsilon}}} 
.
\end{split}
\end{equation}
We consider the first term 
on the right-hand side of \eqref{est-Ne0}. 
By Lemma \ref{18/11/23/17:17}, 
elementary computations, 
H\"older's inequality and \eqref{conv-Ta}, 
one can see that 
\begin{equation}\label{est-I}
\begin{split}
&\|(-\Delta + \alpha_{[n]})^{-1} 
\bigm\{ 
|W + \eta_{[n]}|^{4}(W + \eta_{[n]}) - W^{5} - 5 W^{4} \Re[\eta_{[n]}]
-i W^{4} \Im[\eta_{[n]}]
\bigm\}
\|_{L^{\frac{1}{\varepsilon}}} 
\\[6pt]
&\lesssim 
\bigm\|
|W + \eta_{[n]}|^{4}(W + \eta_{[n]}) - W^{5} 
- 5 W^{4} \Re[\eta_{[n]}]
-i W^{4} \Im[\eta_{[n]}]
\bigm\|_{L^{\frac{3}{2 + 3\varepsilon}}}
\\[6pt]
&\lesssim 
\|W^{3} |\eta_{[n]}|^{2} \|_{L^{\frac{3}{2 + 3\varepsilon}}}
+
\||\eta_{[n]}|^{5} \|_{L^{\frac{3}{2 + 3\varepsilon}}}
\\[6pt]
&\le 
\|W^{3}\|_{L^{2}} 
\|\eta_{[n]}\|_{L^{6}} 
\|\eta_{[n]}\|_{L^{\frac{1}{\varepsilon}}}
+
\|\eta_{[n]}^{4}\|_{L^{\frac{3}{2}}}
\|\eta_{[n]}\|_{L^{\frac{1}{\varepsilon}}}
\le 
o_{n}(1) \|\eta_{[n]}\|_{L^{\frac{1}{\varepsilon}}}, 
\end{split} 
\end{equation}
where the implicit constants 
depend only on $\varepsilon$. 
Next, we consider the second term 
on the right-hand side of \eqref{est-Ne0}. 
It follows from the fundamental 
theorem of calculus and Lemma \ref{Delta-ineq} that 
\begin{equation}\label{19/9/13/14:56}
\begin{split} 
&
\beta_{[n]} \|(-\Delta + \alpha_{[n]})^{-1}
\big\{ |W + \eta_{[n]}|^{p-1}(W+\eta_{[n]}) 
- W^{p}
\big\} 
\|_{L^{\frac{1}{\varepsilon}}} 
\\[6pt]
&\le
\beta_{[n]} \|(-\Delta + \alpha_{[n]})^{-1} 
\Bigm\{ \frac{p+1}{2} \int_{0}^{1} |W + \theta \eta_{[n]}|^{p-1}\,d\theta~\eta_{[n]} 
\Bigm\} 
\|_{L^{\frac{1}{\varepsilon}}} 
\\[6pt]
&\quad +
\beta_{[n]} \|(-\Delta + \alpha_{[n]})^{-1} 
\Bigm\{ \frac{p-1}{2} 
\int_{0}^{1} |W + \theta \eta_{[n]}|^{p-3}(W + \theta \eta_{[n]})^{2}\,d\theta~\overline{\eta_{[n]}} 
\Bigm\} 
\|_{L^{\frac{1}{\varepsilon}}} 
\\[6pt]
&\lesssim 
\beta_{[n]} 
\|(-\Delta + \alpha_{[n]})^{-1}
\bigm\{
W^{p-1} \eta_{[n]}
\bigm\} 
\|_{L^{\frac{1}{\varepsilon}}}
+
\beta_{[n]} 
\|(-\Delta + \alpha_{[n]})^{-1} 
\bigm\{
W^{p-1} \overline{\eta_{[n]}}
\bigm\} 
\|_{L^{\frac{1}{\varepsilon}}}
\\[6pt]
&\quad + 
\beta_{[n]} \|(-\Delta + \alpha_{[n]})^{-1} 
\Big\{ \int_{0}^{1} \big\{~|W + \theta \eta_{[n]}|^{p-1} -W^{p-1}~\big\} \,d\theta~\eta_{[n]} 
\Bigm\} 
\|_{L^{\frac{1}{\varepsilon}}} 
\\[6pt]
&\quad + 
\beta_{[n]} \|(-\Delta + \alpha_{[n]})^{-1}
\Big\{ \int_{0}^{1} 
\big\{~|W + \theta \eta_{[n]}|^{p-3}(W + \theta \eta_{[n]})^{2} -W^{p-1}~\big\} \,d\theta~\overline{\eta_{[n]}} 
\Bigm\} 
\|_{L^{\frac{1}{\varepsilon}}} 
\\[6pt]
&\lesssim 
\beta_{[n]} 
\alpha_{[n]}^{\frac{p-5}{4} + 3\varepsilon} 
\|W^{p-1} \eta_{[n]}\|_{
L^{\frac{6}{p - 1 + 18 \varepsilon}}} 
\\[6pt]
&\quad + 
\beta_{[n]} 
\alpha_{[n]}^{\frac{p-5}{4}} 
\sup_{0\le \theta \le 1} 
\|\big\{ |W + \theta \eta_{[n]}|^{p-1} -W^{p-1} \big\} \eta_{[n]}\|_{L^{\frac{6}{p - 1 + 6 \varepsilon}}}
\\[6pt]
&\quad + 
\beta_{[n]} 
\alpha_{[n]}^{\frac{p-5}{4}} 
\sup_{0\le \theta \le 1} 
\|\big\{ |W + \theta \eta_{[n]}|^{p-3}(W + \theta \eta_{[n]})^{2} -W^{p-1} \big\} \overline{\eta_{[n]}}
\|_{L^{\frac{6}{p - 1 + 6 \varepsilon}}}
, 
\end{split}
\end{equation} 
where the implicit constants 
depend only on $p$ and $\varepsilon$. 
Consider the first term on the right-hand side of \eqref{19/9/13/14:56}. 
By H\"older's inequality, \eqref{re-a-b-1} 
and \eqref{eq2-1-1}, one can see that 
\begin{equation}\label{19/9/14/15:13}
\begin{split}
\beta_{[n]} 
\alpha_{[n]}^{\frac{p-5}{4} + 3\varepsilon} 
\|W^{p-1} \eta_{[n]}\|_{
L^{\frac{6}{p - 1 + 18 \varepsilon}}} 
\le 
\beta_{[n]} 
\alpha_{[n]}^{\frac{p-5}{4} + 3 \varepsilon} 
\|W^{p-1}\|_{L^{\frac{6}{p - 1 + 12 \varepsilon}}}
\|\eta_{[n]}\|_{L^{\frac{1}{\varepsilon}}}
\lesssim 
\alpha_{[n]}^{3 \varepsilon} 
\|\eta_{[n]}\|_{L^{\frac{1}{\varepsilon}}}
,
\end{split}
\end{equation}
where the implicit constant depends only on $p$, $\varepsilon$ and the quantity $\liminf\limits_{n\to \infty}\omega_{n}$ 
(if finite). 
Next, consider 
the second and third terms on the right-hand side of \eqref{19/9/13/14:56}. 
When $1<p\le 2$, 
it follows from \eqref{eq2-1-1}, \eqref{re-a-b-1}, 
H\"older's inequality and the convexity, 
that 
\begin{equation}\label{19/9/13/17:10}
\begin{split} 
&
\beta_{[n]} 
\alpha_{[n]}^{\frac{p-5}{4}} 
\sup_{0\le \theta \le 1} 
\|\big\{ |W + \theta \eta_{[n]}|^{p-1} -W^{p-1} \big\} \eta_{[n]}\|_{L^{\frac{6}{p - 1 + 6 \varepsilon}}}
\\[6pt]
&\quad + 
\beta_{[n]} 
\alpha_{[n]}^{\frac{p-5}{4}} 
\sup_{0\le \theta \le 1} 
\|\big\{ |W + \theta \eta_{[n]}|^{p-3}(W + \theta \eta_{[n]}) -W^{p-1} \big\} \overline{\eta_{[n]}}
\|_{L^{\frac{6}{p - 1 + 6 \varepsilon}}}
\\[6pt]
&\lesssim 
\sup_{0\le \theta \le 1}
\| |\theta \eta_{[n]}|^{p-1} \|_{L^{\frac{6}{p-1}}}
\|\eta_{[n]}\|_{L^{\frac{1}{\varepsilon}}}
\le 
\|\eta_{[n]}\|_{L^{6}}^{p-1} \|\eta_{[n]}\|_{L^{\frac{1}{\varepsilon}}}
=
o_{n}(1) \|\eta_{[n]}\|_{L^{\frac{1}{\varepsilon}}}
,
\end{split}
\end{equation}
where the implicit constant 
depends only on $p$, $\varepsilon$ 
and the quantity $\liminf_{n\to \infty}\omega_{n}$ 
(if finite), 
while when $2< p \le 3$, by 
\eqref{re-a-b-1}, \eqref{eq2-1-1}, H\"older's inequality and \eqref{conv-Ta}, 
\begin{equation}\label{19/9/13/17:11}
\begin{split}
&\beta_{[n]} 
\alpha_{[n]}^{\frac{p-5}{4}} 
\sup_{0\le \theta \le 1} 
\|\bigm\{ |W + \theta \eta_{[n]}|^{p-1} -W^{p-1} \bigm\} 
\eta_{[n]}\|_{L^{\frac{6}{p - 1 + 6 \varepsilon}}}
\\[6pt]
&+ \beta_{[n]} 
\alpha_{[n]}^{\frac{p-5}{4}} 
\sup_{0\le \theta \le 1} 
\|\bigm\{ |W + \theta \eta_{[n]}|^{p-3}(W + \theta \eta_{[n]})^{2} -W^{p-1} 
\bigm\} \overline{\eta_{[n]}}\|_{L^{\frac{6}{p - 1 + 6 \varepsilon}}}
\\[6pt]
&\lesssim 
\|\bigm\{ W^{p-2}+ |\eta_{[n]}|^{p-2} \bigm\} |\eta_{[n]}| 
\|_{L^{\frac{6}{p-1}}} \|\eta_{[n]}\|_{L^{\frac{1}{\varepsilon}}}
\\[6pt]
&\le 
\bigm\{ 
\|W\|_{L^{6}}^{p-2}
+
\|\eta_{[n]}\|_{L^{6}}^{p-2}
\bigm\}
\|\eta_{[n]}\|_{L^{6}}
\|\eta_{[n]}\|_{L^{\frac{1}{\varepsilon}}}
=
o_{n}(1) \|\eta_{[n]}\|_{L^{\frac{1}{\varepsilon}}},
\end{split} 
\end{equation}
where the implicit constant 
depends only on $p$, $\varepsilon$ 
and the quantity $\liminf_{n\to \infty}\omega_{n}$ 
(if finite).

Putting the estimates \eqref{est-Ne0} through \eqref{19/9/13/17:11} together, 
we obtain the desired estimate \eqref{estN-e}. 
\end{proof}

Now, observe from 
\eqref{non-term} and 
\eqref{orthogo-21} that for any $n$, 
\begin{equation} \label{eq1-proof-thm1-gap}
\begin{split}
&-\alpha_{[n]} 
\langle W, \ (-\Delta + \alpha_{[n]})^{-1} V_{+} \Lambda W 
\rangle 
\\[6pt]
&\le 
\beta_{[n]} 
\big| \langle W^{p},\ (-\Delta + \alpha_{[n]})^{-1}V_{+} \Lambda W 
\rangle 
\big| 
+ 
\big| 
\langle N_{[n]}, \ (-\Delta + \alpha_{[n]})^{-1}V_{+}\Lambda W 
\rangle 
\big|
.
\end{split}
\end{equation}
Furthermore, by Lemma \ref{thm-6}, the left-hand side 
of \eqref{eq1-proof-thm1-gap} can be written as 
\begin{equation}
\label{18/06/30/17:18-gap}
-\alpha_{[n]} \langle W, \ 
(-\Delta + \alpha_{[n]})^{-1}V_{+} \Lambda W 
\rangle 
= 
6 \pi \alpha_{[n]}^{\frac{1}{2}}
+ O( \alpha_{[n]}) 
.
\end{equation}
On the other hand, by Lemma \ref{thm-6} (if $2<p\le 3$) and Lemma \ref{thm-6-1} with $\delta=\frac{p-1}{8}$ (if $1<p\le 2$), one can verify that the first term on the right-hand side of \eqref{eq1-proof-thm1-gap} is estimated as follows:
\begin{equation}\label{18/06/30/17:19}
\beta_{[n]} 
\big| \langle W^{p}, \ (-\Delta + \alpha_{[n]})^{-1}V_{+} \Lambda W 
\rangle 
\big|
\lesssim 
\left\{ \begin{array}{lcl}
\beta_{[n]} &\mbox{if} & 2< p \le 3, 
\\[6pt]
\beta_{[n]} \alpha_{[n]}^{-\frac{2-p}{2}- \frac{p-1}{8}}
&\mbox{if}& 1 < p \le 2. 
\end{array} 
\right.
\end{equation}

We consider the second term on the right-hand side 
of \eqref{eq1-proof-thm1-gap}, 
we claim the following:
\begin{lemma}\label{error-estimate-gap}
Assume $1 < p \le 3$ and 
$0<\varepsilon <\frac{p-1}{20}$. 
Then, the following estimate holds 
for any sufficiently large number $n$: 
\begin{align}\label{nest-gap} 
\|(-\Delta + \alpha_{[n]})^{-1}
N_{[n]}\|_{L^{\frac{1}{\varepsilon}}} 
\lesssim \alpha_{[n]}^{\frac{1}{2}+\varepsilon}, 
\end{align}
where the implicit constant depends only on $p$, $\varepsilon$ and the quantity $\liminf_{n\to \infty}\omega_{n}$ (if finite). 
\end{lemma} 
\begin{proof}[Proof of Lemma \ref{error-estimate-gap}]
As in \eqref{est-Ne0}, we see that
\begin{equation}\label{19/9/14/17:13}
\begin{split}
&\|(-\Delta + \alpha_{[n]})^{-1} N_{[n]}
\|_{L^{\frac{1}{\varepsilon}}}
\\[6pt]
&\le 
\|(-\Delta + \alpha_{[n]})^{-1} 
\bigm\{ 
|W + \eta_{[n]}|^{4}(W + \eta_{[n]}) - W^{5} - 5 W^{4} \Re[\eta_{[n]}]-
i W^{4} \Im[\eta_{[n]}]
\bigm\}
\|_{L^{\frac{1}{\varepsilon}}} 
\\[6pt]
&\quad +
\beta_{[n]} \|(-\Delta + \alpha_{[n]})^{-1} 
\big\{ |W + \eta_{[n]}|^{p-1}(W+\eta_{[n]}) 
- W^{p}
\big\} 
\|_{L^{\frac{1}{\varepsilon}}}
.
\end{split}
\end{equation}
From \eqref{explicit-1}, one can verify that 
\begin{equation}\label{19/9/15/12:09}
\begin{split}
&
\bigm|
(-\Delta + \alpha_{[n]})^{-1}
\bigm\{ 
|W + \eta_{[n]}|^{4}(W + \eta_{[n]}) - W^{5} - 5 W^{4} \Re[\eta_{[n]}]
-iW^{5}\Im[\eta_{[n]}]
\bigm\}
\big|
\\[6pt]
&\lesssim 
(-\Delta + \alpha_{[n]})^{-1}\big\{ W^{3}|\eta_{[n]}|^{2}+ |\eta_{[n]}|^{5} \big\}
,
\end{split} 
\end{equation}
and 
\begin{equation}\label{19/9/15/13:45}
\begin{split}
&
\bigm|
(-\Delta + \alpha_{[n]})^{-1}
\bigm\{ |W + \eta_{[n]}|^{p-1}(W+\eta_{[n]}) 
- W^{p}
\bigm\}
\bigm|
\\[6pt]
&\lesssim 
(-\Delta + \alpha_{[n]})^{-1}
\bigm\{ |W|^{p-1} + |\eta_{[n]}|^{p-1} \bigm\} 
|\eta_{[n]}| 
. 
\end{split}
\end{equation}
We consider the first term on the right-hand side of \eqref{19/9/14/17:13}. 
By \eqref{19/9/15/12:09}, Lemma \ref{18/11/23/17:17}, 
Lemma \ref{Delta-ineq}, H\"older's inequality and Lemma \ref{error-estimate-gap-1}, one can see that 
\begin{equation}\label{19/9/15/12:58}
\begin{split}
&
\|
(-\Delta + \alpha_{[n]})^{-1}
\bigm\{ 
|W + \eta_{[n]}|^{4}(W + \eta_{[n]}) - W^{5} - 5 W^{4} \Re[\eta_{[n]}]
-iW^{5}\Im[\eta_{[n]}]
\bigm\}
\|_{L^{\frac{1}{\varepsilon}}}
\\[6pt]
&\lesssim
\| (-\Delta + \alpha_{[n]})^{-1}\big\{ W^{3}|\eta_{[n]}|^{2} \big\} \|_{L^{\frac{1}{\varepsilon}}}
+ 
\| (-\Delta + \alpha_{[n]})^{-1} |\eta_{[n]}|^{5} \|_{L^{\frac{1}{\varepsilon}}}
\\[6pt]
&\lesssim
\|W^{3} |\eta_{[n]}|^{2} \|_{L^{\frac{3}{2+3\varepsilon}}}
+ 
\alpha_{[n]}^{-1+6\varepsilon}\| |\eta_{[n]}|^{5} \|_{L^{\frac{1}{5 \varepsilon}}}
\\[6pt]
&\le 
\|W^{3}\|_{L^{\frac{3}{2-3\varepsilon}}}\| \eta_{[n]} \|_{L^{\frac{1}{\varepsilon}}}^{2}
+
\alpha_{[n]}^{-1+6\varepsilon}\| \eta_{[n]} \|_{L^{\frac{1}{\varepsilon}}}^{5}
\lesssim 
\alpha_{[n]}^{1-4\varepsilon}
+
\alpha_{[n]}^{\frac{3}{2}-4\varepsilon}
\lesssim 
\alpha_{[n]}^{\frac{1}{2}+\varepsilon}
,
\end{split} 
\end{equation}
where the implicit constants depend only on $p$, $\varepsilon$ and the quantity $\liminf_{n\to \infty}\omega_{n}$ (if finite).

Next, we consider the second term on the right-hand side of \eqref{19/9/14/17:13}. It follows from \eqref{19/9/15/13:45}, Lemma \ref{Delta-ineq}, 
\eqref{eq2-1-1}, H\"older's inequality and Lemma \ref{error-estimate-gap-1} that if $\varepsilon$ is sufficiently small depending only on $p$, 
\begin{equation}\label{19/9/15/13:52}
\begin{split}
&
\beta_{[n]}
\|
(-\Delta + \alpha_{[n]})^{-1}
\bigm\{ 
|W + \eta_{[n]}|^{p-1}(W + \eta_{[n]}) - W^{p} \bigm\}
\|_{L^{\frac{1}{\varepsilon}}}
\\[6pt]
&\lesssim
\beta_{[n]}
\| (-\Delta + \alpha_{[n]})^{-1}\big\{ W^{p-1} \eta_{[n]} \big\} \|_{L^{\frac{1}{\varepsilon}}}
+ 
\beta_{[n]}
\| (-\Delta + \alpha_{[n]})^{-1}|\eta_{[n]}|^{p} \|_{L^{\frac{1}{\varepsilon}}}
\\[6pt]
&\lesssim
\beta_{[n]} \alpha_{[n]}^{-\frac{5-p}{4}+3\varepsilon}
\|W^{p-1} \eta_{[n]} \|_{L^{\frac{6}{p-1+18\varepsilon}}}
+ 
\beta_{[n]} 
\alpha_{[n]}^{-1+\frac{3(p-1)\varepsilon}{2}}
\| |\eta_{[n]}|^{p} \|_{L^{\frac{1}{p \varepsilon}}}
\\[6pt]
&\lesssim 
\alpha_{[n]}^{3 \varepsilon}
\|W^{p-1}\|_{L^{\frac{6}{p-1+12\varepsilon}}} \| \eta_{[n]} \|_{L^{\frac{1}{\varepsilon}}}
+
\alpha_{[n]}^{-\frac{p-1}{4}+\frac{3(p-1)\varepsilon}{2}}
\| \eta_{[n]} \|_{L^{\frac{1}{\varepsilon}}}^{p}
\\[6pt]
&\lesssim 
\alpha_{[n]}^{\frac{1}{2}+\varepsilon}
+
\alpha_{[n]}^{\frac{1}{2}+\frac{p-1}{4}-\frac{(p+3)\varepsilon}{2}}
\lesssim 
\alpha_{[n]}^{\frac{1}{2}+\varepsilon}
,
\end{split} 
\end{equation}
where the implicit constants depend 
only on $p$, $\varepsilon$ and the quantity $\liminf_{n\to \infty}\omega_{n}$ (if finite).

Putting \eqref{19/9/14/17:13}, \eqref{19/9/15/12:58} and \eqref{19/9/15/13:52}
together, we obtain the desired estimate \eqref{nest-gap}. 
\end{proof}

Now, we shall finish the proof of Theorem \ref{non-seq}. 
By H\"older's inequality and Lemma \ref{error-estimate-gap}, 
one can estimate the last term on the right-hand side of \eqref{eq1-proof-thm1-gap} as follows: for any $0< \varepsilon<\frac{p-1}{20}$, 
\begin{equation}\label{19/9/15/14:35}
\big|\langle N_{[n]}, \ (-\Delta + \alpha_{[n]})^{-1} 
V_{+} \Lambda W \rangle \big| 
\le 
\|V_{+} \Lambda W \|_{L^{\frac{1}{1-\varepsilon}}}
\| (-\Delta + \alpha_{[n]})^{-1} N_{[n]} \|_{L^{\frac{1}{\varepsilon}}} 
\lesssim 
\alpha_{[n]}^{\frac{1}{2}+\varepsilon},
\end{equation}
where the implicit constants depend only on 
$p$, $\varepsilon$ and the quantity $\liminf\limits_{n\to \infty}\omega_{n}$. 
Then, plugging \eqref{18/06/30/17:18-gap}, \eqref{18/06/30/17:19}, and 
\eqref{19/9/15/14:35} into 
\eqref{eq1-proof-thm1-gap}, 
we see that for any $0< \varepsilon <\frac{p-1}{20}$, 
\begin{equation} \label{eq-proof-thm1-gap}
\begin{split}
\alpha_{[n]}^{\frac{1}{2}}
&\lesssim 
\left\{ \begin{array}{lcl}
\beta_{[n]} + \alpha_{[n]}^{\frac{1}{2}+\varepsilon} &\mbox{if} & 2< p \le 3, 
\\[6pt]
\beta_{[n]} \alpha_{[n]}^{-\frac{2-p}{2}- \frac{p-1}{8}}
+
\alpha_{[n]}^{\frac{1}{2}+\varepsilon}
&\mbox{if}& 1 < p \le 2, 
\end{array} 
\right. 
\end{split}
\end{equation}
where the implicit constant depends 
only on $p$, $\varepsilon$, 
and the quantity $\liminf\limits_{n\to \infty}\omega_{n}$. Furthermore, transposing the second term on the right-hand side of \eqref{eq-proof-thm1-gap} to the left, and dividing both sides 
by $\alpha_{[n]}^{\frac{1}{2}}$, we obtain 
\begin{equation}\label{19/9/15/16:26}
1 
\lesssim 
\left\{ \begin{array}{lcl}
\alpha_{[n]}^{-\frac{1}{2}}\beta_{[n]} &\mbox{if} & 2< p \le 3, 
\\[6pt]
\alpha_{[n]}^{-\frac{1}{2}} \beta_{[n]}
\alpha_{[n]}^{-\frac{2-p}{2}- \frac{p-1}{8}}
&\mbox{if}& 1 < p \le 2, 
\end{array} 
\right. 
. 
\end{equation}
Assume $2< p \le 3$. 
Then, it follows from \eqref{re-a-b-1}, \eqref{eq2-1-1}, \eqref{eq2-1-2} and $\lim\limits_{n\to \infty}\nu_{n}=1$ that 
\begin{equation}\label{19/9/15/17:34}
\lim_{n\to \infty}\alpha_{[n]}^{-\frac{1}{2}}\beta_{[n]}
=
\lim_{n\to \infty} \omega_{n}^{-\frac{5-p}{4}} \alpha_{[n]}^{\frac{3-p}{4}}
=0. 
\end{equation}
However, \eqref{19/9/15/17:34} contradicts \eqref{19/9/15/16:26}. 
Similarly, when $1<p\le 2$, 
from \eqref{re-a-b-1}, \eqref{eq2-1-1}, \eqref{eq2-1-2} and $\lim\limits_{n\to \infty}\nu_{n}=1$, 
we can verify that 
\begin{equation}\label{19/9/15/17:31}
\lim_{n\to \infty}
\alpha_{[n]}^{-\frac{1}{2}} \beta_{[n]}
\alpha_{[n]}^{-\frac{2-p}{2}- \frac{p-1}{8}}
=
\lim_{n\to \infty}
\omega_{n}^{-\frac{5-p}{4}}
\alpha_{[n]}^{\frac{p-1}{4}-\frac{p-1}{8}}
=0.
\end{equation}
This also contradicts \eqref{19/9/15/16:26}. Thus, we have completed the proof of Theorem \ref{non-seq}.
\end{proof}

\section{Proof of Theorem \ref{thm-0}}
\label{section-gs-n}

In this section, we prove Theorem \ref{thm-0}. 
To this end, we will employ Proposition \ref{thm4-1}. 
The following result immediately follows from results in \cite{AIIKN} and Theorem \ref{non-seq}: 
\begin{proposition} \label{thm-4}
Assume 
$1 < p \leq 3$. Then, there exists $\omega_{1} >0$ such that 
for any $\omega > \omega_{1}$, 
there is no ground state (no minimizer for $m_{\omega}^{\mathcal{N}}$). 
\end{proposition}
\begin{proof}[Proof of Proposition \ref{thm-4}] 
We prove by contradiction: If the claim of Proposition \ref{thm-4} was false, then there existed a sequence $\{(\Phi_{n}, \omega_{n})\}$ in $H^{1}(\mathbb{R}^{3}) \times (0,\infty)$ with the following properties:
\\
\noindent 
{\rm (i)}~$\Phi_{n}$ is a ground state for $m_{\omega_{n}}$; we may assume that $\Phi_{n}$ is a positive radial solution to \eqref{sp} with $\omega=\omega_{n}$ (see Remark \ref{remark-1}).
\\
\noindent 
{\rm (ii)}~$\lim_{n\to \infty}\omega_{n}=\infty$. 
\\
Furthermore, defining 
$M_{n}:=\|\Phi_{n}\|_{L^{\infty}}$, 
$\alpha_{n} := \omega_{n} M_{n}^{-4}, 
\beta_{n} := M_{n}^{p-5}$ and 
$\widetilde{\Phi}_{n}:=T_{M_{n}} \Phi_{n}$, 
we can verify that $\widetilde{\Phi}_{n}:=T_{M_{n}}\Phi_{n}$ obeys 
\begin{equation}
- \Delta \widetilde{\Phi}_{n} 
+ \alpha_{n} \widetilde{\Phi}_{n}
- \beta_{n} |\widetilde{\Phi}_{n}|^{p-1}\widetilde{\Phi}_{n} 
- |\widetilde{\Phi}_{n}|^{4}
\widetilde{\Phi}_{n} 
= 0, 
\end{equation}
and 
\\
\noindent 
{\rm (iii)}~$\lim_{n \to \infty} M_{n}= \infty$,
\quad 
$\lim_{n \to \infty}\alpha_{n}=0$ 
\qquad (see Lemma 2.3 of \cite{AIIKN}).
\\
\noindent 
{\rm (iv)}~$\lim_{n \to \infty}\widetilde{\Phi}_{n} = W$
\quad strongly in 
$\dot{H}^{1}(\mathbb{R}^{3})$
\qquad 
(see Proposition 2.1 of \cite{AIIKN}).

However, when $1<p\le 3$, Theorem \ref{non-seq} forbids such a sequence to exist. Thus, the claim of Proposition \ref{thm-4} is true. 
\end{proof} 
From Proposition \ref{thm-4}, we see that 
$\omega_{c}^{\mathcal{N}} < \infty$. 
For $\omega \geq \omega_{c}^{\mathcal{N}}$, we obtain 
the following result: 
\begin{proposition}\label{thm4-2}
Assume $\omega > 0$ and $1 < p \leq 3$. Then, for any $\omega \ge \omega_{c}^{\mathcal{N}}$, we have $m_{\omega}^{\mathcal{N}} = m_{\infty}$. 
\end{proposition}

For a proof of Proposition \ref{thm4-2}, 
we need the following characterization 
of $m_{\omega}^{\mathcal{N}}$: 
\begin{lemma}\label{variational-character4}
Assume $\omega > 0$ and 
$1 < p < 5$. Then, 
\begin{equation}\label{eq0-variational-character4}
m_{\omega}^{\mathcal{N}} 
= 
\inf \bigm\{
\mathcal{\mathcal{J}}(u) 
\colon 
u \in H^{1}(\mathbb{R}^{3})
\setminus \{0\}, \ 
\mathcal{N}_{\omega} (u) \leq 0 
\bigm\}, 
\end{equation}
where $\mathcal{J}$ is the functional defined by 
\eqref{eq-J}. 
\end{lemma}
\begin{proof}[Proof of Lemma \ref{variational-character4}] 
Observe that for any $\omega>0$, 
\begin{equation}\label{re-S-J}
\mathcal{S}_{\omega}(u) - \frac{1}{2} \mathcal{N}_{\omega}(u)
= 
\Big(\frac{1}{2} - \frac{1}{p+1}\Big) \|u\|
_{L^{p+1}}^{p+1} 
+ 
\frac{1}{3}\|u\|_{L^{6}}^{6} 
= \mathcal{J}(u). 
\end{equation}
This yields that 
\begin{equation}\label{eq1-variational-character4}
m_{\omega}^{\mathcal{N}} 
= 
\inf \bigm\{
\mathcal{J}(u) \colon 
u \in H^{1}(\mathbb{R}^{3})
\setminus \{0\}, \ 
\mathcal{N}_{\omega} (u) = 0 
\bigm\}. 
\end{equation}
Hence, it suffices to show that if $\mathcal{N}_{\omega} (u) < 0$, then $m_{\omega}^{\mathcal{N}}<\mathcal{\mathcal{J}}(u)$. Suppose that 
$\mathcal{N}_{\omega}(u)<0$. Then, we can take 
$t_{0} \in (0, 1)$ such that $\mathcal{N}_{\omega} (t_{0} u) = 0$. Furthermore, this together with \eqref{eq1-variational-character4} yields that 
\begin{equation}\label{18/06/24/22:01}
m_{\omega}^{\mathcal{N}} \le \mathcal{\mathcal{J}}(t_{0} u) < \mathcal{\mathcal{J}}(u). 
\end{equation}
Since $u$ is arbitrary, we find from \eqref{18/06/24/22:01} that 
the claim \eqref{eq0-variational-character4} is true. 
\end{proof}
Now, we give a proof of Proposition \ref{thm4-2}. 
\begin{proof}[Proof of Proposition \ref{thm4-2}] 
We divide the proof into two parts: 
\\
{\bf Step 1.}~We shall show that if $\omega>\omega_{c}^{\mathcal{N}}$, 
then we have $m_{\omega}^{\mathcal{N}}=m_{\infty}$. 

First, observe from the definition of 
$\omega_{c}^{\mathcal{N}}$ 
(see \eqref{19/9/18/11:15}) and 
Proposition \ref{thm4-1} that if 
$\omega>\omega_{c}^{\mathcal{N}}$, then 
we see that $m_{\omega}^{\mathcal{N}} \ge m_{\infty}$. 
\par
We prove the claim by contradiction. Hence, suppose to the contrary that there exists 
$\omega_{1} >\omega_{c}^{\mathcal{N}}$ such that $m_{\omega_{1}}^{\mathcal{N}} > m_{\infty}$. 
Put $\mu_{1}:= m_{\omega_{1}}^{\mathcal{N}} - m_{\infty}~(>0)$. 
Moreover, for $\varepsilon>0$, 
we define the functions $W_{\varepsilon}$ and $V_{\varepsilon}$ by
\begin{equation}\label{modi-talenti}
W_{\varepsilon}(x) := 
T_{\varepsilon} W, \qquad 
V_{\varepsilon}(x) : = \chi(|x|) W_{\varepsilon}(x), 
\end{equation}
where $\chi$ is a smooth cut-off function satisfying 
$\chi(r) \equiv 1$ for $0 \leq r \leq 1$, 
$\chi(r) \equiv 0$ for $r \geq 2$ and 
$|\chi^{\prime}(r)| \leq 10$. 
Then, using \eqref{19/9/18/11:15}, we can verify that 
\begin{align} \label{est-Talenti3}
\|\nabla V_{\varepsilon}\|_{L^{2}}^{2} 
&= \sigma^{\frac{3}{2}} + O(\varepsilon^{2}), 
\\[6pt]
\|V_{\varepsilon}\|_{L^{6}}^{6} 
&= \sigma^{\frac{3}{2}} + O(\varepsilon^{6}), 
\label{est-Talenti1}
\\[6pt]
\|V_{\varepsilon}\|_{L^{p+1}}^{p+1} 
&= 
\begin{cases}
O(\varepsilon^{p+1}) & \mbox{if $1\le p< 2$}, 
\\
O(\varepsilon^{3} |\log \varepsilon|) & \mbox{if $p=2$}, 
\\
O(\varepsilon^{5-p}) & \mbox{if $2< p< 5$}. 
\end{cases}
\label{est-Talenti2}
\end{align} 
We see from \eqref{est-Talenti3} 
through \eqref{est-Talenti2} 
that for any $\varepsilon>0$, $t > 0$ 
and $1 < p \leq 3$, 
\begin{equation}\label{18/07/01/10:37}
t^{-2} \mathcal{N}_{\omega_{1}}(t V_{\varepsilon}) 
= 
\sigma^{\frac{3}{2}} + O(\varepsilon^{2}) 
+ \omega_{1} O(\varepsilon^{2})
- 
t^{p-1} \|V_{\varepsilon}\|_{L^{p+1}}^{p+1}
- 
t^{4} \sigma^{\frac{3}{2}}
+ 
t^{4} O(\varepsilon^{6}). 
\end{equation}
In particular, we find that for any sufficiently small 
$\varepsilon>0$, there exists $t_{\varepsilon} > 0$ 
such that 
\begin{equation}\label{18/07/01/10:41}
\mathcal{N}_{\omega_{1}}(t_{\varepsilon} V_{\varepsilon}) = 0, 
\qquad 
t_{\varepsilon} = 1 + O(\varepsilon^{2}).
\end{equation}
Hence, we see from Lemma \ref{variational-character4}, \eqref{18/07/01/10:41}, \eqref{est-Talenti1}, \eqref{est-Talenti2} and $p\le 3$ that for any sufficiently small $\varepsilon>0$, 
\begin{equation}\label{19/9/16/15:38}
\begin{split}
m_{\omega_{1}}^{\mathcal{N}}
& \leq \mathcal{\mathcal{J}}(t_{\varepsilon} V_{\varepsilon}) 
= \left(\frac{1}{2} - \frac{1}{p+1}\right)t_{\varepsilon}^{p+1}
\|V_{\varepsilon}\|_{L^{p+1}}^{p+1} 
+ \frac{t_{\varepsilon}^{3}}{3}
\|V_{\varepsilon}\|_{L^{6}}^{6} 
\\[6pt]
& \le \frac{1}{3}\sigma^{\frac{3}{2}}
+ 
O(\varepsilon^{2})
=
m_{\infty} + O(\varepsilon^{2}). 
\end{split}
\end{equation}
Furthermore, taking $\varepsilon \ll \mu_{1}$ in \eqref{19/9/16/15:38}, and 
recalling the definition of $\mu_{1}$ 
and the hypothesis 
$m_{\omega_{1}}^{\mathcal{N}} > m_{\infty}$, we find that 
\begin{equation}\label{18/07/01/10:52} 
m_{\omega_{1}}^{\mathcal{N}} 
\le 
m_{\infty} + \frac{1}{2}\mu_{1}
=
\frac{1}{2}m_{\omega_{1}}^{\mathcal{N}} 
+
\frac{1}{2} m_{\infty}
<
m_{\omega_{1}}^{\mathcal{N}}.
\end{equation} 
However, this a contradiction. Thus, we have proved that if $\omega > \omega_{c}^{\mathcal{N}}$, then $m_{\omega}^{\mathcal{N}} = m_{\infty}$.

\noindent 
{\bf Step 2.}~We shall show that if $\omega = \omega_{c}^{\mathcal{N}}$, then 
$m_{\omega}^{\mathcal{N}} = m_{\infty}$. 
\par 
Having proved that $m_{\omega}^{\mathcal{N}} = m_{\infty}$ for any $\omega > \omega_{c}^{\mathcal{N}}$, we see from Proposition \ref{thm4-1} that $m_{\omega_{c}^{\mathcal{N}}}^{\mathcal{N}} \le m_{\infty}$. 
Then, it follows from Proposition \ref{thm-0-1} that a ground state $\Phi$ for $\omega_{c}^{\mathcal{N}}$ exists. In particular, 
$\mathcal{N}_{\omega_{c}^{\mathcal{N}}}(\Phi)=0$
and $\mathcal{S}_{\omega_{c}^{\mathcal{N}}}(\Phi)=m_{\omega_{c}^{\mathcal{N}}}^{\mathcal{N}}$. 
Furthermore, we may assume that $\Phi$ is a positive radial solution to \eqref{sp} with $\omega=\omega_{c}^{\mathcal{N}}$ (see Remark \ref{remark-1}). 
Put $\mu_{c}^{\mathcal{N}} 
:= m_{\infty} - 
m_{\omega_{c}^{\mathcal{N}}}^{\mathcal{N}}$
($\mu_{c}^{\mathcal{N}} > 0$ by the hypothesis). 
Moreover, let $\{\omega_{n}\}$ be a sequence in $(0,\infty)$ such that 
$\lim_{n \to \infty}\omega_{n} = \omega_{c}^{\mathcal{N}}$, and $\omega_{n} > \omega_{c}^{\mathcal{N}}$ for all $n \ge 1$. 
We have already proved that for any $n \ge 1$, 
\begin{equation}\label{18/07/01/11:45}
m_{\omega_{n}}^{\mathcal{N}} = m_{\infty}.
\end{equation} 
We also see from $\mathcal{N}_{\omega_{c}^{\mathcal{N}}}(\Phi)=0$ that 
\begin{equation*}
\lim_{n\to \infty} \mathcal{N}_{\omega_{n}}(\Phi) 
= 
\lim_{n\to \infty} 
(\omega_{n} - \omega_{c}^{\mathcal{N}}) \|\Phi \|_{L^{2}}^{2} 
= 0. 
\end{equation*} 
This implies that there exists a sequence $\{t_{n}\}$ in $(0,\infty)$ such that \begin{align} 
\label{18/07/01/11:47}
& \mathcal{N}_{\omega_{n}}(t_{n}\Phi) = 0
\quad 
\mbox{for all $n\ge 1$}, 
\\[6pt]
\label{18/07/01/11:48} 
&\lim_{n \to \infty}t_{n} = 1.
\end{align} 
Furthermore, we can take a number $n_{c}$ such that for any $n\ge n_{c}$, 
\begin{equation}
\label{18/07/01/11:51}
\big| 
\mathcal{S}_{\omega_{n}} (t_{n}\Phi)- \mathcal{S}_{\omega_{c}^{\mathcal{N}}}(t_{n}\Phi) 
\big|
= 
\frac{t_{n}^{2}}{2}|\omega_{c}^{\mathcal{N}} - \omega_{n}|
\|\Phi\|_{L^{2}}^{2}< 
\frac{1}{4}\mu_{c}^{\mathcal{N}}. 
\end{equation}
Observe from \eqref{18/07/01/11:51} and 
$\mathcal{S}_{\omega_{c}^{\mathcal{N}}}(\Phi) =m_{\omega_{c}^{\mathcal{N}}}^{\mathcal{N}}$ that 
\begin{equation}
\label{18/07/01/11:59}
\begin{split}
\mathcal{S}_{\omega_{c}^{\mathcal{N}}} (t_{n}\Phi)
& \leq \mathcal{S}_{\omega_{c}^{\mathcal{N}}} (\Phi) 
+ |\mathcal{S}_{\omega_{c}^{\mathcal{N}}} (t_{n}\Phi) - 
\mathcal{S}_{\omega_{c}^{\mathcal{N}}} (\Phi)| \\
&< \mathcal{S}_{\omega_{c}^{\mathcal{N}}} (\Phi) 
+ \frac{1}{4}\mu_{c}^{\mathcal{N}}
=
m_{\omega_{c}^{\mathcal{N}}}^{\mathcal{N}} 
+ \frac{1}{4}\mu_{c}^{\mathcal{N}}.
\end{split}
\end{equation}
Using \eqref{18/07/01/11:45}, the definition of $m_{\omega}^{\mathcal{N}}$, \eqref{18/07/01/11:47}, 
\eqref{18/07/01/11:51}, \eqref{18/07/01/11:59} 
and $\mu_{c}^{\mathcal{N}} =m_{\infty}-m_{\omega_{c}^{\mathcal{N}}}^{\mathcal{N}}$, 
we find that for any $n \ge n_{c}$, 
\begin{equation*}
m_{\infty}
= 
m_{\omega_{n}}^{\mathcal{N}} \leq \mathcal{S}_{\omega_{n}} (t_{n}\Phi)
< 
\mathcal{S}_{\omega_{c}^{\mathcal{N}} }(t_{n}\Phi) 
+ \frac{1}{4}\mu_{c}^{\mathcal{N}} 
< m_{\omega_{c}^{\mathcal{N}}}^{\mathcal{N}} 
+ \frac{1}{2}\mu_{c}^{\mathcal{N}} 
= m_{\infty} - \frac{1}{2}\mu_{c}^{\mathcal{N}}. 
\end{equation*}
However, this is a contradiction. 
Thus, we have completed the proof. 
\end{proof}

Now, we shall prove Theorem \ref{thm-0}: 
\begin{proof}[Proof of Theorem \ref{thm-0}]
The claim \textrm{(i)} follows immediately 
from Remark \ref{remark-0} \textrm{(ii)} and 
Proposition \ref{thm-4}.

We move on to the proof of the claim {\rm (ii)}. We find from Proposition \ref{thm-0-1} and Proposition \ref{thm4-1} that if $\omega< \omega_{c}^{\mathcal{N}}$, then a ground state exists. It remains to show that for any $\omega>\omega_{c}^{\mathcal{N}}$, there is no ground state (no minimizer for $m_{\omega}^{\mathcal{N}}$). We prove this by contradiction. 
Hence, suppose to the contrary that there exists $\omega_{*} >\omega_{c}^{\mathcal{N}}$ for which a ground state $u_{*}$ exists. 
Since $\mathcal{N}_{\omega_{*}}(u_{*}) = 0$ and $\omega_{*} > \omega_{c}^{\mathcal{N}}$, we see that 
$\mathcal{N}_{\omega_{c}^{\mathcal{N}}}(u_{*}) < 0$. 
Furthermore, it is easy to see that there exists $t_{c} \in (0, 1)$ 
such that $\mathcal{N}_{\omega_{c}^{\mathcal{N}} }
(t_{c} u_{*}) = 0$. Hence, we find from Proposition \ref{thm4-2} and 
the definition of $m_{\omega_{c}}^{\mathcal{N}}$ (see \eqref{m-value}) 
that 
\begin{equation*}
\begin{split}
m_{\infty}
&=
m_{\omega_{c}^{\mathcal{N}}}^{\mathcal{N}} 
\le 
\mathcal{S}_{\omega_{c}^{\mathcal{N}}}(t_{c} u_{*})- \frac{1}{2}\mathcal{N}_{\omega_{c}^{\mathcal{N}}}(t_{c}u_{*})
=
\Big(\frac{1}{2} - \frac{1}{p+1}\Big) \|t_{c}u_{*}\|_{L^{p+1}}^{p+1} 
+ 
\frac{1}{3} \|t_{c}u_{*}\|_{L^{6}}^{6}
\\[6pt]
&< \Big(\frac{1}{2} - \frac{1}{p+1}\Big) \|u_{*}\|_{L^{p+1}}^{p+1} 
+ 
\frac{1}{3} \|u_{*}\|_{L^{6}}^{6}
=
\mathcal{S}_{\omega_{*}}(u_{*})- \frac{1}{2}\mathcal{N}_{\omega_{*}}(u_{*})
\\[6pt]
&=
\mathcal{S}_{\omega_{*}}(u_{*})
=
m_{\omega_{*}}^{\mathcal{N}}
=
m_{\infty}. 
\end{split}
\end{equation*}
However, this is a contradiction. 
Thus, we have proved the claim (ii) and completed 
the proof of Theorem \ref{thm-0}.
\end{proof}



\section{Proof of Theorem \ref{thm-2mini-1}} 
\label{section-gap}
Our aim in this section is to prove 
Theorem \ref{thm-2mini-1}. 
First, we recall the following fact: 
\begin{lemma} \label{re-regularity} 
Let $1<p<5, 
\omega>0$ and $u_{\omega} \in H^{1}(\mathbb{R}^{3})$ be a solution to \eqref{sp}. 
Then, 
there exists $(x_{0}, \theta_{0}) \in \mathbb{R}^{3}\times [-\pi, \pi)$ such that 
$\|u_{\omega}\|_{L^{\infty}}=e^{i\theta_{0}}u_{\omega}(x_{0})$.
\end{lemma}
\begin{proof}
It is known that 
$u_{\omega} 
\in W_{\text{loc}}^{2, q}(\mathbb{R}^{3})$ 
for all $2 < q < \infty$ and 
$\lim_{|x| \to \infty} u_{\omega}(x) = 0$ 
(see e.g. Brezis and Lieb~\cite[Theorem 2.3]{Brezis-Lieb}). 
Furthermore, using Schauder's estimate, 
we can verify that $u_{\omega} 
\in C^{2}(\mathbb{R}^{3})$ 
(see e.g. 
Gilberg and Trudinger~\cite[Theorems 6.2 and 6.6]{GT}, 
Struwe~\cite[Appendix B]{Struwe}). 
Therefore, there exists 
$(x_{0}, \theta_{0}) \in \mathbb{R}^{3} \times [-\pi, \pi)$ such that 
$e^{i \theta_{0}}u_{\omega}(x_{0}) 
= \|u_{\omega}\|_{L^{\infty}} < \infty$. 
\end{proof}
The key to the proof of Theorem \ref{thm-2mini-1} 
is the following proposition: 
\begin{proposition} \label{thm-conv}
Assume $1<p\le 3$ and 
$\omega > \omega_{c}^{\mathcal{N}}$. 
Suppose that 
\begin{equation}\label{hypo-eq}
m_{\omega}^{S} 
= m_{\omega}^{\mathcal{N}}.
\end{equation} 
Then, for any minimizing sequence $\{u_{n}\}$
for $m_{\omega}^{S}$, and 
there exists a sequence $\{(x_{n},\theta_{n})\}$ 
in $\mathbb{R}^{3}\times [-\pi, \pi)$ and 
a subsequence of $\{u_{n}\}$ (still denoted by the same symbol) 
such that putting $\alpha_{n}:=\omega \|u_{n}\|_{L^{\infty}}^{-4}$ 
and $\widetilde{u}_{n}:=T_{\|u_{n}\|_{L^{\infty}}}[ e^{i\theta_{n}}u_{n}(\cdot+x_{n})]$, 
we have: 
$\|u_{n}\|_{K^{\infty}}
=
e^{i\theta_{n}}u_{n}(x_{n})$, 
\begin{align}
\label{19/9/10/18:17}
\lim_{n\to \infty}\|u_{n}\|_{L^{\infty}}
& = \infty, 
\\[6pt]
\lim_{n\to \infty}\alpha_{n}
& = 0, 
\\[6pt] 
\label{conv-W}
\lim_{n \to \infty} \widetilde{u}_{n} 
&= W \qquad 
\mbox{strongly in $\dot{H}^{1}(\mathbb{R}^{3})$}. 
\end{align}
\end{proposition}
For the proof of Proposition \ref{thm-conv}, 
we also use the following fact (Lemma \ref{19/9/7/17:12} 
through Lemma \ref{thm-bl}): 
\begin{lemma}\label{19/9/7/17:12}
Assume 
$1 < p < 5$. 
Then, there exists a constant $C_{\omega} > 0$ depending only on $p$ and $\omega > 0$ such that 
\begin{equation} \label{m-lower}
m_{\omega}^{\mathcal{N}} \geq C_{\omega}
. 
\end{equation}
\end{lemma}
\begin{proof}[Proof of Lemma \ref{19/9/7/17:12}] 
For any $\varepsilon>0$, Lemma \ref{variational-character4} 
enables us to take
a function $u_{\varepsilon} \in H^{1}
(\mathbb{R}^{3})\setminus \{0\}$ such that 
$\mathcal{N}_{\omega}(u_{\varepsilon}) \le 0$ and $\mathcal{J}(u_{\varepsilon}) \le m_{\omega}^{\mathcal{N}}+\varepsilon$. 
Then, it follows from $\mathcal{N}_{\omega}
(u_{\varepsilon}) \le 0$ and elementary computations that 
\begin{equation}\label{19/9/7/15:19}
\|\nabla u_{\varepsilon}\|_{L^{2}}^{2} 
+ 
\omega \|u_{\varepsilon}\|_{L^{2}}^{2} 
\le 
\|u_{\varepsilon}\|_{L^{p+1}}^{p+1} 
+ 
\|u_{\varepsilon}
\|_{L^{6}}^{6} 
\le 
\frac{\omega}{2} 
\|u_{\varepsilon}\|_{L^{2}}^{2} 
+ 
C(\omega)
\|u_{\varepsilon}
\|_{L^{6}}^{6},
\end{equation}
where $C(\omega)>0$ is some constant depending only on $p$ and $\omega$. Notice that the first term on the right-hand side of 
\eqref{19/9/7/15:19} can be absorbed into the second term 
on the left-hand side. Hence, we see from Sobolev's inequality \eqref{19/9/7/15:31} and \eqref{19/9/7/15:19} that 
\begin{equation}\label{19/9/7/16:13}
\sigma \|u_{\varepsilon}\|_{L^{6}}^{2} 
\le 
\|\nabla u_{\varepsilon}\|_{L^{2}}^{2} 
\le 
C(\omega)\|u_{\varepsilon}\|_{L^{6}}^{6}. 
\end{equation}
This together with $\mathcal{J}(u_{\varepsilon})\le m_{\omega}^{\mathcal{N}}+\varepsilon$ implies 
\begin{equation}\label{19/9/7/16:15}
\frac{1}{d}
\left( \frac{\sigma}{C(\omega)}\right)^{\frac{3}{2}} 
\le 
\frac{1}{3}\|u_{\varepsilon}
\|_{L^{6}}^{6} 
\le 
\mathcal{J}(u_{\varepsilon}) 
\le 
m_{\omega}^{\mathcal{N}}+\varepsilon. 
\end{equation}
Taking $\varepsilon \to 0$ in \eqref{19/9/7/16:15} yields \eqref{m-lower}. 
\end{proof}

Furthermore, we will employ the following lemmas (Lemma \ref{thm-fll} through Lemma \ref{thm-bl}):
\begin{lemma}[Fr\"ohlich, Lieb and Loss~\cite{FLL}]\label{thm-fll}
Let $d\ge 1$, $1<p_{1} <p_{2} <p_{3}$, and let $C_{1},C_{2}, C_{3}$ be positive constants. Then, there exist constants $C>0$ and $\eta>0$ such that for any measurable function $u$ on $\mathbb{R}^{d}$ satisfying 
\begin{align}
\label{16/05/05/07:01}
& \|u\|_{L^{p_{1}}}\le C_{1},
\\[6pt]
\label{16/05/05/07:02}
&C_{2}\le \|u\|_{L^{p_{2}}},
\\[6pt]
\label{16/05/05/07:03}
&\|u\|_{L^{p_{3}}}\le C_{3},
\end{align}
we have 
\begin{equation}\label{16/03/25/10:52}
|\{ x\in \mathbb{R}^{d} \colon |u(x)| >\eta \}|\ge C
.
\end{equation}
\end{lemma}

\begin{lemma}[Lieb~\cite{Lieb}]\label{thm-l}
Let $d\ge 1$, and let $\{u_{n}\}$ be a bounded sequence in $H^{1}(\mathbb{R}^{d})$. 
Assume that 
there exist $C>0$ and $\eta>0$ such that 
\begin{equation}\label{16/03/25/10:58}
\inf_{n \in \mathbb{N}}|\{ x\in \mathbb{R}^{d} 
\colon |u_{n}(x)| >\eta \}|\ge C.
\end{equation}
Then, there exist a subsequence of $\{u_{n}\}$ (still denoted by the same symbol $\{u_{n}\}$), 
a nontrivial function $u_{\infty}\in H^{1}(\mathbb{R}^{d})$ and a sequence $\{y_{n}\}$ in $\mathbb{R}^{d}$ such that
\begin{equation}\label{16/03/25/11:02}
\lim_{n\to \infty}u_{n}(\cdot +y_{n})=u_{\infty} 
\qquad 
\mbox{weakly in $H^{1}(\mathbb{R}^{d})$.}
\end{equation}
\end{lemma}

\begin{lemma}[Brezis and 
Lieb~\cite{Brezis-Lieb}]\label{thm-bl}
Let $d\ge 1$, $1\le r <\infty$, 
$u_{\infty}\in H^{1}(\mathbb{R}
^{d})$, and let $\{u_{n}\}$ be a bounded sequence in $\dot{H}^{1}(\mathbb{R}^{d})$ such that 
\begin{align}
\label{19/9/8/12:8}
\lim_{n\to \infty}u_{n}
&=u_{\infty} 
\quad 
\mbox{weakly in $\dot{H}^{1}(\mathbb{R}^{d})$, and almost everywhere in $\mathbb{R}^{d}$},
\\[6pt]
\sup_{n\ge 1}\|u_{n}\|_{L^{r}}
&<\infty. 
\end{align}
Then, we have:
\begin{equation}\label{16/04/21/13:15}
\lim_{n\to \infty} 
\int_{\mathbb{R}^{d}} 
\Bigm\{ 
|\nabla u_{n}|^{2}-|\nabla \{ u_{n}-u_{\infty}\}|^{2}-|\nabla u_{\infty}|^{2} \Bigm\} \,dx =0,
\end{equation}
and 
\begin{equation}
\label{16/03/25/11:12}
\lim_{n\to \infty}
\int_{\mathbb{R}^{d}} 
\Bigm| |u_{n}|^{r}-|u_{n}-u_{\infty}|^{r}-|u_{\infty}|^{r} 
\Bigm| \,dx =0.
\end{equation}
\end{lemma}

Now, we shall give a proof of Proposition \ref{thm-conv}: 
\begin{proof}[Proof of Proposition \ref{thm-conv}]
Let $\omega>\omega_{c}^{\mathcal{N}}$, and assume $m_{\omega}^{S}=m_{\omega}^{\mathcal{N}}$. Furthermore, let $\{u_{n}\}$ be a minimizing sequence for $m_{\omega}^{S}$ (hence each $u_{n}$ is a solution to \eqref{sp}). 
Note that for any $n \in \mathbb{N}$, 
\begin{equation}\label{19/9/6/15:52}
\mathcal{N}_{\omega}(u_{n})=0. 
\end{equation} 
It follows from Proposition \ref{thm4-2} and 
\eqref{sob-bc-rela} that 
\begin{equation}\label{19/9/6/15:35}
\lim_{n \to \infty} \mathcal{S}_{\omega}(u_{n})
=
m_{\omega}^{S}=m_{\omega}^{\mathcal{N}}=m_{\infty}
=\frac{1}{3}\sigma^{\frac{3}{2}}. 
\end{equation}

We divide the proof into several steps:
\\ 
\noindent 
{\bf Step 1.}~We shall show that the sequence $\{u_{n}\}$ is bounded in $H^{1}(\mathbb{R}^{3})$; more precisely, there exists a number $N_{1}$ such that for any $n\ge N_{1}$, 
\begin{equation}\label{thm-bound}
\|\nabla u_{n}\|_{L^{2}}^{2} + \omega \|u_{n}\|_{L^{2}}^{2} 
\le 
\frac{4(p+1)}{p-1}m_{\infty}
.
\end{equation} 

We see from \eqref{19/9/6/15:35} and \eqref{19/9/6/15:52} that there exists a number $N_{1}$ such that for any $n \ge N_{1}$, 
\begin{equation}\label{19/9/6/16:11}
\begin{split}
2 m_{\infty} 
&\ge \mathcal{S}_{\omega} (u_{n}) 
= \mathcal{S}_{\omega} (u_{n}) - \frac{1}{p+1}\mathcal{N}_{\omega}(u_{n}) 
\\[6pt]
& = \left(\frac{1}{2} - \frac{1}{p+1}\right)
\left\{\|\nabla u_{n}\|_{L^{2}}^{2} + \omega \|u_{n}\|_{L^{2}}^{2} 
\right\} 
+ \left(\frac{1}{p+1} - \frac{1}{6}\right)\|u_{n}\|_{L^{6}}^{6}
.
\end{split}
\end{equation}
This gives us the desired estimate \eqref{thm-bound}. 

\noindent 
{\bf Step 2.}~We shall show that 
\begin{equation} \label{eq-p+1}
\lim_{n \to \infty} \|u_{n}\|_{L^{p+1}} = 0.
\end{equation} 

Suppose to the contrary that the claim \eqref{eq-p+1} was false. 
Then, passing to some subsequence, we may assume that there exists $C>0$ such that 
\begin{equation} \label{lower-bound}
\inf_{n \in \mathbb{N}}\|u_{n}\|_{L^{p+1}} \ge C.
\end{equation} 
Furthermore, we see from \eqref{lower-bound}, \eqref{thm-bound}, Lemma \ref{thm-fll} and Lemma \ref{thm-l} that there exist a nontrivial function $u_{\infty} \in H^{1}(\mathbb{R}^{3})$ and a sequence $\{y_{n}\}$ in $\mathbb{R}^{3}$ 
such that, passing to some subsequence, we have 
\begin{equation}\label{19/9/7/12:9}
\lim_{n \to \infty} u_{n}(\cdot + y_{n}) = u_{\infty} 
\quad \mbox{weakly in $H^{1}(\mathbb{R}^{3})$, and almost everywhere in $\mathbb{R}^{3}$}. 
\end{equation} 
Here, Lemma \ref{thm-bl} together with \eqref{thm-bound}, \eqref{19/9/7/12:9} and \eqref{19/9/6/15:52} shows that 
\begin{equation}\label{eq-N2} 
\lim_{n \to \infty} 
\mathcal{N}_{\omega}\big(u_{n}(\cdot+y_{n}) - u_{\infty}\big) 
= - \mathcal{N}_{\omega}(u_{\infty}). 
\end{equation}
Moreover, Lemma \ref{thm-bl} together with 
\eqref{19/9/7/12:9}, \eqref{re-S-J}, 
$\mathcal{J} =\mathcal{S}_{\omega} - 
\frac{1}{2}\mathcal{N}_{\omega}$, \eqref{19/9/6/15:52} and $u_{\infty}\not \equiv 0$ shows that 
\begin{equation}\label{19/9/7/14:35}
0 < \mathcal{J}(u_{\infty}) = 
m_{\omega}^{\mathcal{N}} - \mathcal{J}\big(u_{n}(\cdot+y_{n}) - u_{\infty}\big)
+o_{n}(1). 
\end{equation}

Now, suppose for contradiction that $\mathcal{N}_{\omega}(u_{\infty})>0$. 
Then, it follows from \eqref{eq-N2} that for any sufficiently large $n$, 
\begin{equation}\label{19/9/7/14:37}
\mathcal{N}_{\omega}\big(u_{n}(\cdot+y_{n}) - u_{\infty}\big) < 0
.
\end{equation}
Furthermore, Lemma \ref{variational-character4} together with \eqref{19/9/7/14:37} shows that for any sufficiently large $n$,
\begin{equation}\label{19/9/7/14:38}
m_{\omega}^{\mathcal{N}} 
\le 
\mathcal{J}\big(u_{n}(\cdot+y_{n}) - u_{\infty}\big)
. 
\end{equation} 
However, plugging \eqref{19/9/7/14:38} into \eqref{19/9/7/14:35} reaches a contradiction. Thus, we have $\mathcal{N}_{\omega}(u_{\infty}) \le 0$. Then, we see
from Lemma \ref{variational-character4}, the lower semicontinuity of the weak limit, \eqref{19/9/6/15:35} and \eqref{19/9/6/15:52} that
\begin{equation*}
m_{\omega}^{\mathcal{N}} 
\le 
\mathcal{J}(u_{\infty}) 
\le 
\liminf_{n \to \infty} \mathcal{J}(u_{n}) 
=
\liminf_{n \to \infty} \big\{ \mathcal{S}_{\omega}(u_{n})-\frac{1}{p+1}
\mathcal{N}_{\omega}(u_{n})
\big\}
= 
m_{\omega}^{\mathcal{N}}, 
\end{equation*}
so that 
\begin{equation}\label{19/9/7/14:55} 
\mathcal{J}(u_{\infty}) = m_{\omega}^{\mathcal{N}}.
\end{equation} 
Furthermore, we have $\mathcal{N}_{\omega}(u_{\infty}) = 0$; otherwise 
(namely when $\mathcal{N}_{\omega}(u_{\infty}) < 0$) $u_{\infty}$ 
must obey $m_{\omega}^{\mathcal{N}} < \mathcal{J}(u_{\infty})$ 
by the same argument as the proof of Lemma \ref{variational-character4}. 
Thus, $u_{\infty}$ is a minimizer for $m_{\omega}^{\mathcal{N}}$ 
(ground state). 
However, since $\omega>\omega_{c}^{\mathcal{N}}$, 
this contradicts Theorem \ref{thm-0} {\rm (ii)}. 
Hence, the claim \eqref{eq-p+1} must hold.

\noindent 
{\bf Step 3.}~We shall show that 
\begin{equation} \label{eq-lq-conv}
\lim_{n \to \infty} \|u_{n}\|_{L^{2}} = 0.
\end{equation}

Since $u_{n}$ is an $H^{1}$-solution to \eqref{sp}, we have $\mathcal{K}(u_{n})=0$ (see \eqref{func-k}) and $\mathcal{N}(u_{n})=0$. Furthermore, the difference between the equations $\mathcal{K}(u_{n})=0$ and $\mathcal{N}(u_{n})=0$ yields the following identity (Pohozaev's identity):
\begin{equation}\label{19/9/8/14:32}
\omega \|u_{n}\|_{L^{2}}^{2} 
= 
\frac{5-p}{2(p+1)} 
\|u_{n}\|_{L^{p+1}}^{p+1}. 
\end{equation}
This together with \eqref{eq-p+1} shows \eqref{eq-lq-conv}.

\noindent 
{\bf Step 4.}~We shall show that 
\begin{equation} \label{eq-6}
\lim_{n \to \infty} \|u_{n}\|_{L^{6}}^{6} 
= 3 m_{\infty}. 
\end{equation}

We see from \eqref{19/9/6/15:35}, \eqref{19/9/6/15:52} 
$\mathcal{J}=\mathcal{S}_{\omega}-\frac{1}{2}\mathcal{N}_{\omega}$
and \eqref{eq-p+1} 
that 
\begin{equation}\label{19/9/7/16:51}
m_{\infty} 
= 
\lim_{n \to \infty} \mathcal{J}(u_{n})
= 
\frac{1}{3} \lim_{n \to \infty} \|u_{n}\|_{L^{6}}^{6}, 
\end{equation}
so that \eqref{eq-6} holds. 

\noindent 
{\bf Step 5.}~We shall show that 
\begin{equation} \label{eq-infty}
\lim_{n \to \infty} \|u_{n}\|_{L^{\infty}} = \infty. 
\end{equation}

Suppose to the contrary that 
$\liminf_{n \to \infty} \|u_{n}\|_{L^{\infty}}<\infty$. Then, passing to some subsequence, we have $\sup_{n \ge 1} \|u_{n}\|_{L^{\infty}} < \infty$. Furthermore, we see from 
$m_{\infty}=m_{\omega}^{\mathcal{N}}$ (see \eqref{19/9/6/15:52}), 
\eqref{eq-lq-conv} and \eqref{eq-6} that the subsequence satisfies 
\begin{equation}\label{19/9/7/17:23}
3 m_{\infty} = \lim_{n \to \infty}\|u_{n}\|_{L^{6}}^{6}
\le
\sup_{n\ge 1}\|u_{n}\|_{L^{\infty}}^{4} 
\lim_{n \to \infty} \|u_{n}\|_{L^{2}}^{2} = 0
. 
\end{equation}
However, this contradicts Lemma \ref{19/9/7/17:12}. Thus, \eqref{eq-infty} holds. 

\noindent 
{\bf Step 6.}~We shall finish the proof of Proposition \ref{thm-conv}. 

By Lemma \ref{re-regularity}, 
we can take a sequence 
$\{(x_{n}, \theta_{n})\}$ 
in $\mathbb{R}^{3} 
\times [-\pi, \pi)$ such that $e^{i\theta_{n}}u_{n}(x_{n})=\|u_{n}\|_{L^{\infty}}$. Then, define $M_{n}$ and $\widetilde{u}_{n}$ 
by 
$M_{n}:= \|u_{n}\|_{L^{\infty}}$, and $\widetilde{u}_{n}:=T_{M_{n}}[e^{i\theta_{n}}u_{n} (\cdot+x_{n})]=M_{n}^{-1}e^{i\theta_{n}}u_{n}(M_{n}^{-2}\cdot +x_{n})$. We can verify that $\widetilde{u}_{n}$ satisfies 
\begin{align}
\label{eq-msp}
- \Delta \widetilde{u}_{n} + \alpha_{n} \widetilde{u}_{n} 
- \beta_{n} |\widetilde{u}_{n}|^{p-1} \widetilde{u}_{n} 
- |\widetilde{u}_{n}|^{4} \widetilde{u}_{n} 
&= 0 
, 
\\[6pt]
\label{19/9/6/17:45}
\|\widetilde{u}_{n}\|_{L^{\infty}}&=1
, 
\end{align}
where 
\begin{equation} \label{eq-coeffi}
\alpha_{n} := \omega M_{n}^{-4}, \qquad 
\beta_{n} := M_{n}^{p-5}. 
\end{equation}
Observe from \eqref{eq-infty} that 
\begin{equation}\label{19/9/6/17:35}
\lim_{n \to \infty} \alpha_{n} 
= \lim_{n \to \infty} \beta_{n} 
= 0,
\end{equation}
and from \eqref{thm-bound} that for any $n\ge N_{1}$, 
\begin{equation}\label{19/9/6/17:18}
\|\nabla \widetilde{u}_{n}\|_{L^{2}}^{2} 
\le 
\frac{4(p+1)}{p-1}m_{\infty}
.
\end{equation} 
Moreover, multiplying the equation \eqref{eq-msp} by 
the complex conjugate of $\widetilde{u}_{n}$, and 
integrating the resulting equation 
over $\mathbb{R}^{3}$, we obtain 
\begin{equation}\label{19/9/8/14:15}
\| \nabla \widetilde{u}_{n}\|_{L^{2}}^{2}
=
\| \widetilde{u}_{n}\|_{L^{6}}^{6}
-
\alpha_{n} \| \widetilde{u}_{n}\|_{L^{2}}^{2}
+
\beta_{n}
\| \widetilde{u}_{n}\|_{L^{p+1}}^{p+1}
.
\end{equation}

Next, use \eqref{19/9/6/17:18}, $\|\widetilde{u}_{n}\|_{L^{\infty}}=1$, the $W_{loc}^{2,q}$ estimate, and Schauder's estimate (see \cite{GT}) to verify that there exists $u_{\infty} \in \dot{H}^{1}(\mathbb{R}^{3})$ such that, passing to some subsequence, 
\begin{equation}\label{19/9/6/16:47}
\lim_{n\to \infty}\widetilde{u}_{n} 
= u_{\infty} 
\quad 
\text{weakly in $\dot{H}^1(\mathbb{R}^{3})$, and strongly in 
$C^2_{\rm loc}(\mathbb{R}^{3})$}.
\end{equation}
Furthermore, Lemma \ref{thm-bl} together with \eqref{19/9/6/17:18} and \eqref{19/9/6/16:47} shows that 
\begin{align}
\label{eq-bl-H}
& \lim_{n \to \infty} \left\{ 
\mathcal{H}^{\ddagger}(\widetilde{u}_{n}) 
- \mathcal{H}^{\ddagger}(\widetilde{u}_{n} - u_{\infty}) 
- \mathcal{H}^{\ddagger}(u_{\infty})
\right\} = 0, 
\\[6pt]
\label{eq-bl-K}
& \lim_{n \to \infty} \left\{ 
\mathcal{N}^{\ddagger}(\widetilde{u}_{n}) 
- \mathcal{N}^{\ddagger}(\widetilde{u}_{n} - u_{\infty}) 
- \mathcal{N}^{\ddagger}(u_{\infty})
\right\} = 0, 
\end{align}
where $\mathcal{H}^{\ddagger}$ and $\mathcal{N}^{\ddagger}$ are the functional defined by \eqref{h-infty} and \eqref{n-infty}, respectively.

We see from \eqref{eq-msp}, \eqref{19/9/6/16:47}, \eqref{19/9/6/17:35} and \eqref{19/9/6/17:45} that the limit $u_{\infty}$ satisfies 
\begin{equation}\label{eq-talenti}
\left\{ \begin{array}{l}
-\Delta u_{\infty} 
= |u_{\infty}|^{4}u_{\infty} \quad {\rm in} \ \mathbb{R}^{3},
\\[6pt]
u_{\infty}(0)=1.
\end{array} \right. 
\end{equation}
Moreover, it is easy to verify that 
\begin{equation}\label{19/9/8/11:37}
\mathcal{N}^{\ddagger}(u_{\infty})
=
\|\nabla u_{\infty}\|_{L^{2}}^{2} 
- 
\|u_{\infty}\|_{L^{6}}^{6}
=0.
\end{equation}

Now, assume that $u_{\infty}$ attains the equality in Sobolev's inequality \eqref{19/9/7/15:31}:
\begin{equation}\label{19/9/8/11:18}
\|\nabla u_{\infty}\|_{L^{2}}^{2} 
= 
\sigma \|u_{\infty}\|_{L^{6}}^{2}.
\end{equation}
Then, by the results in \cite{Aubin, Talenti}, $u_{\infty}= z_{0} \lambda_{0}^{-1} W (\lambda_{0}^{-2} \cdot + x_{0})$ for some $\lambda_{0}>0$, $x_{0}\in \mathbb{R}^{3}$ and $z_{0} \in \mathbb{C}$. Furthermore, it follows from 
$\|u_{\infty}\|_{L^{\infty}}=1$, \eqref{19/9/8/11:37} and $1=W(0)>W(x)$ for all $x\in \mathbb{R}^{3}\setminus \{0\}$ that $\lambda_{0}=z_{0}=1$ and $x_{0}=0$. Thus, \eqref{19/9/8/11:18} implies that $u_{\infty}=W$.

We shall prove \eqref{19/9/8/11:18}. We see from Sobolev's inequality \eqref{19/9/7/15:31} and \eqref{19/9/8/11:37} that 
\begin{equation} \label{sobo-eq1} 
\sigma\|u_{\infty}\|_{L^{6}}^{2} \le \|u_{\infty}\|_{L^{6}}^{6}.
\end{equation}
By \eqref{19/9/8/11:37}, \eqref{sobo-eq1} and 
\eqref{sob-bc-rela}, we see that 
\begin{equation}\label{eq-energy}
\mathcal{H}^{\ddagger}(u_{\infty}) 
= 
\frac{1}{2} \|\nabla u_{\infty}\|_{L^{2}}^{2}
-
\frac{1}{6} \|u_{\infty}\|_{L^{6}}^{6} 
=
\frac{1}{3} \|u_{\infty}\|_{L^{6}}^{6} 
\ge 
\frac{1}{3}\sigma^{\frac{3}{2}}
=
m_{\infty}
.
\end{equation}
On the other hand, we see from \eqref{19/9/8/14:15}, 
\eqref{eq-coeffi}, 
\eqref{eq-lq-conv}, \eqref{eq-p+1} and \eqref{eq-6} that 
\begin{equation} \label{eq-H-n}
\begin{split}
\lim_{n \to \infty}
\mathcal{H}^{\ddagger}(\widetilde{u}_{n}) 
& = 
\lim_{n \to \infty} 
\Bigm\{ 
\frac{1}{3} \|\widetilde{u}_{n}\|_{L^{6}}^{6} 
- \frac{\alpha_{n}}{2} 
\|\widetilde{u}_{n}\|_{L^{2}}^{2} 
+ 
\frac{\beta_{n}}{2} 
\|\widetilde{u}_{n}\|_{L^{p+1}}^{p+1}
\Bigm\} 
\\[6pt]
& = \lim_{n \to \infty} 
\Bigm\{ 
\frac{1}{3} \|u_{n}\|_{L^{6}}^{6} - \frac{\omega}{2} 
\|u_{n}\|_{L^{2}}^{2} + \frac{1}{2} 
\|u_{n}\|_{L^{p+1}}^{p+1}
\Bigm\} 
= 
m_{\infty}
. 
\end{split}
\end{equation}
Similarly, we can verify that 
\begin{equation} \label{eq-K-n}
\lim_{n \to \infty}
\mathcal{N}^{\ddagger}(\widetilde{u}_{n}) 
= 
\lim_{n \to \infty} 
\bigm\{ 
- \alpha_{n} 
\|\widetilde{u}_{n}\|_{L^{2}}^{2} + \beta_{n}
\|\widetilde{u}_{n}\|_{L^{p+1}}^{p+1}
\bigm\} 
= 0. 
\end{equation}

Putting \eqref{eq-bl-K}, \eqref{eq-K-n} and \eqref{19/9/8/11:37} together, we obtain 
\begin{equation}\label{19/9/8/14:57}
\lim_{n \to \infty} \mathcal{N}^{\ddagger}(\widetilde{u}_{n} - u_{\infty}) 
= 
\lim_{n \to \infty}
\bigm\{ 
\|\nabla (\widetilde{u}_{n} - u_{\infty})\|_{L^{2}}^{2}
-
\|\widetilde{u}_{n} - u_{\infty}\|_{L^{6}}^{6}
\bigm\}
=
0.
\end{equation} 
Furthermore, we see from \eqref{19/9/8/14:57} that 
\begin{equation} \label{eq-H-n2}
\mathcal{H}^{\ddagger}(\widetilde{u}_{n} - u_{\infty}) 
= 
\frac{1}{3} \| \nabla ( \widetilde{u}_{n} - u_{\infty})\|_{L^{2}}^{2} +
o_{n}(1)
. 
\end{equation}

Now, we find from \eqref{eq-energy}, \eqref{eq-bl-H}, \eqref{eq-H-n2} and 
\eqref{eq-H-n} 
that 
\begin{equation}\label{19/9/8/15:10}
\begin{split}
m_{\infty} 
\le
\mathcal{H}^{\ddagger}(u_{\infty}) 
&=
\mathcal{H}^{\ddagger}(\widetilde{u}_{n}) 
- \mathcal{H}^{\ddagger}(\widetilde{u}_{n} - u_{\infty})
+
o_{n}(1)
\\[6pt]
&\le
\mathcal{H}^{\ddagger}(\widetilde{u}_{n}) 
+
o_{n}(1)
= 
m_{\infty}
+
o_{n}(1)
. 
\end{split} 
\end{equation} 
Using \eqref{19/9/8/11:37}, and taking $n\to \infty$ in \eqref{19/9/8/15:10},
we see that 
\begin{equation}\label{19/9/8/16:52}
\frac{1}{3}\|u_{\infty}\|_{L^{6}}^{6}=\mathcal{H}^{\ddagger}(u_{\infty}) = m_{\infty}
.
\end{equation} 
Furthermore, this together with \eqref{19/9/8/11:37} and \eqref{sob-bc-rela} implies that 
\begin{equation}\label{19/9/8/16:45}
\|\nabla u_{\infty}\|_{L^{2}}^{2}=\|u_{\infty}\|_{L^{6}}^{6}=\sigma^{\frac{3}{2}}.
\end{equation}
Then, the claim \eqref{19/9/8/11:18} follows from \eqref{19/9/8/16:45}, and 
therefore $u_{\infty}=W$. In addition, it follows from \eqref{eq-bl-H}, \eqref{eq-H-n}, \eqref{eq-H-n2} and \eqref{19/9/8/16:52} that 
\begin{equation}\label{19/9/8/17:05}
\frac{1}{3}\|\nabla (\widetilde{u}_{n} -u_{\infty}) \|_{L^{2}}^{2}
=o_{n}(1),
\end{equation}
so that $\lim_{n\to \infty}\widetilde{u}_{n} =u_{\infty}=W$ strongly in $\dot{H}^{1}(\mathbb{R}^{3})$. Thus, we have completed the proof. 
\end{proof}

We are now in a position to give a proof of 
Theorem \ref{thm-2mini-1}: 
\begin{proof}[Proof of Theorem \ref{thm-2mini-1}]
Suppose to the contrary that the claim of Theorem \ref{thm-2mini-1} is false. Then, one sees that: if $1<p<3$, then there exists $\omega_{2} > \omega_{c}^{\mathcal{N}}$ such that $m_{\omega_{2}}^{S}=m_{\omega_{2}}^{\mathcal{N}}$; and if $p=3$, then there exists a sequence $\{\omega_{k}\}$ such that $m_{\omega_{k}}^{S}=m_{\omega_{k}}^{\mathcal{N}}$, and $\lim_{k\to \infty}\omega_{k}=\infty$. When $1<p<3$, we define $\omega_{k}\equiv \omega_{2}$, so that $m_{\omega_{k}}^{S}=m_{\omega_{k}}^{\mathcal{N}}$, and $\liminf_{k\to \infty}\omega_{k}=\omega_{2}>\omega_{c}^{\mathcal{N}}>0$ (see Theorem \ref{thm-0}). 

Now, for each $k$, let $\{u_{n}^{k}\}$ be a minimizing sequence for $m_{\omega_{k}}^{S}$. Notice that $u_{n}^{k}$ is a solution to \eqref{sp} with $\omega=\omega_{k}$. Then, applying Proposition \ref{thm-conv} 
to the sequence $\{u_{n}^{k} \}$ 
and taking a number $n_{k}$ such that 
$\omega_{k} \|u_{n}^{k}\|_{L^{\infty}}^{-4}\le \frac{1}{k}$, 
one can obtain sequences $\{(u_{n},\omega_{n}) \}$ and $\{(x_{n},\theta_{n})\}$ such that: $u_{n}$ is a solution to \eqref{sp} with $\omega=\omega_{n}$; $\liminf_{n\to \infty}\omega_{n}=\omega_{2}>0$ if $1<p<3$, and $\lim_{n\to \infty}\omega_{n}=\infty$ if $p=3$; and putting 
$M_{n}:=\|u_{n}\|_{L^{\infty}}$, $\alpha_{n}:=\omega_{n} M_{n}^{-4}$, and $\widetilde{u}_{n}:=T_{M_{n}}[e^{i\theta_{n}}u_{n}(\cdot+x_{n})]$, we have 
\begin{align}
\label{19/9/16/16:51}
\lim_{n\to \infty}M_{n}&=\infty, 
\quad 
\lim_{n\to \infty}\alpha_{n}=0, 
\\[6pt]
\label{19/9/16/16:49} 
\lim_{n \to \infty} \widetilde{u}_{n} 
&= 
W \quad \text{strongly in $\dot{H}^{1}(\mathbb{R}^{3})$}. 
\end{align}
However, Theorem \ref{non-seq} forbids such a sequence 
$\{(\widetilde{u}_{n}, \omega_{n})\}$ to exist. Thus, the claim of Theorem \ref{thm-2mini-1} must be true. 
\end{proof}

\section{Proof of Theorem 
\ref{ex-critical-o}}
In this section, we shall give a proof of 
Theorem \ref{ex-critical-o}. 
For this purpose, we prepare the following lemma: 
\begin{lemma} \label{conti-m}
Assume $1<p< 5$. 
Then, $m_{\omega}^{\mathcal{N}}$ 
is continuous at $\omega=\omega_{c}^{\mathcal{N}}$, namely 
for any sequence $\{\omega_{n}\}$ in 
$(0, \omega_{c}^{\mathcal{N}})$ with 
$\lim_{n \to \infty} \omega_{n} 
= \omega_{c}^{\mathcal{N}}$, 
we have $\lim_{n \to \infty} m_{\omega_{n}}^{\mathcal{N}} 
= m_{\omega_{c}^{\mathcal{N}}}^{\mathcal{N}}$. 
\end{lemma}

\begin{proof}
By Proposition \ref{thm4-2}, it suffices to show the left-continuity.

Let $\{\omega_{n}\}$ be a sequence in $[\frac{\omega_{c}^{\mathcal{N}}}{2}, \omega_{c}^{\mathcal{N}})$ such that 
$\lim_{n\to \infty}\omega_{n}=\omega_{c}^{\mathcal{N}}$. 
It follows from Proposition \ref{thm4-1} and 
the definition of $\omega_{c}^{\mathcal{N}}$ that 
\begin{equation}\label{21/3/19/14:49}
m_{\omega_{n}}^{\mathcal{N}} < m_{\infty}
\quad
\mbox{for all $n\ge 1$}
.
\end{equation} 
By \eqref{21/3/19/14:49} and Proposition \ref{thm-0-1}, 
we see that 
there exists a positive 
radial ground state $\Phi_{\omega_{n}}$ 
for $m_{\omega_{n}}^{\mathcal{N}}$. 
Note that $\mathcal{S}_{\omega_{n}}(\Phi_{\omega_{n}}) 
= m_{\omega_{n}}^{\mathcal{N}}<m_{\infty}$ and $\mathcal{N}_{\omega_{n}}(\Phi_{\omega_{n}})=0$. 
Here, we may assume that $\frac{1}{2} \omega_{c} 
\le \omega_{n} <\omega_{c}$ for all $n$. 
Then, a computation similar to \eqref{19/9/6/16:11} 
shows that there exists $C_{1}>0$ 
depending only on $p$ such that for any $n$, 
\begin{equation} \label{upper-b-p-1}
\|\Phi_{\omega_{n}}\|_{H^{1}}^{2} \leq C_{1}, 
\end{equation}
Here, suppose to the contrary that 
there existed a subsequence 
of $\{\omega_{n}\}$ 
(still denoted by the same symbol) such that 
$\lim_{n \to \infty} 
m_{\omega_{n}}^{\mathcal{N}} 
< m_{\infty}$. 
Put 
$\mu_{c}^{\mathcal{N}}:= m_{\infty} - 
\lim_{n \to \infty} 
m_{\omega_{n}}^{\mathcal{N}} >0$, 

We see from $\mathcal{N}_{\omega_{n}}
(\Phi_{\omega_{n}})=0$ and 
\eqref{upper-b-p-1} that 
\begin{equation*}
\lim_{n\to \infty} \mathcal{N}_{\omega_{c}
^{\mathcal{N}}}(\Phi_{\omega_{n}}) 
= 
\lim_{n\to \infty} 
(\omega_{n} - \omega_{c}^{\mathcal{N}}) 
\|\Phi_{\omega_{n}}\|_{L^{2}}^{2} 
= 0. 
\end{equation*} 
This implies that there exists a sequence $\{t_{n}\}$ in $(0,\infty)$ such that 
\begin{align} 
\label{18/07/01/11:47-1}
& \mathcal{N}_{\omega_{c}^{\mathcal{N}}}
(t_{n}\Phi_{\omega_{n}}) = 0
\quad 
\mbox{for all $n\ge 1$}, 
\\[6pt]
\label{18/07/01/11:48-1} 
&\lim_{n \to \infty}t_{n} = 1.
\end{align} 
From \eqref{upper-b-p-1} and 
\eqref{18/07/01/11:47-1}, 
we can take a number $n_{c}$ such that for any $n\ge n_{c}$, 
\begin{align}
& 
\big| 
\mathcal{S}_{\omega_{n}} 
(t_{n}\Phi_{\omega_{n}}) - 
\mathcal{S}_{\omega_{c}^{\mathcal{N}}}
(t_{n}\Phi_{\omega_{n}}) 
\big|
= 
\frac{t_{n}^{2}}{2}|\omega_{c}^{\mathcal{N}} - \omega_{n}|
\|\Phi_{\omega_{n}}\|_{L^{2}}^{2}< 
\frac{1}{4}\mu_{c}^{\mathcal{N}}, 
\label{18/07/01/11:51-1} \\
& 
\big|
\mathcal{S}_{\omega_{n}} 
(t_{n}\Phi_{\omega_{n}}) 
- \mathcal{S}_{\omega_{n}} 
(\Phi_{\omega_{n}})
\big| 
< \frac{1}{4}\mu_{c}^{\mathcal{N}}. 
\label{18/07/01/11:51-22}
\end{align}
Observe from 
the definition of $m_{\omega}^{\mathcal{N}}$, 
$m_{\omega_{c}}^{\mathcal{N}} = m_{\infty}$,
\eqref{18/07/01/11:51-1}, 
$\mathcal{S}_{\omega_{n}}
(\Phi_{n}) =m_{\omega_{n}}^{\mathcal{N}}$ and 
\eqref{18/07/01/11:51-22} that
for $n \geq n_{c}$, 
\begin{equation}
\label{18/07/01/11:59-1}
\begin{split}
m_{\infty} \leq 
\mathcal{S}_{\omega_{c}^{\mathcal{N}}} 
(t_{n}\Phi_{\omega_{n}})
& \leq \mathcal{S}_{\omega_{n}} 
(\Phi_{\omega_{n}}) 
+ |\mathcal{S}_{\omega_{n}} 
(t_{n}\Phi_{\omega_{n}}) 
- \mathcal{S}_{\omega_{n}} 
(\Phi_{\omega_{n}})| \\
& \quad 
+ |\mathcal{S}_{\omega_{n}} 
(t_{n}\Phi_{\omega_{n}}) - 
\mathcal{S}_{\omega_{c}^{\mathcal{N}}} 
(t_{n}\Phi_{\omega_{n}})| \\
&< 
m_{\omega_{n}}^{\mathcal{N}} 
+ \frac{1}{2}\mu_{c}^{\mathcal{N}}.
\end{split}
\end{equation}
Since $\mu_{c}^{\mathcal{N}} = m_{\infty} - 
\lim_{n \to \infty}m_{\omega_{n}}^{\mathcal{N}}$, 
we have 
\begin{equation}
m_{\infty} 
\leq \lim_{n \to \infty} 
m_{\omega_{n}}^{\mathcal{N}}
+ \frac{1}{2}\mu_{c}^{\mathcal{N}}
= m_{\infty} - \frac{1}{2}\mu_{c}^{\mathcal{N}}. 
\end{equation}
However, this is a contradiction. 
Thus, we have completed the proof. 
\end{proof}

We are now in a position to prove 
Theorem \ref{ex-critical-o}. 
\begin{proof}[Proof of Theorem \ref{ex-critical-o}]
Suppose to the contrary that 
there is no ground state for 
$m_{\omega_{c}}^{\mathcal{N}}$. 
Let $\{(\Phi_{\omega_{n}}, \omega_{n})\}$ 
be a sequence in $H^{1}(\mathbb{R}^{3}) \times (0, \omega_{c}^{\mathcal{N}})$ 
with the following properties:
\begin{enumerate}
\item[{\rm (i)}]~$\Phi_{n}$ is a positive 
radial ground state for $m_{\omega_{n}}^{\mathcal{N}}$, 
\item[{\rm (ii)}]~ 
$\lim_{n\to \infty}\omega_{n}= 
\omega_{c}^{\mathcal{N}} > 0$. 
\end{enumerate}
Note that $\mathcal{N}_{\omega_{n}}(\Phi_{\omega_{n}})=0$ 
for all $n \ge 1$.
From Lemma \ref{conti-m} and Proposition \ref{thm4-2}, 
we see that 
\begin{equation} \label{eqc-10}
\lim_{n \to \infty} 
\mathcal{S}_{\omega_{n}}(\Phi_{\omega_{n}}) 
=
\lim_{n \to \infty} m_{\omega_{n}}^{\mathcal{N}} 
= m_{\omega_{c}^{\mathcal{N}}}^{\mathcal{N}} 
= m_{\infty}. 
\end{equation}
Furthermore, define $M_{n}:=\|\Phi_{n}\|_{L^{\infty}}$ and 
$\widetilde{\Phi}_{n}:=T_{M_{n}} \Phi_{n} 
=M_{n}^{-1}\Phi_{n}(M_{n}^{-2}\cdot)$. 
We emphasize that \eqref{eqc-10} corresponds 
to \eqref{19/9/6/15:35}. 
Furthermore, the hypothesis that 
there is no ground state for $m_{\omega_{c}^{\mathcal{N}}}$ 
leads us to $\lim_{n\to \infty}
\|\Phi_{\omega_{n}}\|_{L^{p+1}}=0$, 
as well as \eqref{lower-bound}. 
Then, we can verify that 
$\widetilde{\Phi}_{n}$ 
obeys 
\begin{equation}
- \Delta \widetilde{\Phi}_{n} 
+ \alpha_{n} \widetilde{\Phi}_{n}
- \beta_{n} |\widetilde{\Phi}_{n}|^{p-1}\widetilde{\Phi}_{n} 
- |\widetilde{\Phi}_{n}|^{4}\widetilde{\Phi}_{n} 
= 0 
.
\end{equation}
where 
$\alpha_{n}:=\omega_{n}M_{n}^{- 4}$, 
and $\beta_{n}:=M_{n}^{p - 5}$. 

\par 
Now, we suppose for contradiction that 
there is no ground state (no minimizer) 
for $m_{\omega_{c}^{\mathcal{N}}}^{\mathcal{N}}$.
Then, by a similar argument to the proof of 
Proposition \ref{thm-conv}, we can verify the 
following: 
\begin{enumerate}
\item[{\rm (iii)}]~$\lim_{n \to \infty} M_{n}= \infty$,
\quad 
$\lim_{n \to \infty}\alpha_{n}=0$, 
\item[{\rm (iv)}]~$\lim_{n \to \infty}\widetilde{\Phi}_{n} = W$
\quad strongly in 
$\dot{H}^{1}(\mathbb{R}^{3})$.
\end{enumerate}

However, Theorem \ref{non-seq} 
does not allow such a sequence to exist. Thus, there exists a ground state for 
$m_{\omega_{c}^{\mathcal{N}}}^{\mathcal{N}}$. 
\end{proof}


\section{Proof of Theorem \ref{thm-1}}\label{sec-7}
In this section, we give a proof of Theorem \ref{thm-1}. 
The following lemma follows from the facts that 
the definition of $\mathcal{K}$ does not 
include $\omega$, and $\mathcal{S}_{\omega_{1}} < \mathcal{S}_{\omega_{2}}$ for all $\omega_{1}<\omega_{2}$: 
\begin{lemma} \label{monotonicity-k}
Assume $d \geq 3$ and 
$1+\frac{4}{d} < p < \frac{d+2}{d-2}$. Then, the value $m_{\omega}^{\mathcal{K}}$ is nondecreasing with respect to $\omega$ on $(0, \infty)$. 
\end{lemma}
First, notice that by the definition of $\omega_{c}^{\mathcal{K}}$ (see \eqref{threshold-2}) and Lemma \ref{monotonicity-k}, 

Next, recall the following result obtained in \cite{AIKN}: 
\begin{proposition}[Proposition 2.1 in \cite{AIKN}]\label{thm-0-2}
Assume $d \geq 3$, 
$1+\frac{4}{d}<p<\frac{d+2}{d-2}$ 
and $\omega > 0$. 
If $m_{\omega}^{\mathcal{K}} < m_{\infty}$, 
then a minimizer for $m_{\omega}^{\mathcal{K}}$ exists. 
\end{proposition}
Now, observe from Lemma \ref{monotonicity-k} that 
the set 
$\{\omega>0 \colon m_{\omega}^{\mathcal{K}} <m_{\infty}\}$ 
is connected, so that 
\begin{equation}\label{19/9/16/21:41}
m_{\omega}^{\mathcal{K}} < m_{\infty}
\quad 
\mbox{for any $0< \omega < 
\omega_{c}^{\mathcal{K}}$}.
\end{equation}

Moreover, we can prove the following proposition by the same argument as in Section 2 of \cite{AKN}:
\begin{proposition}\label{19/9/16/21:16}
Assume $d \geq 3$ and 
$1+\frac{4}{d}<p<\frac{d+2}{d-2}$, 
and let $\omega>0$. Then, any minimizer for $m_{\omega}^{\mathcal{K}}$ becomes a solution to \eqref{sp}. 
\end{proposition}

\begin{proposition}\label{super-omega-K}
Assume $d = 3$ 
and $\frac{7}{3} <p < 5$. 
Then, for any $\omega \ge \omega_{c}^{\mathcal{K}}$, we have
$m_{\omega}^{\mathcal{K}} = m_{\infty}$. 
\end{proposition}
\begin{proof}
First, consider the case where $\omega>\omega_{c}^{\mathcal{K}}$. We see from Lemma \ref{monotonicity-k} 
and the definition of $\omega_{c}^{\mathcal{K}}$ 
(see \eqref{threshold-2}).
that $m_{\omega}^{\mathcal{K}} 
\ge m_{\omega_{c}^{\mathcal{K}}}^{\mathcal{K}} 
\ge m_{\infty}$ for any $\omega >\omega_{c}^{\mathcal{K}}$. 
Here, suppose to the contrary that there exists $\omega_{2}>\omega_{c}^{\mathcal{K}}$ such that $m_{\omega_{2}}^{\mathcal{K}}> m_{\infty}$. Put $\mu_{2}:=m_{\omega_{2}}^{\mathcal{K}}- m_{\infty} (>0)$. Then, as well as \eqref{18/07/01/10:37} in the proof of Proposition \ref{thm4-2}, we can verify that for any $\varepsilon>0$ and any $t > 0$, 
\begin{equation}\label{18/07/01/10:37-1}
t^{-2} \mathcal{K}(t V_{\varepsilon}) 
= 
\sigma^{\frac{3}{2}} + O(\varepsilon^{2}) 
- 
t^{p-1} O(\varepsilon^{5-p})
- 
t^{4} \sigma^{\frac{3}{2}}
+ 
t^{4} O(\varepsilon^{6}),
\end{equation}
where $V_{\varepsilon}$ is the same function as in \eqref{modi-talenti}. Furthermore, we find from \eqref{18/07/01/10:37-1} that for any sufficiently small $\varepsilon>0$, there exists $t_{\varepsilon} > 0$ such that 
\begin{equation}\label{18/07/01/10:41-2}
\mathcal{K}(t_{\varepsilon} V_{\varepsilon}) = 0, \qquad 
t_{\varepsilon} = 1 + O(\varepsilon^{2}).
\end{equation}
Hence, we see from \eqref{18/07/01/10:41-2}, the definition of $m_{\omega_{2}}^{\mathcal{K}}$ (see \eqref{mini-k}), \eqref{est-Talenti1}, \eqref{est-Talenti2}, $p \leq 2$, 
and the hypothesis $m_{\omega_{2}}^{\mathcal{K}}>m_{\infty}$ 
that if $\varepsilon^{2}\ll \mu_{2}$, then
one has 
\begin{equation}
\begin{split}
m_{\omega_{2}}^{\mathcal{K}}
& \le \mathcal{\mathcal{S}}_{\omega_{2}}
(t_{\varepsilon} V_{\varepsilon}) 
=
\mathcal{\mathcal{S}}_{\omega_{2}}
(t_{\varepsilon} V_{\varepsilon}) 
- \frac{1}{2} \mathcal{K}(t_{\varepsilon} V_{\varepsilon}) \\
& = 
\frac{\omega_{2}}{2} 
t_{\varepsilon}^{2}\|V_{\varepsilon}\|_{L^{2}}^{2} 
+ \left(\frac{3(p-1)}{4(p+1)} - \frac{1}{p+1}\right) 
t_{\varepsilon}^{p+1} \|V_{\varepsilon}\|_{L^{p+1}}^{p+1} 
+ \frac{1}{3} t_{\varepsilon}^{6} \|V_{\varepsilon}\|_{L^{6}}^{6} \\
& = 
t_{\varepsilon}^{2} \, O(\varepsilon^{2}) 
+ 
t_{\varepsilon}^{p+1}
O(\varepsilon^{5-p}) 
+ t_{\varepsilon}^{6} m_{\infty} 
+ O(\varepsilon^{6})
\\[6pt]
&= 
m_{\infty} + O(\varepsilon^{2})
\le 
m_{\infty} + \frac{1}{2}\mu_{2}
=
\frac{1}{2}
m_{\omega_{2}}^{\mathcal{K}}
+
\frac{1}{2}m_{\infty}
< 
m_{\omega_{2}}^{\mathcal{K}}, 
\end{split}
\end{equation}
which is a contradiction. 
Thus, we have proved that 
$m_{\omega}^{\mathcal{K}} = m_{\infty}$ for all 
$\omega > \omega_{c}^{\mathcal{N}}$. 
Next, we consider the case where $\omega = \omega_{c}^{\mathcal{K}}$. 
We see from Lemma \ref{monotonicity-k} and 
the definition of $\omega_{c}^{\mathcal{K}}$ 
that $m_{\omega_{c}^{\mathcal{K}}}^{\mathcal{K}} \le m_{\infty}$. 
Suppose to the contrary that 
$m_{\omega_{c}^{\mathcal{K}}}
^{\mathcal{K}} < m_{\infty}$. Then, 
applying Propositions \ref{thm-0-2} and 
\ref{19/9/16/21:16}, we find that a minimizer 
$\Phi_{\omega_{c}}^{\mathcal{K}}$ 
for $m_{\omega_{c}^{\mathcal{K}}}^{\mathcal{K}}$ exists 
and satisfies \eqref{sp}. 
Note that $\mathcal{S}_{\omega_{c}^{\mathcal{K}}}
(\Phi_{\omega_{c}^{\mathcal{K}}}) = 
m_{\omega_{c}^{\mathcal{K}}}^{\mathcal{K}}$ 
and 
$\mathcal{N}_{\omega_{c}^{\mathcal{K}}}
(\Phi_{\omega_{c}^{\mathcal{K}}}) 
= \langle \mathcal{S}_{\omega_{c}^{\mathcal{K}}}^{\prime}
(\Phi_{\omega_{c}^{\mathcal{K}}}), 
\Phi_{\omega_{c}^{\mathcal{K}}} \rangle =0$. 
Put $\mu_{c}^{\mathcal{K}} := m_{\infty} 
- m_{\omega_{c}^{\mathcal{K}}}^{\mathcal{K}}$. 
Moreover, let $\{\omega_{n}\}$ be a sequence of positive numbers such that 
$\lim_{n \to \infty}\omega_{n} = \omega_{c}^{\mathcal{K}}$, and $\omega_{n} > \omega_{c}^{\mathcal{K}}$ for all $n \ge 1$. 
We have already proved that for any $n \ge 1$, 
\begin{equation}\label{18/07/01/11:45-2}
m_{\omega_{n}}^{\mathcal{K}} = m_{\infty}.
\end{equation} 
Furthermore, we can take a number $n_{c}$ 
such that for any $n\ge n_{c}$, 
\begin{align}
\label{18/07/01/11:51-2}
\big| 
\mathcal{S}_{\omega_{n}} (\Phi_{\omega_{c}^{\mathcal{K}}}) - \mathcal{S}_{\omega_{c}^{\mathcal{K}}}(\Phi_{\omega_{c}^{\mathcal{K}}}) 
\big|
= 
\frac{1}{2}|\omega_{c}^{\mathcal{K}} - \omega_{n}|
\|\Phi_{\omega_{c}^{\mathcal{K}}}\|_{L^{2}}^{2}< 
\frac{1}{2}\mu_{c}^{\mathcal{K}},
\end{align}
Using \eqref{18/07/01/11:45-2}, 
the definition of $m_{\omega}^{\mathcal{K}}$, 
$\mathcal{K} (\Phi_{\omega_{c}^{\mathcal{K}}})=0$, 
\eqref{18/07/01/11:51-2}
and $\mu_{c}^{\mathcal{K}} 
= m_{\infty} - 
m_{\omega_{c}^{\mathcal{K}}} 
= m_{\infty} - 
\mathcal{S}_{\omega_{c}^{\mathcal{K}}}
(\Phi_{\omega_{c}^{\mathcal{K}}}) $, 
we find that for any $n \ge n_{c}$, 
\begin{equation}
m_{\infty} 
= 
m_{\omega_{n}}^{\mathcal{K}} 
\leq \mathcal{S}_{\omega_{n}} (\Phi_{\omega_{c}^{\mathcal{K}}})
< 
\mathcal{S}_{\omega_{c}^{\mathcal{K}}}(\Phi_{\omega_{c}^{\mathcal{K}}}) 
+ \frac{1}{2}\mu_{c}^{\mathcal{K}} 
= m_{\infty} - \frac{1}{2}\mu_{c}^{\mathcal{K}}. 
\end{equation}
However, this is a contradiction. 
Thus, we have completed the proof. 
\end{proof}

Now, we are in a position to prove Theorem \ref{thm-1}: 
\begin{proof}[Proof of Theorem \ref{thm-1}]
We shall prove the claim {\rm (i)}. 

We shall show that $\omega_{c}^{\mathcal{K}} 
\le \omega_{c}^{\mathcal{N}}$. 
Suppose to the contrary that $\omega_{c}^{\mathcal{K}} > \omega_{c}^{\mathcal{N}}$. Then, we can take $\omega_{2}$ such that $\omega_{c}^{\mathcal{N}} <\omega_{2} < \omega_{c}^{\mathcal{K}}$, and Proposition \ref{thm-0-2} together with \eqref{19/9/16/21:41} shows that a minimizer $Q$ for $m_{\omega_{2}}^{\mathcal{K}}$ exists. 
Note that $\mathcal{S}_{\omega_{2}}(Q)=m_{\omega_{2}}^{\mathcal{K}}$ 
and $\mathcal{K}(Q) = 0$. 
Furthermore, it follows from Proposition \ref{19/9/16/21:16} that 
$\mathcal{N}_{\omega_{2}}(Q) 
= \langle 
\mathcal{S}_{\omega}^{\prime}(Q), Q \rangle =0$. 

Now, we find from $\omega_{c}^{\mathcal{N}}<\omega_{2}$, Proposition 
\ref{thm4-2}, $\mathcal{N}_{\omega_{2}}(Q)=0$, the definition of 
$m_{\omega}^{\mathcal{N}}$ (see \eqref{m-value}), $\mathcal{S}
_{\omega_{2}}(Q)=m_{\omega_{2}}^{\mathcal{K}}$, $\omega_{2} < 
\omega_{c}^{\mathcal{K}}$ and \eqref{19/9/16/21:41} that 
\begin{equation}
m_{\infty} = m_{\omega_{2}}^{\mathcal{N}} 
\leq \mathcal{S}_{\omega_{2}}(Q) 
= m_{\omega_{2}}^{\mathcal{K}} < 
m_{\infty}, 
\end{equation}
which is a contradiction. Thus, we have proved the claim.

Suppose to the contrary that $\omega_{c}^{\mathcal{K}} 
< \omega_{c}^{\mathcal{N}}$. 
Let $\omega \in (\omega_{c}^{\mathcal{K}}, 
\omega_{c}^{\mathcal{N}})$. 
Then, we see from Proposition \ref{thm-0-1} that 
there exists a positive radial 
ground state $\Phi_{\omega}^{\mathcal{N}}$ of 
$m_{\omega}^{\mathcal{N}}$. 
Then, we can prove in a way similar to Lemma 2 of \cite{Berestycki-Lions} that 
$\Phi_{\omega}^{\mathcal{N}}$ and the first and the 
second derivatives decays exponentially. In 
particular, we have $x \cdot \nabla 
\Phi_{\omega}^{\mathcal{N}} \in H^{1}
(\mathbb{R}^{3})$. 
Hence, we see that 
\begin{equation} \label{identity-K}
\mathcal{K}(\Phi_{\omega}^{\mathcal{N}}) 
= \langle \mathcal{S}_{\omega}^{\prime}
(\Phi_{\omega}^{\mathcal{N}}), 
\frac{3}{2} \Phi_{\omega}^{\mathcal{N}} 
+ x \cdot 
\nabla \Phi_{\omega}^{\mathcal{N}} \rangle = 0.
\end{equation} 
It follows from the definition of $m_{\omega}^{\mathcal{K}}$, $\omega > \omega_{c}^{\mathcal{K}}$ and Proposition \ref{super-omega-K} that 
\begin{equation} \label{vari-value-ineq}
m_{\infty} = m_{\omega}^{\mathcal{K}} 
\leq \mathcal{S}_{\omega} (\Phi_{\omega}^{\mathcal{N}}). 
\end{equation}
On the other hand, we see from 
$\omega < \omega_{c}^{\mathcal{N}}$ 
and the definition of $\omega_{c}^{\mathcal{N}}$ that 
\begin{equation}
\mathcal{S}_{\omega} (\Phi_{\omega}^{\mathcal{N}}) < m_{\infty},
\end{equation}
which contradicts \eqref{vari-value-ineq}. 
Thus, we have $\omega_{c}^{\mathcal{N}} = \omega_{c}^{\mathcal{K}}$. 

Next, we shall prove the claim {\rm (ii)}. 
Put $\omega_{c} := \omega_{c}^{\mathcal{N}} 
= \omega_{c}^{\mathcal{K}}$, and let $\omega < \omega_{c}$. 
Then, it follows from Propositions \ref{thm-0-1} 
and \ref{thm-0-2} that there exist a minimizer 
$\Phi_{\omega}^{\mathcal{N}}$ of $m_{\omega}^{\mathcal{N}}$ and 
a minimizer 
$\Phi_{\omega}^{\mathcal{K}}$ of $m_{\omega}^{\mathcal{K}}$. 
Note that 
$\mathcal{N}_{\omega}(\Phi_{\omega}^{\mathcal{K}}) = 0$ 
and $\mathcal{K}(\Phi_{\omega}^{\mathcal{N}})=0$ 
(see \eqref{identity-K}). 
Hence, we have 
\begin{equation} \label{vari-ineq-2}
m_{\omega}^{\mathcal{N}} \leq \mathcal{S}_{\omega}(\Phi_{\omega}^{\mathcal{K}}) 
= m_{\omega}^{\mathcal{K}}, 
\end{equation}
and 
\begin{equation} \label{vari-ineq-3}
m_{\omega}^{\mathcal{K}} \leq \mathcal{S}_{\omega}(\Phi_{\omega}^{\mathcal{N}}) 
= m_{\omega}^{\mathcal{N}}. 
\end{equation}

Note that $\omega_{c}=\infty$ for $3<p<5$.
Hence, we may assume that 
$\frac{7}{3}<p\le 3$. 
By \eqref{vari-ineq-2} and \eqref{vari-ineq-3}, 
we have $m_{\omega}^{\mathcal{N}} = m_{\omega}^{\mathcal{K}}$ for any 
$\omega < \omega_{c}$. 
Then, for any $\omega \geq \omega_{c}$, it follows from 
Propositions \ref{thm4-2} and \ref{super-omega-K} that 
\begin{equation}
m_{\omega}^{\mathcal{N}} = m_{\omega}^{\mathcal{K}} = 
m_{\infty}. 
\end{equation}
Thus, we obtain $m_{\omega}^{\mathcal{N}} 
= m_{\omega}^{\mathcal{K}}$ 
for any $\omega > 0$. 

Next, we shall prove the claim {\rm (iii)}. From Proposition \ref{thm-0-2}, 
there exists a minimizer $\Phi_{\omega}^{\mathcal{K}}$ 
of $m_{\omega}^{\mathcal{K}}$
for $\omega \in (0, \omega_{c})$. 
We shall show that 
there exists no minimizer of $m_{\omega}^{\mathcal{K}}$ 
for $\omega \in 
(\omega_{c}, \infty)$. 
Suppose to the contrary that 
there exists a $\omega_{*} \in 
(\omega_{c}^{\mathcal{K}}, \infty)$ such that 
$m_{\omega_{*}}^{\mathcal{K}}$ has a minimizer 
$\Phi_{\omega_{*}}^{\mathcal{K}}$. 
By $\omega_{*} > \omega_{c}^{\mathcal{K}}, 
\Phi_{\omega_{*}}^{\mathcal{K}} \neq 0$, 
Lemma \ref{monotonicity-k} and
Proposition \ref{super-omega-K}, 
we have 
\begin{equation}
m_{\infty} 
= \mathcal{S}_{\omega_{*}}(\Phi_{\omega_{*}}^{\mathcal{K}}) 
> \mathcal{S}_{\omega_{c}}(\Phi_{\omega_{*}}^{\mathcal{K}})
\geq m_{\omega_{c}^{\mathcal{K}}}^{\mathcal{K}} 
= m_{\infty}, 
\end{equation} 
which is a contradiction. 

Finally, we will give a proof of (\textrm{iv}). 
Assume $\frac{7}{3}<p<3$.
It follows from Theorem \ref{ex-critical-o} 
and $m_{\omega}^{\mathcal{N}}=m_{\omega}^{\mathcal{K}}$ (see Claim (ii)) that 
there exists a minimizer $\Phi_{\omega_{c}}
^{\mathcal{N}}$ such that 
$\mathcal{S}_{\omega}(\Phi_{\omega_{c}}
^{\mathcal{N}}) = m_{\infty}$. 
Since $\Phi_{\omega_{c}}^{\mathcal{N}}$ 
satisfies $\mathcal{K}(\Phi_{\omega_{c}}
^{\mathcal{N}}) = 0$, one has, by Proposition 
\ref{super-omega-K}, that 
\begin{equation}
m_{\infty} = m_{\omega_{c}}^{\mathcal{K}} 
\leq \mathcal{S}_{\omega_{c}}(\Phi_{\omega_{c}}
^{\mathcal{N}}) = m_{\infty}. 
\end{equation}
Thus, we see that $\Phi_{\omega_{c}}
^{\mathcal{N}}$ is also a minimizer for 
$m_{\omega_{c}}^{\mathcal{K}}$. 
This completes the proof. 
\end{proof}


\section{Proofs of Theorems \ref{thm-ufe} and \ref{thm-imp}}\label{19/9/24/6:1}
\label{18/11/11/10:47}

In this section, we give a proof of Theorems \ref{thm-ufe} 
and \ref{thm-imp}.
\subsection{Preliminaries} 
First, let us recall the information about the spectrum of $L_{+}$ (see, e.g., Section 5.2 of \cite{DM}, and Proposition 2.2 of \cite{CMR}): 
\begin{lemma}\label{18/11/14/10:43}
Assume $d\ge 3$. Then, there exists $e_{0}>0$ such that $-e_{0}$ is the only one negative eigenvalue of the operator $L_{+}:=-\Delta +V_{+}$. Furthermore, the spectrum $\sigma(L_{+})$ of $L_{+}$ is contained in $\{-e_{0} \}\cup [0,\infty)$, and 
\begin{equation}\label{18/11/11/11:01}
\{ u \in \dot{H}^{1}(\mathbb{R}^{d}) \colon \mbox{$L_{+}u =0$ in the distribution sense}\} =\operatorname{span}_{\mathbb{C}}
\{\Lambda W, \partial_{1}W, \ldots, \partial_{d}W \},
\end{equation}
where $\partial_{j}W$ means the partial derivative of $W$ with respect to the $j$-th variable for each $1\le j \le d$. 
\end{lemma} 

We can also derive the following information about the spectrum of $L_{-}$:
\begin{lemma}\label{Lem-ker-L-}
Assume $d\ge 3$. Then, the spectrum $\sigma(L_{-})$ of $L_{-}$ equals $[0,\infty)$, and 
\begin{equation}\label{ker-L-}
\begin{split}
& 
\{ u \in \dot{H}^{1}(\mathbb{R}^{d}) \colon \mbox{$L_{-}u =0$ in the distribution sense}\} 
\\ 
& 
=\operatorname{span}_{\mathbb{C}}\{
W \colon 
(y, \nu, \theta) \in \mathbb{R}^{d} \times 
(0, \infty) \times [-\pi, \pi) \}.
\end{split}
\end{equation}
\end{lemma}
\begin{proof}[Proof of Lemma \ref{Lem-ker-L-}]
It is well known that the essential spectrum of $L_{-}$ equals $[0,\infty)$. 
Let $u\in H^{2}(\mathbb{R}^{3})$ be a function such that $L_{-}u=-\mu u $ for some $\mu > 0$. 
Then, it follows from H\"older's inequality, \eqref{19/9/18/11:15} and 
Sobolev's inequality \eqref{19/9/7/15:31} 
that 
\begin{equation}\label{19/9/24/7:5}
\begin{split}
\| \nabla u \|_{L^{2}}^{2}
&=
\langle -\Delta u, u \rangle
=
\langle -V_{-} u, u \rangle
-
\langle \mu u, u \rangle
\\[6pt]
&\le 
\|W^{4}\|_{L^{\frac{3}{2}}}
\|u^{2}\|_{L^{\frac{d}{d-2}}}
-
\mu \|u\|_{L^{2}}^{2}
=
\sigma \|u\|_{L^{6}}^{2}
-
\mu \|u\|_{L^{2}}^{2}
< \| \nabla u \|_{L^{2}}^{2}, 
\end{split} 
\end{equation}
which is a contradiction. 
In particular, there is no eigenvalue of $L_{-}$ 
for $\mu > 0$. 
When $\mu=0$, \eqref{19/9/24/7:5} means 
that $u$ attains the equality 
in Sobolev's inequality \eqref{19/9/7/15:31}, 
which together with \eqref{scale-W} and $L_{-} u = 0$ 
implies that 
$u$ must be of the form in \eqref{scale-W}, 
that is, $u = z W$ for some $z \in \mathbb{C}$. 
\end{proof}

For $r\ge 1$, we define 
\begin{equation}\label{18/11/06/13:57}
X_{r}:=
\bigm\{
f\in L^{r}(\mathbb{R}^{3}) \colon 
\langle f,\, V_{+} \nabla W \rangle 
=0,
\ 
\langle f,\, V_{+} \Lambda W \rangle
=
0,
\ 
\langle f,\, i V_{-} W \rangle
=0
\bigm\}
.
\end{equation}
Moreover, we define the operators $G_{0}$ and $G$ by 
\begin{align}
\label{19/10/7/10:51}
G_{0} f 
&:= 
(-\Delta)^{-1}V_{+}\Re[f]
+ 
i (-\Delta)^{-1}V_{-} \Im[f],
\\[6pt]
\label{19/10/6/17:1}
G f 
&:= (1+G_{0})f
=
\big\{ 1+(-\Delta)^{-1}V_{+} \big\} \Re[f]
+ 
i \big\{ 1+ (-\Delta)^{-1}V_{-} \big\} \Im[f]
.
\end{align}
Notice that by \eqref{first-res}, 
\begin{equation}\label{19/01/27/18:15}
G(\alpha)
=
G -\alpha (-\Delta+\alpha)^{-1}G_{0}
,
\end{equation}
where $G(\alpha)$ is the operator defined by \eqref{def-G}.

It follows from the Hardy-Littlewood-
Sobolev inequality (\eqref{HLS-ineq} 
with $s=2$) that for $3 < r < \infty$ 
and any $f \in L^{r}(\mathbb{R}^{3})$,
\begin{equation}\label{18/11/13/14:47}
\begin{split}
\|G_{0}f\|_{L^{r}}
&\lesssim 
\|V_{+} \Re[f]
\|_{L^{\frac{3r}{2r+3}}}
+
\|V_{+} \Im[f]\|_{L^{\frac{3r}
{2r+3}}}
\\[6pt]
&\le 
\| V_{+} \|_{L^{\frac{3}{2}}} \|\Re[f]
\|_{L^{r}}
+
\| V_{-} \|_{L^{\frac{3}{2}}} \|\Im[f]\|
_{L^{r}}
\lesssim 
\|f\|_{L^{r}}, 
\end{split} 
\end{equation}
where the implicit constants depend only on $r$. 
Hence, or any $3 < r<\infty$, 
$G_{0}$ is a bounded operator from $L^{r}(\mathbb{R}^{3})$ to itself. Furthermore, we have:
\begin{lemma}\label{19/10/8/9:52}
Let $3 < r <\infty$. 
Then, $G_{0}$ maps $X_{r}$ into itself.
\end{lemma}
\begin{proof}[Proof of Lemma \ref{19/10/8/9:52}]
Let $f \in X_{r}$. By \eqref{18/11/13/14:47}, it suffices to show that $G_{0}f$ satisfies the orthogonality conditions of $X_{r}$. Since $V_{+}\partial_{j}W = \Delta \partial_{j}W$ ($1\le j \le 3$) 
(see Lemma \ref{18/11/14/10:43})
and $\langle f ,\, V_{+}\partial_{j} W\rangle=0$, one sees that 
\begin{equation}\label{18/11/13/14:35}
\langle G_{0} f ,\, V_{+} \partial_{j} W\rangle
=
\langle (-\Delta)^{-1}V_{+} \Re[f] ,\, V_{+} \partial_{j} W\rangle
=
\langle V_{+} f ,\, \partial_{j} W\rangle
=0
.
\end{equation}
Similarly, one can verify that 
\begin{equation}\label{19/9/19/15:31}
\begin{split} 
\langle G_{0} f ,\, V_{+} \Lambda W\rangle
=
\langle G_{0} f ,\, i V_{-} W\rangle
= 0.
\end{split} 
\end{equation}
Thus, we have proved that $G_{0}f\in X_{r}$. 
\end{proof}

\begin{lemma}
\label{18/11/14/11:15}
For any $6 < r < \infty$, 
$G_{0}$ is a compact operator from 
$X_{r}$ to itself. 
\end{lemma}
\begin{proof}[Proof of Lemma 
\ref{18/11/14/11:15}]
It follows from \eqref{delta0-inv}, 
$|V_{+}(x)|\sim |V_{-}(x)| \sim (1+|x|)^{- 4}$ and the Hardy-Littlewood-Sobolev inequality 
(\eqref{HLS-ineq} with $s = 1$) that for 
any $f \in L^{r}(\mathbb{R}^{3})$, 
\begin{align}\label{18/11/15/16:27}
\begin{split}
\|\nabla G_{0} f\|_{L^{r}}
&\le 
\|\nabla (-\Delta)^{-1} V_{+} \Re[f]\|_{L^{r}}
+
\|\nabla (-\Delta)^{-1} V_{-} \Im[f] \|_{L^{r}}
\\[6pt]
&\lesssim 
\| \int_{\mathbb{R}^{3}} |x-y|^{-2} 
(1+|y|)^{-4} |f(y)|\,dy \|_{L^{r}}
\\[6pt]
&\lesssim 
\| (1+|y|)^{-4} f\|_{L^{\frac{3r}
{r+3}}}
\le 
\| (1+|y|)^{-4} \|_{L^{3}} \|f\|_{L^{r}}
\lesssim 
\|f\|_{L^{r}}, 
\end{split} 
\end{align}
where the implicit constants depend only on $r$. 
Let $\{f_{n}\}$ be a bounded sequence in $X_{r}$. Then, \eqref{18/11/13/14:47} shows that $\{ G_{0} f_{n}\}$ is also a bounded sequence in $X_{r}$. Moreover, by the fundamental theorem of calculus and \eqref{18/11/15/16:27}, one can see that for any $y\in \mathbb{R}^{3}$,
\begin{equation}\label{18/11/15/12:01}
\begin{split}
&\| G_{0} f_{n} - G_{0} f_{n}(\cdot+y) \|_{L^{r}}
\\[6pt]
&\le 
\| (-\Delta)^{-1} V_{+} \Re[f_{n}] - (-\Delta)^{-1} V_{+} 
\Re[f_{n}](\cdot+y) \|_{L^{r}}
\\[6pt]
&\quad +
\| (-\Delta)^{-1} V_{-} \Im[f_{n}] - (-\Delta)^{-1} V_{-} \Im[f_{n}](\cdot+y) \|_{L^{r}}
\\[6pt]
&=
\| 
\int_{0}^{1} \frac{3}{3\theta} (-
\Delta)^{-1} V_{+} 
\Re[f_{n}](\cdot+\theta y) \,d\theta
\|_{L^{r}} 
\\[6pt]
&\quad +
\| 
\int_{0}^{1} \frac{3}{3 \theta} 
(-\Delta)^{-1} V_{-} \Im[f_{n}]
(\cdot+\theta y) \,d\theta
\|_{L^{r}} 
\\[6pt]
&\le 
|y|
\| \nabla (-\Delta)^{-1} V_{+} 
\Re[f_{n}] \|_{L^{r}}
+ |y|
\| \nabla (-\Delta)^{-1} V_{-} \Im[f_{n}] \|_{L^{r}} 
\\[6pt]
&
\lesssim 
|y| \|\nabla G_{0} f_{n}\|_{L^{r}}
\lesssim 
|y| 
\sup_{n\ge 1}\| f_{n} \|_{L^{r}},
\end{split}
\end{equation}
which implies the equicontinuity of $\{G_{0} f_{n}\}$ in $X_{r}$. 

From the point of view of Lemma \ref{18/11/10/13:29}, it suffices to show the tightness of $\{G_{0} f_{n}\}$ in $L^{r}(\mathbb{R}^{3})$. 
Notice that if $|x-y| \le \frac{1}{2}|x|$, 
then $|y|\ge |x|-|x-y|\ge \frac{1}{2}|x|$. 
Hence, we see from \eqref{delta0-inv}, 
the Hardy-Littlewood-Sobolev inequality (\eqref{HLS-ineq} with $s = 2$), 
H\"older's inequality and $|V{+}(x)|\sim |V_{-}(x)| 
\lesssim (1+|x|)^{- 4}$ that for any $R\ge 1$, 
\begin{equation}\label{18/11/15/12:02}
\begin{split}
&\int_{|x|\ge R} 
| G_{0} f_{n}(x) |^{r}\,dx
\\[6pt]
&\lesssim 
\int_{|x|\ge R} 
\Bigm| (-\Delta)^{-1} V_{+} \Re[f_{n}(x)] \Bigm|^{r}\,dx
+ 
\int_{|x|\ge R} 
\Bigm| (-\Delta)^{-1} V_{-}\Im[f_{n}(x)] \Bigm|^{r}\,dx
\\[6pt]
&\lesssim 
\int_{|x|\ge R}
\Big\{
\int_{\frac{1}{2}|x| \le |x-y|} |x-y|^{-1} 
(1+|y|)^{- 4} |f_{n}(y)|\,dy
\Big\}^{r}\,dx 
\\[6pt]
&\quad + 
\int_{|x|\ge R}
\Big\{
\int_{|x-y|\le \frac{1}{2}|x|} |x-y|^{-1} 
(1+|y|)^{- 4} |f_{n}(y)|\,dy
\Big\}^{r}\,dx 
\\[6pt]
&\lesssim 
R^{-\frac{r}{2}}
\bigm\|
\int_{\mathbb{R}^{3}} |x-y|^{- \frac{1}{2}} 
(1+|y|)^{- 4} 
|f_{n}(y)|\,dy
\bigm\|_{L^{r}}^{r} 
\\[6pt]
&\quad + 
R^{-\frac{r}{2}}
\bigm\| 
\int_{\mathbb{R}^{3}} |x-y|^{-1} 
|y|^{\frac{1}{2}}(1+|y|)^{-4} 
|f_{n}(y)| \,dy
\bigm\|_{L^{r}}^{r} 
\\[6pt]
&\lesssim 
R^{-\frac{r}{2}}
\| (1+|x|)^{- 4} 
f_{n} \|_{L^{\frac{6r}{5r+6}}}^{r} 
+ 
R^{-\frac{r}{2}}
\bigm\| 
|x|^{\frac{1}{2}} (1+|x|)^{- 4} |f_{n}| \bigm\|_{L^{\frac{3r}{2r+3}}}^{r}
\\[6pt]
&\le 
R^{-\frac{r}{2}}
\| (1+|x|)^{-4} 
\|_{L^{\frac{6}{5}}}^{r} \|f_{n}\|_{L^{r}}^{r} 
+ 
R^{-\frac{r}{2}}
\| |x|^{\frac{1}{2}}(1+|x|)^{-4} 
\|_{L^{\frac{3}{2}}}^{r}
\|f_{n} \|_{L^{r}}^{r} 
\\[6pt]
&\lesssim 
R^{-\frac{r}{2}}
\|f_{n}\|_{L^{r}}^{r}, 
\end{split}
\end{equation}
where the implicit constants depend only on $r$, so that 
the tightness in $L^{r}(\mathbb{R}^{3})$ holds. Thus, the compactness follows from Lemma \ref{18/11/10/13:29}.
\end{proof}

\begin{lemma}\label{18/11/13/14:25}
Assume $6 <r < \infty$. Then, the inverse of $G$ exists as a bounded operator from $X_{r}$ to itself.
\end{lemma}
\begin{proof}[Proof of Lemma \ref{18/11/13/14:25}]
It follows from Lemma \ref{18/11/14/11:15} that $G_{0}$ is a compact operator from $X_{r}$ to itself. Hence, 
it follows from the Riesz-Schauder theorem (see, e.g., Theorem VI.15 of \cite{RS}) that if $-1$ is not an eigenvalue of 
the operator $G_{0}: X_{r} \to X_{r}$, 
then the inverse $G^{-1}=(1+G_{0})^{-1}$ exists as a bounded operator from $X_{r}$ to itself, namely $-1 \in \rho(G_{0})$. 
Suppose that $u \in X_{r}$ satisfies $G u =0$. Then, for any $\phi \in C_{c}^{\infty}(\mathbb{R}^{3})$, 
\begin{equation}\label{18/12/10/21:27}
\langle H u, \, \phi \rangle 
= 
\langle (-\Delta) G u, \, \phi \rangle
=
\langle G u, \, -\Delta \phi \rangle
=0. 
\end{equation}
Hence, it follows from Lemma \ref{18/11/14/10:43}, 
Lemma \ref{Lem-ker-L-} and \eqref{18/12/10/21:27} that 
\begin{equation}\label{19/10/7/17:22}
\Re[u]=
\vec{a} \cdot \nabla W + b \Lambda W,
\quad 
\Im[u]=c W 
\end{equation} 
for some $\vec{a}\in \mathbb{R}^{3}$ and $b,c \in \mathbb{R}$. 
Observe that 
$\langle \nabla W, V_{+}\Lambda W \rangle 
=\langle \Lambda W, V_{+}\nabla W=0$. 
Then, \eqref{19/10/7/17:22} together 
with $u\in X_{r}$ implies that 
\begin{equation}\label{19/9/19/14:50}
\langle \vec{a} \cdot \nabla W + b \Lambda W, \ V_{+} \Lambda W \rangle = 0,
\quad 
\langle \vec{a} \cdot \nabla W + b\Lambda W, \ 
V_{+} \nabla W \rangle = 0
\quad 
\langle c W, \ V_{-} W \rangle = 0,
\end{equation}
so that $\vec{a}=0$ and $b=c=0$. Thus, we find that $u$ must be trivial and therefore $-1$ is not an eigenvalue of $(-\Delta)^{-1}V$. 
\end{proof}

Now, for $1\le r \le \infty$, we introduce 
the operators $P$ and $\Pi$ on $L^{r}(\mathbb{R}^{3})$ 
as 
\begin{align}
\label{19/01/27/15:38}
Pf&:=
\sum_{j=1}^{3}
\frac{\langle f, V_{+} \partial_{j} W \rangle}{\langle \partial_{j} W, 
V_{+} \partial_{j} W \rangle}
\partial_{j} W
+
\frac{\langle f, V_{+} \Lambda W \rangle}{\langle \Lambda W, 
V_{+} \Lambda W \rangle}\Lambda W
+
\frac{\langle f, i V_{-} W \rangle}{\langle iW, i V_{-} W \rangle}
(i W),
\\[6pt]
\label{19/9/20/15:50}
\Pi f&:=
\sum_{j=1}^{3}
\frac{\langle f, 
V_{+} \partial_{j} W \rangle}{\| V_{+} \partial_{j} W \|_{L^{2}}^{2}}
V_{+}\partial_{j} W
+
\frac{\langle f, V_{+} \Lambda W \rangle}{\|V_{+} \Lambda W \|_{L^{2}}^{2}}
V_{+}\Lambda W
+
\frac{\langle f, i V_{-} W \rangle}{\| V_{-} W \|_{L^{2}}^{2}}
(i V_{-} W).
\end{align}

Observe that $P^{2}=P$, 
$(1-P)^{2}=(1-P)$, 
and for any 
$3 \le r \le \infty$ and 
any $f\in L^{r}(\mathbb{R}^{3})$, 
\begin{equation}\label{19/01/28/15:30}
\| Pf \|_{L^{r}}+ \|(1-P)f\|_{L^{r}} 
+ \|\Pi f \|_{L^{r}}
+ \|(1-\Pi) f \|_{L^{r}}
\lesssim 
\|f\|_{L^{r}}
\lesssim 
\|f\|_{L^{r}},
\end{equation}
where the implicit constant depends only on $r$. 
Moreover, it follows from \eqref{19/10/9/10:22} that 
\begin{equation}\label{19/01/28/09:53}
G P = \Pi G =0.
\end{equation}

\begin{lemma}\label{19/01/31/14:17}
Let $3 < r \le \infty$. 
Then, 
\begin{align}\label{19/10/7/8:25}
(1-P)L^{r}(\mathbb{R}^{3})
&:=\{(1-P)f \colon f\in L^{r}(\mathbb{R}^{3})\}
=X_{r},
\\[6pt]
(1-\Pi)L^{r}(\mathbb{R}^{3})
&:=\{(1-\Pi)f \colon f\in L^{r}(\mathbb{R}^{3})\}
=X_{r}.
\end{align}
Furthermore, $(1-P)f = (1-\Pi)f =f$ 
for all $f\in X_{r}$. 
\end{lemma}
\begin{proof}[Proof of Lemma \ref{19/01/31/14:17}]
Let $g \in (1-P)L^{r}(\mathbb{R}^{3})$. 
Then, there exists $f \in L^{r}(\mathbb{R}^{3})$ such that $g=(1-P)f$. 
By \eqref{19/01/28/15:30}, $g=(1-P)f \in L^{r}(\mathbb{R}^{3})$. 
Furthermore, it is not difficult to see that 
\begin{equation}\label{19/01/27/16:25}
\langle (1-P) f,\ V_{+}\partial_{j} W \rangle=0, 
\quad 
\langle (1-P) f,\ V_{+}\Lambda W \rangle=0,
\quad 
\langle (1-P) f,\ i V_{-}W \rangle=0.
\end{equation}
Hence, $(1-P)L^{r}(\mathbb{R}^{3})\subset X_{r}$. On the other hand, let $f\in X_{r}$. Then, it is obvious that $f=Pf+(1-P)f=(1-P)f$. Thus, $X_{r} \subset (1-P)L^{r}(\mathbb{R}^{3})$. Similarly, we can prove that $(1-\Pi)L^{r}(\mathbb{R}^{3})=X_{r}$.
\end{proof}
\subsection{Estimates of scalar products}
Using the fundamental theorem of calculus 
and \eqref{F-Talenti-1}, one can rewrite the Talenti function $W$ as 
\begin{equation}\label{19/9/26/15:7}
W (x)
= 
\Big( \frac{|x|^{2}}{3} \Big)^{-\frac{1}{2}}
-
\frac{1}{2}\int_{0}^{1} 
\Big(\theta +\frac{|x|^{2}}{3} 
\Big)^{-\frac{3}{2}}\,d\theta.
\end{equation}
Furthermore, it is known (see, e.g., Corollary 5.10 of \cite{Lieb-Loss}) that 
\begin{equation}\label{18/12/22/10:07}
\mathcal{F}\bigm[|x|^{-1}\bigm](\xi)
=
\frac{4 \pi^{\frac{3}{2}}}{\Gamma(\frac{1}{2})}
|\xi|^{-2} ,
\end{equation}
where $\Gamma$ denotes the gamma function.

\begin{proposition}\label{18/12/02/11:25}
There exist $A_{+}>0$ 
and $A_{-}>0$ 
such that for any $0<\alpha <1$, $2\le q<\infty$ and 
$g \in L^{q}(\mathbb{R}^{3})$, 
the following estimates hold: 
\begin{equation}\label{18/12/02/11:26}
\big|
\langle (-\Delta + \alpha)^{-1} V_{+}g , \Lambda W \rangle
+ A_{+} \Re \mathcal{F}[V_{+}g](0) \alpha^{-\frac{1}{2}} 
\big|
\lesssim 
\min{\{ \| g \|_{L^{q}},\, \|(1+|x|)^{-1}g\|_{L^{2}}\}},
\end{equation}
\begin{equation}\label{19/10/1/10:17}
\big|
\langle (-\Delta + \alpha)^{-1} V_{-}g , i W \rangle
- A_{-} \Im \mathcal{F}[V_{-}g](0) \alpha^{-\frac{1}{2}} 
\big|
\lesssim 
\min{\{ \| g \|_{L^{q}},\, 
\|(1+|x|)^{-1}g\|_{L^{2}}\}},
\end{equation}
where the implicit constants depend only on $q$.
\end{proposition}
We will prove Proposition 
\ref{18/12/02/11:25} shortly. 
Here, it is worthwhile noting that for any $q\ge 2$, 
\begin{equation}\label{19/10/4/10:11}
\big|\mathcal{F}[V_{+}g](0)\big|
\le 
\|V_{+}g\|_{L^{1}} 
\lesssim 
\min{\{ \| g \|_{L^{q}},\, \|(1+|x|)^{-1}g\|_{L^{2}}\}}
.
\end{equation}
Hence, Proposition \ref{18/12/02/11:25} 
together with \eqref{19/10/4/10:11} yields:
\begin{corollary}
\label{18/12/02/11:25cor}
It holds that for any $0< \alpha <1$, 
$2\le q <\infty$ and $g \in L^{q}
(\mathbb{R}^{3})$, 
\begin{align}
\label{18/12/02/11:26cor}
&\big|
\langle (-\Delta + \alpha)^{-1} V_{+}g ,\ \Lambda W \rangle
\big|
\lesssim 
\alpha^{-\frac{1}{2}}
\min{\bigm\{ 
\| g \|_{L^{q}},\ \|(1+|x|)^{-1}g\|_{L^{2}}
\bigm\}}
,
\\[6pt]
\label{19/10/1/10:17cor}
&
\big|
\langle (-\Delta + \alpha)^{-1} V_{-}g ,\ i W \rangle
\big|
\lesssim
\alpha^{-\frac{1}{2}}
\min{\bigm\{ 
\| g \|_{L^{q}},\ \|(1+|x|)^{-1}g\|_{L^{2}}
\bigm\}}, 
\end{align}
where the implicit constants depend only on $q$.
\end{corollary}

Now, we give a proof of 
Proposition \ref{18/12/02/11:25}:
\begin{proof}[Proof of Proposition \ref{18/12/02/11:25}] 
Observe that 
we can rewrite $\Lambda W$ as 
\begin{equation}\label{18/12/03/10:17}
\begin{split}
\Lambda W (x)
&=
\Big(1+\frac{|x|^{2}}{3}\Big)^{-\frac{3}{2}}
\Big\{ 
\frac{1}{2}\Big(1+\frac{|x|^{2}}
{3}\Big)
-
\frac{|x|^{2}}{3}
\Big\}
\\[6pt] 
&=
\Big(1+\frac{|x|^{2}}{3}\Big)^{-\frac{3}{2}}
\Big\{ 
\frac{1}{2} - 
\frac{|x|^{2}}{6}
\Big\}
\\[6pt] 
&=-\frac{|x|^{2}}{6}
\Big(\frac{|x|^{2}}{3} \Big)^{-\frac{3}{2}}
+
\frac{1}{2}\Big(1+\frac{|x|^{2}}{3} \Big)^{-\frac{3}{2}}
+
Y(x),
\end{split}
\end{equation}
where 
\begin{equation}\label{18/12/04/09:43}
Y(x):=
-
\frac{|x|^{2}}{6}
\Big(1+\frac{|x|^{2}}{3} \Big)^{-\frac{3}{2}}
+
\frac{|x|^{2}}{6}
\Big(\frac{|x|^{2}}{3} \Big)^{-\frac{3}{2}}.
\end{equation}
Notice that by the fundamental theorem of calculus, one has 
\begin{equation}\label{19/10/3/11:27}
Y(x)
=
\frac{|x|^{2}}{4}
\int_{0}^{1} 
\Big(\theta+\frac{|x|^{2}}{3}\Big)^{-\frac{3}{2}-1} \,d \theta
\lesssim
\int_{0}^{1} 
\Big(\theta+\frac{|x|^{2}}{3}\Big)^{-\frac{3}{2}} \,d \theta.
\end{equation}
By \eqref{19/9/26/15:7}, \eqref{18/12/03/10:17} and \eqref{19/10/3/11:27}, we are convinced that the proofs of \eqref{18/12/02/11:26} and \eqref{19/10/1/10:17} are similar. Hence, we omit the proof of \eqref{19/10/1/10:17}. 

We shall prove \eqref{18/12/02/11:26} by an argument similar to the proof of Lemma 2.5 of \cite{CG}. According to \eqref{18/12/03/10:17}, one has 
\begin{equation}\label{18/12/03/16:02}
\begin{split}
\langle (-\Delta+\alpha)^{-1} V_{+} g, \ \Lambda W \rangle
&= 
-\frac{3^{\frac{3}{2}}}{6}
\langle (-\Delta+\alpha)^{-1} V_{+} g ,\ 
|x|^{-1} \rangle
\\[6pt]
&\quad +
\frac{1}{2}\langle (-\Delta+\alpha)^{-1} V_{+} g, \ 
\Big(1+\frac{|x|^{2}}{3} 
\Big)^{-\frac{3}{2}}
\rangle
\\[6pt]
&\quad + 
\langle (-\Delta+\alpha)^{-1} V_{+} g, \ Y \rangle.
\end{split}
\end{equation}

In what follows, let $2\le q <\infty$, and let $q'$ 
denote the H\"older conjugate of $q$, 
namely $q'=\frac{q}{q-1}$. 
Furthermore, we allow the implicit constants 
to depend on $q$.

\noindent 
{\bf Step 1.}~We consider the first term on the right-hand side of \eqref{18/12/03/16:02}. Using Parseval's identity, \eqref{18/12/22/10:07} and the formula $(-\Delta+\alpha)^{-1}= \mathcal{F}^{-1}(|\xi|^{2}+\alpha)^{-1} \mathcal{F}$, one can see that 
\begin{equation}\label{18/12/02/17:01}
\begin{split}
&
-
\frac{3^{\frac{3}{2}}}{6}
\langle (-\Delta+\alpha)^{-1} 
V_{+} g, \ |x|^{-1} \rangle
\\[6pt]
&=
- C_{1}
\langle (|\xi|^{2}+\alpha)^{-1}\mathcal{F}[V_{+} g] ,\ |\xi|^{-2} 
\rangle
\\[6pt]
&= 
-C_{1} 
\Re \int_{|\xi|\le 1} 
\frac{\mathcal{F}[V_{+} g](0)}{(|\xi|^{2}+\alpha)|\xi|^{2}} 
\,d\xi -
C_{1} \Re
\int_{|\xi|\le 1} 
\frac{\mathcal{F}[V_{+} g](\xi)
-
\mathcal{F}[V_{+} g](0)}{(|\xi|^{2}+\alpha)|\xi|^{2}} 
\,d\xi 
\\[6pt]
&\quad -
C_{1} \Re \int_{1\le|\xi|} 
\frac{\mathcal{F}[V_{+} g](\xi)}{(|\xi|^{2}+\alpha)|\xi|^{2}} 
\,d\xi ,
\end{split}
\end{equation}
where $C_{1}$ is some positive constant. 

We consider the first term on the right-hand side of \eqref{18/12/02/17:01}. Using the spherical coordinates, one can verify that 
\begin{equation}\label{18/12/04/09:59}
\int_{|\xi|\le 1} 
\frac{\mathcal{F}[V_{+} g](0)}{(|\xi|^{2}+\alpha)|\xi|^{2}} 
\,d\xi 
=
C_{2}
\mathcal{F}[V_{+} g](0) \alpha^{-\frac{1}{2}}\arctan{(\alpha^{-\frac{1}{2}})}, 
\end{equation}
where $C_{2}$ is some positive constant. 
Note here that 
\begin{equation}
\label{18/12/25/09:44}
\Big|\frac{\pi}{2} -\arctan{(\alpha^{-\frac{1}{2}})} \Big|
= 
\int_{\alpha^{-\frac{1}{2}}}^{\infty}\frac{1}{1+t^{2}}\,dt 
\le 
\int_{\alpha^{-\frac{1}{2}}}^{\infty}t^{-2} \,dt 
\le 
\alpha^{\frac{1}{2}}
.
\end{equation}

We consider the second term on the 
right-hand side of \eqref{18/12/02/17:01}. First, observe that 
\begin{equation}\label{18/12/05/09:35}
\int_{|\xi|\le 1} \frac{\xi}{(|\xi|^{2}+\alpha)|\xi|^{2}}\,d\xi
=0.
\end{equation}
Then, it follows from $d = 3$, 
\eqref{18/12/05/09:35} and 
elementary computations that 
\begin{equation}\label{19/10/2/14:27}
\begin{split}
&
\Bigm| 
\int_{|\xi|\le 1} 
\frac{\mathcal{F}[V_{+} g](\xi)-\mathcal{F}[V_{+} g](0)}{(|\xi|^{2}+\alpha)|\xi|^{2}} 
\,d\xi
\Bigm| 
\\[6pt]
&=
\Big| 
\int_{|\xi|\le 1} 
\frac{\mathcal{F}[V_{+} g](\xi)-\mathcal{F}[V_{+} g](0)
-\xi\cdot \nabla \mathcal{F}[V_{+}g](0)}{(|\xi|^{2}+\alpha)|\xi|^{2}} 
\,d\xi
\Big| 
\\[6pt]
&\le 
\int_{|\xi|\le 1} \frac{1}{(|\xi|^{2}+\alpha)|\xi|^{2}}
\int_{\mathbb{R}^{3}} \big| e^{-ix\cdot\xi} -1 - \xi \cdot (-ix) \big| 
|V_{+}(x) g(x)| \,dx d\xi
\\[6pt]
&\lesssim 
\int_{|\xi|\le 1} \frac{1}{(|\xi|^{2}+\alpha)|\xi|^{2}}
\int_{\mathbb{R}^{3}} 
\min\{|x||\xi| , \ |x|^{2}|\xi|^{2}\}
|V_{+}(x) g(x)| \,dx d\xi.
\end{split} 
\end{equation}
Let $1< q_{1} \le 2$ be a fixed number depending only on $q$ such that $q_{1}-1 <\frac{3}{q}$ (hence $(4-q_{1}) q'>3$). 
Then, it follows from \eqref{19/10/2/14:27}, 
$|V_{+}(x)| \lesssim (1+|x|)^{-4}$ 
and H\"older's inequality that 
\begin{equation}\label{19/10/2/15:41}
\begin{split}
&\Big| 
\int_{|\xi|\le 1} 
\frac{\mathcal{F}[V_{+} g](\xi)-\mathcal{F}[V_{+} g](0)}{(|\xi|^{2}+\alpha)|\xi|^{2}} 
\,d\xi
\Big| 
\\[6pt]
&\lesssim 
\int_{|\xi|\le 1} 
\frac{|\xi|^{q_{1}}}{(|\xi|^{2}+\alpha)|\xi|^{2}} 
\,d\xi
\int_{\mathbb{R}^{3}} 
|x|^{q_{1}}|V_{+}(x) g(x)| \,dx
\\[6pt]
&\lesssim 
\int_{|\xi|\le 1} 
|\xi|^{q_{1} - 4} 
\,d\xi
\|(1+|x|)^{- 4 + q_{1}} \|_{L^{q'}}
\|g\|_{L^{q}} 
\lesssim \|g\|_{L^{q}}.
\end{split} 
\end{equation}
Similarly, one can verify that 
\begin{equation}\label{19/10/2/14:53}
\begin{split}
&
\Big| 
\int_{|\xi|\le 1} 
\frac{\mathcal{F}[V_{+} g](\xi) -
\mathcal{F}[V_{+} g](0)}{(|\xi|^{2}+\alpha)|\xi|^{2}} 
\,d\xi
\Big| 
\\[6pt]
&\lesssim 
\int_{|\xi|\le 1} 
\frac{|\xi|^{\frac{5}{4}}}{(|\xi|^{2}+\alpha)|\xi|^{2}} 
\,d\xi
\int_{\mathbb{R}^{3}} 
|x|^{\frac{5}{4}}|V_{+}(x) g(x)| \,dx
\\[6pt]
&\lesssim 
\||x|^{\frac{5}{4}} (1+|x|)^{-3} \|_{L^{2}}
\int_{|\xi| \leq 1} 
|\xi|^{- \frac{11}{4}} d\xi
\|(1+|x|)^{-1}g\|_{L^{2}} 
\lesssim 
\|(1+|x|)^{-1}g\|_{L^{2}}.
\end{split} 
\end{equation}

Putting the estimates \eqref{19/10/2/15:41} 
and \eqref{19/10/2/14:53} together, we find that 
for any $2\le q <\infty$, 
\begin{equation}\label{19/10/3/10:17}
\Big| 
\int_{|\xi|\le 1} 
\frac{\mathcal{F}[V_{+} g](\xi) 
-\mathcal{F}[V_{+} g](0)}{(|\xi|^{2}+\alpha)|\xi|^{2}} 
\,d\xi
\Big| 
\lesssim 
\min\bigm\{
\|g\|_{L^{q}}
, \ 
\|(1+|x|)^{-1} g\|_{L^{2}}
\bigm\}.
\end{equation}

We consider the last term 
on the right-hand side of \eqref{18/12/02/17:01}. 
By H\"older's inequality, Plancherel's theorem 
and $|V_{+}(x)|\lesssim (1+|x|)^{- 4}$, 
one can see that 
\begin{equation}\label{18/12/04/11:15}
\begin{split}
\Big| 
\int_{1\le |\xi|} 
\frac{\mathcal{F}[V_{+} g](\xi)}{(|\xi|^{2}+\alpha)|\xi|^{2}} 
\,d\xi
\Big| 
&\le 
\||\xi|^{-4} \|_{L^{2}(|\xi|\ge 1)}
\| \mathcal{F}[ V_{+} g] \|_{L^{2}} 
\\[6pt]
&\lesssim 
\| V_{+} g \|_{L^{2}}
\le 
\min{\bigm\{ \|g\|_{L^{q}}, \ \|(1+|x|)^{-1} g \|_{L^{2}}
\bigm\}}.
\end{split} 
\end{equation}

Now, we conclude from \eqref{18/12/02/17:01} through \eqref{18/12/25/09:44}, \eqref{19/10/3/10:17} and \eqref{18/12/04/11:15} that 
there exists a positive constant $A_{+}$ such that for any $2\le q <\infty$, 
\begin{equation}\label{18/12/04/11:19}
\begin{split}
&
\big| 
- \frac{3^{\frac{3}{2}}}{6}
\langle (-\Delta + \alpha)^{-1} V_{+}g ,\ |x|^{-1} \rangle
-
A_{+} \Re \mathcal{F}[V_{+}g](0) \alpha^{-\frac{1}{2}} 
\big| 
\\[6pt]
&\lesssim 
\min\big\{ 
\| g \|_{L^{q}}, \ 
\| (1+|x|)^{-1} g\|_{L^{2}}
\big\}.
\end{split} 
\end{equation}

\noindent 
{\bf Step 2.}~We shall derive an estimate 
for the second term on the right-hand side 
of \eqref{18/12/03/16:02}. 
It follows from H\"older's inequality, 
Lemma \ref{18/11/23/17:17} that for any $2\le q <\infty$, 
\begin{equation}\label{18/12/02/14:01}
\begin{split}
&\big|
\langle (-\Delta + \alpha)^{-1} V_{+} g ,\ 
\Big(1+\frac{|x|^{2}}{3} \Big)^{-\frac{3}{2}}
\rangle
\big|
\\[6pt]
&\lesssim 
\|(-\Delta + \alpha)^{-1} V_{+} g \|_{L^{6}} 
\| \Big(1+\frac{|x|^{2}}{3} \Big)^{-\frac{3}{2}} 
\|_{L^{\frac{6}{5}}}
\\[6pt]
& \lesssim 
\|V_{+} g \|_{L^{\frac{6}{5}}} 
\lesssim
\min\big\{ 
\| g \|_{L^{q}}, \ 
\| (1+|x|)^{-1} g\|_{L^{2}}
\big\}.
\end{split} 
\end{equation}

\noindent 
{\bf Step 3.}~Finally, we consider the last term 
on the right-hand side of \eqref{18/12/03/16:02} 
and finish the proof.
It follows from \eqref{19/10/3/11:27} that 
\begin{equation}\label{18/12/04/08:55}
\begin{split}
\|Y\|_{L^{\frac{6}{5}}}
&\lesssim 
\Big\| \int_{0}^{1} 
\Big(\theta+\frac{|x|^{2}}{3}\Big)^{-\frac{3}{2}} 
\,d \theta
\Big\|_{L^{\frac{6}{5}}}
\\[6pt]
&\lesssim 
\int_{0}^{1}
\Big\| 
\theta^{-\frac{7}{8}} 
|x|^{-d+\frac{7}{4}} 
\Big\|_{L^{\frac{6}{5}}(|x|\le 1)}
\, d\theta +
\big\| |x|^{-d}
\big\|_{L^{\frac{6}{5}}(1\le |x|)}
\lesssim 1.
\end{split} 
\end{equation}
Hence, by the same computation as 
\eqref{18/12/02/14:01}, and 
\eqref{18/12/04/08:55}, 
one can see that for any 
$2\le q <\infty$, 
\begin{equation}\label{18/12/04/11:42}
\begin{split} 
\big|
\langle (-\Delta + \alpha)^{-1} V_{+} g ,\ 
Y 
\rangle
\big|
&\lesssim 
\| (-\Delta + \alpha)^{-1} V_{+} g \|_{L^{6}} 
\| Y \|_{L^{\frac{2d}{5}}}
\\[6pt]
&\lesssim
\min\big\{ 
\| g \|_{L^{q}}, \ 
\| (1+|x|)^{-1} g\|_{L^{2}}
\big\}. 
\end{split} 
\end{equation}
\par 
Putting \eqref{18/12/03/16:02}, \eqref{18/12/04/11:19}, \eqref{18/12/02/14:01} and \eqref{18/12/04/11:42} together, 
we obtain the desired estimate \eqref{18/12/02/11:26}. 
\end{proof}


\begin{proposition}\label{19/9/26/15:4}
The following estimate holds for any $0<\alpha <1$, $2\le q <\infty$ and any $g \in L^{q}(\mathbb{R}^{3})$: 
\begin{equation}\label{19/9/26/15:15}
\bigm|
\langle (-\Delta + \alpha)^{-1} V_{+}g ,\, \nabla W \rangle
\bigm|
\lesssim 
\min\bigm\{\| g \|_{L^{q}}, \ 
\|(1+|x|)^{-1}g\|_{L^{2}} 
\bigm\},
\end{equation}
where the implicit constant depends only on $q$.
\end{proposition}
\begin{proof}[Proof of Proposition 
\ref{19/9/26/15:4}] 
the general symbol $d$ to make the 
dependence on the dimension clear. 
Using the integration by parts and \eqref{19/9/26/15:7}, one can see that for any $1\le j \le 3$, 
\begin{equation}\label{19/9/26/15:23}
\begin{split}
&\langle (-\Delta + \alpha)^{-1} V_{+} g, \ \partial_{j} W \rangle
\\[6pt]
&= 
\frac{-1}{3}\langle (-\Delta + \alpha)^{-1} \partial_{j}(V_{+} g), \ 
|x|^{-1}
\rangle
\\[6pt]
&\quad +
\bigm\langle 
(-\Delta+ \alpha)^{-1} \partial_{j}(V_{+} g), \ 
\frac{1}{2} 
\int_{0}^{1} 
\Big(\theta +\frac{|x|^{2}}{3} 
\Big)^{-\frac{3}{2}}\,d\theta
\bigm\rangle.
\end{split} 
\end{equation}

\noindent 
Throughout the proof, 
we allow the implicit constants to 
depend on $q$. 

\noindent
{\bf Step 1.}~We consider the first term on the right-hand side of \eqref{19/9/26/15:23}. 
First, note that 
\begin{equation}\label{19/9/29/8:37}
\int_{|\xi|\le 1} 
\frac{\xi_{j}}{(|\xi|^{2}+\alpha)|\xi|^{2}} 
\,d\xi =0.
\end{equation}
Then, by the same computation as 
\eqref{18/12/02/17:01}, and \eqref{19/9/29/8:37}, 
one sees that 
\begin{equation}\label{19/9/26/15:47}
\begin{split}
&
-\frac{1}{3}\langle (-\Delta + \alpha)^{-1} 
\partial_{j}(V_{+} g), \ |x|^{-1} \rangle
\\[6pt]
&= 
- C_{1}
\Re\int_{|\xi|\le 1} 
\frac{i\xi_{j}\bigm\{ \mathcal{F}[V_{+} g](\xi)-\mathcal{F}[V_{+} g](0)\bigm\}}{(|\xi|^{2}+\alpha)|\xi|^{2}} 
\, d\xi 
\\[6pt]
&\quad +
C_{1} \Re\int_{1\le|\xi|} 
\frac{i\xi_{j}\mathcal{F}[V_{+} g](\xi)}{(|\xi|^{2}+\alpha)|\xi|^{2}} 
\,d\xi,
\end{split}
\end{equation}
where $C_{1}$ is some positive 
constant. 

We consider the first term on the right hand side of \eqref{19/9/26/15:47}. 
By elementary computations and 
$|V_{+}(x)|\sim (1+|x|)^{- 4}$, 
one can verify that 
\begin{equation}\label{19/9/26/16:20}
\begin{split} 
\big| 
\mathcal{F}[V_{+} g](\xi)
-
\mathcal{F}[V_{+} g](0)
\big| 
&=
\big| 
\int_{\mathbb{R}^{3}} \big\{ e^{-ix\cdot\xi} -1 \big\} (V_{+}g)(x)\,dx 
\big|
\\[6pt]
&\lesssim 
\int_{\mathbb{R}^{3}} 
\min\{ 1, \, |\xi| |x|\} (1+|x|)^{-4} 
|g(x)| \,dx
\\[6pt]
&\lesssim 
|\xi|
\min\bigm\{ 
\|g\|_{L^{q}} , \ 
\|(1+|x|)^{-1}g\|_{L^{2}}
\bigm\}. 
\end{split}
\end{equation}
Using \eqref{19/9/26/16:20}, one can find that 
\begin{equation}\label{18/12/04/10:31}
\Big| 
\int_{|\xi|\le 1} 
\frac{i\xi_{j}\bigm\{ 
\mathcal{F}[V_{+} g](\xi)-\mathcal{F}[V_{+} g](0)\bigm\}
}{(|\xi|^{2}+\alpha)|\xi|^{2}} 
\,d\xi
\Big| 
\lesssim 
\min\bigm\{
\|g\|_{L^{q}}, \ \|(1+|x|)^{-1}g\|_{L^{2}}\bigm\}.
\end{equation}

We consider the second term 
on the right-hand side of \eqref{19/9/26/15:47}. 
It is easy to verify that if $q\ge 2$, then 
\begin{equation}\label{19/9/29/9:39}
\begin{split} 
&
\Bigm| 
\int_{1\le |\xi|} 
\frac{i\xi_{j} 
\mathcal{F}[V_{+} g](\xi)
}{(|\xi|^{2}+\alpha)|\xi|^{2}} 
\,d\xi
\Bigm| 
\le
\int_{1\le |\xi|} 
\frac{ 
\big| \mathcal{F}[V_{+} g](\xi) \big|
}{|\xi|^{3}} 
\,d\xi
\\[6pt]
&\le 
\||\xi|^{-3} \|_{L^{2}(|\xi|\ge 1)}
\|V_{+}g\|_{L^{2}} 
\lesssim 
\min\bigm\{ 
\|g\|_{L^{q}},
\ \|(1+|x|)^{-1}g\|_{L^{2}}\bigm\}.
\end{split}
\end{equation}

Thus, putting \eqref{19/9/26/15:47}, \eqref{18/12/04/10:31} and \eqref{19/9/29/9:39} together, we find that for any $q\ge 2$, 
the first term on the right-hand side of 
\eqref{19/9/26/15:23} is estimated as follows: 
\begin{equation}\label{19/9/30/8:15}
\big| 
\bigm\langle (-\Delta + \alpha)^{-1} \partial_{j}(V_{+}g) , \ |x|^{-1} 
\bigm\rangle
\big| 
\lesssim 
\min\{ \|g\|_{L^{q}}, 
\|(1+|x|)^{-1}g \|_{L^{2}}
\}. 
\end{equation}
\\ 
\noindent 
{\bf Step 2.}~We move on to the second term on the right-hand side of \eqref{19/9/26/15:23}, and finish the proof. By H\"older's inequality, one has 
\begin{equation}\label{19/9/30/8:27}
\begin{split}
&
\big|
\bigm\langle 
(-\Delta+ \alpha)^{-1} \partial_{j}(V_{+} g), \ 
\frac{1}{2}
\int_{0}^{1} 
\Big(\theta +\frac{|x|^{2}}{3} 
\Big)^{-\frac{3}{2}}\,d\theta
\bigm\rangle
\big|
\\[6pt]
&\le 
\| (-\Delta + \alpha)^{-1} \partial_{j} (V_{+} g) \|_{L^{6}}
\| \frac{1}{2} \int_{0}^{1} 
\Big(\theta +\frac{|x|^{2}}{3} 
\Big)^{-\frac{3}{2}}\,d\theta \|_{L^{\frac{6}{5}}}.
\end{split} 
\end{equation}
Using the first resolvent equation \eqref{first-res} 
with $\alpha_{0} = 0$, Sobolev's inequality, and 
\eqref{res-est1} in Lemma \ref{Delta-ineq}, we see that for any $2 \le q<\infty$, 
\begin{equation}\label{19/9/30/8:33}
\begin{split}
&
\| (-\Delta + \alpha)^{-1} \partial_{j} (V_{+} g) \|_{L^{6}}
\\[6pt]
&\le
\| (-\Delta)^{-1} \partial_{j} (V_{+} g) \|_{L^{6}}
+
\alpha \| (-\Delta +\alpha)^{-1} (-\Delta)^{-1} \partial_{j} (V_{+} g) \|_{L^{6}}
\\[6pt]
&\lesssim 
\| |\nabla|^{-1} (V_{+} g) \|_{L^{6}}
+
\alpha \||\nabla|^{-1} (-\Delta +\alpha)^{-1} 
(V_{+} g) \|_{L^{6}}
\\[6pt]
&\lesssim 
\| V_{+} g \|_{L^{2}}
+
\alpha \|(-\Delta +\alpha)^{-1} (V_{+} g) \|_{L^{2}}
\\[6pt]
&\lesssim 
\| V_{+} g \|_{L^{2}}
\lesssim 
\min\bigm\{
\|g\|_{L^{q}}, \ 
\|(1+|x|)^{-1}g\|_{L^{2}}
\bigm\}. 
\end{split} 
\end{equation}
Moreover, 
it follows that 
\begin{equation}\label{19/9/30/16:51}
\begin{split}
&
\big\| 
\int_{0}^{1} 
\Big(\theta+\frac{|x|^{2}}{3}
\Big)^{-\frac{3}{2}} \,d \theta
\big\|_{L^{\frac{6}{5}}}
\\[6pt]
&\lesssim 
\int_{0}^{1}
\Big\| 
\theta^{-\frac{3}{4}} 
|x|^{- \frac{3}{2}} 
\Big\|_{L^{\frac{6}{5}}
(|x|\le 1)}
\,d\theta 
+
\big\| |x|^{-3}
\big\|_{L^{\frac{6}{5}}
(1\le |x|)}
\lesssim 1.
\end{split} 
\end{equation}

Plugging \eqref{19/9/30/8:33} and \eqref{19/9/30/16:51} into \eqref{19/9/30/8:27}, 
we find that for any $2 \le q<\infty$,
\begin{equation}\label{19/9/30/9:21}
\begin{split}
&
\big|
\bigm\langle 
(-\Delta + \alpha)^{-1} \partial_{j}(V_{+} g), \ 
\frac{1}{2}
\int_{0}^{1} 
\Big(\theta +\frac{|x|^{2}}{3} 
\Big)^{-\frac{3}{2}}\,d\theta
\bigm\rangle
\big|
\\[6pt]
&\lesssim 
\min\bigm\{
\|g\|_{L^{q}},\ \|(1+|x|)^{-1}g\|_{L^{2}} \bigm\}. 
\end{split}
\end{equation}

Now, we find from \eqref{19/9/26/15:23}, \eqref{19/9/30/8:15} and \eqref{19/9/30/9:21} that the desired estimate \eqref{19/9/26/15:15} holds. 
\end{proof}

In the special case $g=\partial_{j}W$ 
in Proposition \ref{19/9/26/15:4}, we can obtain a detailed estimate:
\begin{lemma}\label{19/9/30/9:26}
It holds that for any $\alpha>0$ 
and $1\le j \le 3$, 
\begin{equation}\label{19/9/30/9:35}
\big|
\langle (-\Delta + \alpha)^{-1} V_{+}\partial_{j} W, \, \partial_{j} W 
\rangle 
+
\| \partial_{j} W \|_{L^{2}}^{2} 
\big| \lesssim 
\alpha^{\frac{1}{4}}. 
\end{equation} 
\end{lemma}
\begin{proof}[Proof of Lemma \ref{19/9/30/9:26}]
It follows from 
the first resolvent equation 
\eqref{first-res} with 
$\alpha_{0}=0$, 
\eqref{19/10/9/10:22} and 
\eqref{res-est2} in Lemma \ref{Delta-ineq} that 
\begin{equation}\label{19/9/28/17:15}
\begin{split} 
&
\big|
\langle (-\Delta + \alpha)^{-1} V_{+}\partial_{j} W, \, \partial_{j} W \rangle
-
\langle (-\Delta)^{-1} V_{+}\partial_{j} W, \, \partial_{j} W \rangle
\big|
\\[6pt]
&=
\alpha 
\big|\langle (-\Delta+ \alpha)^{-1} (-\Delta)^{-1} V_{+}\partial_{j}W, \, \partial_{j} W 
\rangle \big|
\\[6pt]
&\lesssim 
\alpha
\|(-\Delta)^{-1} V_{+}\partial_{j} W\|_{L^{2}}
\|(-\Delta +\alpha)^{-1} \partial_{j} W\|_{L^{2}}
\\[6pt]
&=
\alpha 
\| \partial_{j}W\|_{L^{2}}
\alpha^{-\frac{3}{4}}
\| \partial_{j}W\|_{L_{\rm weak}^{\frac{3}{2}}}
\lesssim \alpha^{\frac{1}{4}}. 
\end{split} 
\end{equation}
Moreover, by \eqref{19/10/9/10:22}, 
one can see that 
\begin{equation}\label{19/9/30/10:12}
\langle (-\Delta)^{-1} V_{+}\partial_{j} W, \, \partial_{j} W \rangle
=
\langle (-\Delta)^{-1} \Delta \partial_{j} W, \, \partial_{j} W \rangle
=
-\|\partial_{j}W \|_{L^{2}}^{2}
. 
\end{equation}

Putting \eqref{19/9/28/17:15} and 
\eqref{19/9/30/10:12} together, we obtain \eqref{19/9/30/9:35}. 
\end{proof}
\subsection{Resolvent estimate}
\label{sec7-reso}
For $3 <r \le \infty$, 
we introduce the linear subspaces 
$\mathbf{X}_{r}$ and $\mathbf{Y}_{r}$ of 
$L^{r}(\mathbb{R}^{3})\times L^{r}(\mathbb{R}^{3})$ 
equipped with the norms $\|\cdot\|_{\mathbf{X}_{r}}$ 
and $\|\cdot \|_{\mathbf{Y}_{r}}$: 
\begin{align}
\label{19/01/30/13:41}
\mathbf{X}_{r} &:=X_{r} \times P L^{r}(\mathbb{R}^{3}),
\qquad 
\|(f_{1},f_{2})\|_{\mathbf{X}_{r}}
:=
\|f_{1}\|_{L^{r}}+\|f_{2}\|_{L^{r}}
,
\\[6pt]
\label{19/02/01/10:45}
\mathbf{Y}_{r} &:=X_{r} \times \Pi L^{r}(\mathbb{R}^{3}),
\qquad 
\|(f_{1},f_{2})\|_{\mathbf{Y}_{r}}
:=
\|f_{1}\|_{L^{r}}+\|f_{2}\|_{L^{r}}
,
\end{align}
where $PL^{r}(\mathbb{R}^{3})
:=\{ Pf \colon f\in L^{r}(\mathbb{R}^{3})\}$ and 
$\Pi L^{r}(\mathbb{R}^{3}):=\{ \Pi f \colon f\in L^{r}(\mathbb{R}^{3})\}$ 
(see \eqref{19/01/27/15:38} and \eqref{19/9/20/15:50}). 

Next, for any $\varepsilon >0$, we define the linear 
operator $B_{\varepsilon} \colon 
\mathbf{X}_{r} \to L^{r}(\mathbb{R}^{3})$ by 
\begin{equation}\label{19/01/30/16:03}
B_{\varepsilon}f=B_{\varepsilon}(f_{1},f_{2})
:=\varepsilon f_{1} + f_{2} 
\quad 
\mbox{for $f=(f_{1}, f_{2})\in \mathbf{X}_{r}$}. 
\end{equation}
We also define the linear operator 
$\mathbf{C} \colon L^{r}(\mathbb{R}^{3}) \to 
\mathbf{Y}_{r}$ by 
\begin{equation}\label{19/01/30/17:02}
\mathbf{C}f:=((1-P)f,\ \Pi f)
\quad 
\mbox{for $f\in L^{r}(\mathbb{R}^{3})$} .
\end{equation}
Note here that by Lemma \ref{19/01/31/14:17}, 
$(1-P)L^{r}(\mathbb{R}^{3})=X_{r}$. 
\begin{lemma}\label{19/01/28/16:00}
Let $3 <r \le \infty$ and 
$\varepsilon>0$. 
Then, $B_{\varepsilon}$ is surjective. 
Moreover, $C$ is injective.
\end{lemma}
\begin{proof}[Proof of Lemma \ref{19/01/28/16:00}] 
Let $g \in L^{r}(\mathbb{R}^{3})$, 
and put $f:=(\varepsilon^{-1}(1-P)g, P g)$. 
Then, by Lemma \ref{19/01/31/14:17}, 
$f \in \mathbf{X}_{r}$. 
Moreover, it is obvious that $B_{\varepsilon}f=g$. 
Hence, $B_{\varepsilon} \colon \mathbf{X}_{r} \to 
L^{r}(\mathbb{R}^{3})$ is 
surjective. 
\par 
Next, suppose $\mathbf{C}f=0$. Then, $(1-P)f=0$ and $\Pi f=0$. 
In particular, 
\begin{equation}\label{19/10/7/8:56}
f=Pf+(1-P)f = Pf
= 
\sum_{j=1}^{3} c_{j} \partial_{j}W 
+ a \Lambda W + i b W 
\end{equation}
for some $c_{1}, c_{2}, c_{3}, a, b \in \mathbb{R}$. 
Furthermore, this together with 
$\Pi f =0$ implies that 
$c_{1}=c_{2} =c_{d}=a=b=0$. 
Thus, $f=0$ and the linear operator 
$\mathbf{C}\colon L^{r}(\mathbb{R}^{3}) \to 
\mathbf{Y}_{r}$ is injective. 
\end{proof}

\begin{lemma}\label{19/01/31/17:09}
Assume $3 < r <\infty$. 
Let $G_{0}$ be the operator defined 
by \eqref{19/10/7/10:51}. 
Then, the following estimates hold for all $\alpha>0$: 
\begin{align}
\label{19/01/31/17:41}
& \bigm\|
\alpha (-\Delta+\alpha)^{-1} G_{0} P 
\bigm\|_{L^{r}\to L^{r}}
\lesssim 
\alpha^{\frac{1}{2}-\frac{3}{2r}},
\\[6pt]
\label{19/01/31/17:42}
&\bigm\| 
\alpha \Pi (-\Delta+\alpha)^{-1} G_{0} 
\bigm\|_{L^{r} \to \Pi L^{r}}
\lesssim 
\alpha^{-\frac{1}{2}} \alpha,
\end{align} 
where the implicit constants depend only on $r$. 
\end{lemma}
\begin{proof}[Proof of Lemma \ref{19/01/31/17:09}]
First, we shall prove \eqref{19/01/31/17:41}. It follows from \eqref{19/10/9/10:22} and Lemma \ref{Delta-ineq} 
that for any $f \in L^{r}(\mathbb{R}^{3})$, 
\begin{equation}\label{19/01/31/17:10}
\begin{split} 
&\|\alpha (-\Delta+\alpha)^{-1} G_{0} P f\|_{L^{r}}
\\[6pt]
&\le 
\bigm\| 
\alpha (-\Delta+\alpha)^{-1}(-\Delta)^{-1}V_{+} \Re[Pf] 
\bigm\|_{L^{r}}
+
\bigm\| 
\alpha (-\Delta+\alpha)^{-1}(-\Delta)^{-1}V_{-} \Im[Pf] 
\bigm\|_{L^{r}}
\\[6pt]
&\le 
\alpha \sum_{j=1}^{3}
\frac{|\langle f, V_{+} \partial_{j} W \rangle|}
{|\langle \partial_{j} W, V_{+} \partial_{j} W \rangle|} 
\| (-\Delta+\alpha)^{-1}
(-\Delta)^{-1}V_{+}\partial_{j} W \|_{L^{r}}
\\[6pt]
&\quad +
\alpha \frac{|\langle f, V_{+} \Lambda W \rangle|}
{|\langle \Lambda W, V_{+} \Lambda W \rangle|} 
\| (-\Delta+\alpha)^{-1}
(-\Delta)^{-1}V_{+}\Lambda W \|_{L^{r}}\\[6pt]
&\quad +
\alpha \frac{|\langle f, V_{-} W \rangle|}
{|\langle W, V_{-} W \rangle|} 
\| (-\Delta+\alpha)^{-1}(-\Delta)^{-1}V_{-} W \|_{L^{r}}
\\[6pt]
&\lesssim 
\alpha \|f\|_{L^{r}} \sum_{j=1}^{3}
\| (-\Delta+\alpha)^{-1}\partial_{j} W \|_{L^{r}}
\\[6pt]
&\quad +
\alpha \|f\|_{L^{r}} \| (-\Delta+\alpha)^{-1}\Lambda W \|_{L^{r}}
+
\alpha \|f\|_{L^{r}} \| (-\Delta+\alpha)^{-1} W \|_{L^{r}}
\\[6pt]
&\lesssim
\alpha^{\frac{3}{2}(\frac{1}{3}-\frac{1}{r})} \|f\|_{L^{r}}
\Bigm\{ 
\sum_{j=1}^{3} \|\partial_{j} W \|_{L_{\rm weak}^{3}}
+
\|\Lambda W \|_{L_{\rm weak}^{3}}
+
\|W \|_{L_{\rm weak}^{3}}
\Bigm\} 
\\[6pt]
&\lesssim
\alpha^{\frac{1}{2}-\frac{3}{2r}} \|f\|_{L^{r}}, 
\end{split}
\end{equation}
where the implicit constant depends only on $r$. 
Thus, the desired estimate \eqref{19/01/31/16:28} holds. 
\par 
Next, we shall prove \eqref{19/01/31/17:42}. 
It follows from \eqref{19/10/9/10:22}, 
Corollary \ref{18/12/02/11:25cor} and 
Proposition \ref{19/9/26/15:4} 
that for any $f \in L^{r}(\mathbb{R}^{3})$, 
\begin{equation}\label{19/01/31/17:43}
\begin{split} 
& 
\bigm\|
\alpha \Pi (-\Delta+\alpha)^{-1} G_{0} f 
\bigm\|_{L^{r}}
\\[6pt]
&\le 
\bigm\|
\alpha \Pi (-\Delta+\alpha)^{-1} 
(-\Delta)^{-1}V_{+} \Re[f] 
\bigm\|_{L^{r}}
+ 
\bigm\|
\alpha \Pi (-\Delta+\alpha)^{-1} 
(-\Delta)^{-1}V_{+} \Re[f] 
\bigm\|_{L^{r}}
\\[6pt]
&\le 
\alpha \sum_{j=1}^{3}
\frac{|\langle (-\Delta)^{-1}(-\Delta+\alpha)^{-1}V_{+}\Re[f],\ (-\Delta) \partial_{j} W \rangle|}{
\|V_{+} \partial_{j} W \|_{L^{2}}^{2}} \| V_{+}\partial_{j} W \|_{L^{r}}
\\[6pt]
&\quad +
\alpha \frac{|\bigm\langle 
(-\Delta)^{-1}(-\Delta+\alpha)^{-1}V_{+}\Re[f], 
\ (-\Delta) \Lambda W \bigm\rangle|}{
\|V_{+} \Lambda W \|_{L^{2}}^{2}} \| V_{+} \Lambda W \|_{L^{r}}
\\[6pt]
&\quad +
\alpha \frac{|\bigm\langle 
(-\Delta)^{-1}(-\Delta+\alpha)^{-1}V_{-}\Im[f],\ 
(-\Delta) W \bigm\rangle|}{
\|V_{-} W \|_{L^{2}}^{2}} \| V_{-} W \|_{L^{r}}
\\[6pt]
&
\lesssim 
\alpha^{\frac{1}{2}} \| f\|_{L^{r}}.
\end{split}
\end{equation}
Thus, we have completed the proof. 
\end{proof}

Now, we are in a position to prove Theorem \ref{thm-ufe}: 
\begin{proof}[Proof of Theorem 
\ref{thm-ufe}]
Let $0< \alpha < e_{0}$, where $-e_{0}$ is the only one negative eigenvalue of $L_{+}=-\Delta +V_{+}$. 
Assume $6<r<\infty$. 
Furthermore, let $\varepsilon >0$ 
be a small constant to be chosen later, dependently on $r$. 
Note here that by Lemma \ref{18/11/23/17:17}, 
the operator $G(\alpha)$ defined by \eqref{def-G}
maps $L^{r}(\mathbb{R}^{3})$ to itself. 
Then, we define the operator $A_{\varepsilon}(\alpha)$ 
from $\mathbf{X}_{r}$ to $\mathbf{Y}_{r}$ by 
\begin{equation}\label{19/01/30/17:21}
A_{\varepsilon}(\alpha)
:= \mathbf{C} G(\alpha) B_{\varepsilon},
\end{equation}
where $G(\alpha)$ is the operator 
defined by \eqref{def-G}. 
Observe that for any $(f_{1},f_{2})\in \mathbf{X}_{r}$, 
\begin{equation}\label{19/01/31/13:15}
A_{\varepsilon}(\alpha) (f_{1},f_{2})
=
\begin{pmatrix}
A_{11}f_{1}+ A_{12}f_{2} \\
A_{21}f_{1}+A_{22}f_{2}
\end{pmatrix}
=
\left( \begin{array}{cc} A_{11}& A_{12} \\ A_{21} & A_{22} 
\end{array}\right)
\left( \begin{array}{c} f_{1} \\ f_{2} 
\end{array}\right),
\end{equation}
where 
we do not care about the distinction 
between column vectors and row ones, and 
\begin{align}\label{A11}
A_{11}&:=\varepsilon (1-P)G(\alpha)|_{X_{r}},
\quad 
A_{12}:=(1-P)G(\alpha) P,
\\[6pt]
\label{A21} 
A_{21}&:=\varepsilon \Pi G(\alpha)|_{X_{r}},
\qquad 
A_{22}:= \Pi G(\alpha) P.
\end{align}

Now, we claim that the inverse of 
$A_{\varepsilon}(\alpha)\colon 
\mathbf{X}_{r}\to \mathbf{Y}_{r}$ exists. 
Notice that by Lemma \ref{19/01/28/16:00}, 
and Lemma 3.12 of \cite{Jensen-Kato}, 
this claim implies that $G(\alpha) \colon 
L^{r}(\mathbb{R}^{3}) \to L^{r}(\mathbb{R}^{3})$ has the inverse and 
\begin{equation}\label{19/01/30/17:18}
G(\alpha)^{-1} = 
B_{\varepsilon} A_{\varepsilon}(\alpha)^{-1}\mathbf{C}.
\end{equation}

We shall prove the claim in several steps. 
Throughout the proof, 
we allow the implicit constants to depend on $r$. 
\\
\noindent 
{\bf Step 1.}~We consider the operator $A_{11} \colon X_{r} \to X_{r}$. 

Recall from Lemma \ref{18/11/13/14:25} 
that $G|_{X_{r}}\colon X_{r}\to X_{r}$ has the inverse, 
say $K_{1}$, such that 
\begin{equation}\label{19/01/31/14:14}
\|K_{1} \|_{X_{r}\to X_{r}} \lesssim 1. 
\end{equation}
Note that Lemma \ref{19/01/31/14:17} shows that 
\begin{equation}\label{19/01/31/15:30}
(1-P) G|_{X_{r}}
=
G|_{X_{r}}.
\end{equation} 
Then, using \eqref{19/01/31/15:30}
and \eqref{19/01/27/18:15}, 
one can rewrite $A_{11}$ as 
\begin{equation}\label{19/01/31/13:52}
A_{11}
=
\varepsilon G|_{X_{r}}
+
\varepsilon S_{11}
=
\varepsilon G|_{X_{r}}(1+K_{1}S_{11}), 
\end{equation}
where 
\begin{equation}\label{19/01/31/14:12}
S_{11}:= - \alpha (1-P) (-\Delta+\alpha)^{-1} G_{0} |_{X_{r}}. 
\end{equation}
By \eqref{19/01/28/15:30}, 
Lemma \ref{18/11/23/17:17} 
and the Hardy-Littlewood-Sobolev inequality 
(\eqref{HLS-ineq} with $s=2$), we see that 
for any $f\in L^{r}(\mathbb{R}^{3})$, 
Notice that by \eqref{19/01/31/14:14} and Lemma \ref{19/01/31/17:09}, 
\begin{equation}\label{19/01/31/15:48}
\begin{split}
\|S_{11} f\|_{L^{r}}
&\lesssim 
\alpha 
\|(-\Delta +\alpha)^{-1} G_{0}f \|_{L^{r}}
\\[6pt]
&\lesssim 
\alpha^{\frac{3}{2}(\frac{1}{4}-\frac{1}{r})}
\big\{
\|(-\Delta)^{-1}V_{+}\Re[f]\|_{L^{4}}
+
\|(-\Delta)^{-1}V_{-}\Im[f] \|_{L^{4}}
\big\}
\\[6pt]
& \lesssim 
\alpha^{\frac{3}{8}-\frac{3}{2q}}
\|(1+|x|)^{-4} f\|_{L^{\frac{12}{11}}}
\\[6pt]
& \lesssim
\alpha^{\frac{3}{8}-\frac{3}{2q}}
\| f\|_{L^{r}}.
\end{split}
\end{equation} 
Observe from \eqref{19/01/31/15:48} 
and Lemma \ref{19/01/31/14:17} that 
$S_{11}$ maps $X_{r}$ into itself. 
Furthermore, it follows from 
\eqref{19/01/31/14:14}, \eqref{19/01/31/16:10} and 
the theory of Neumann's series that for any sufficiently small $\alpha>0$, the operator $1+K_{1}S_{11}$ has the inverse, say $K_{2}$, such that 
\begin{equation}\label{19/01/31/16:05}
\|K_{2} \|_{X_{r}\to X_{r}}\le (1-\|K_{1}S_{11}\|_{X_{r}\to X_{r}})^{-1} \lesssim 1. 
\end{equation}
We summarize: 
\begin{equation}\label{19/01/31/16:10}
A_{11}^{-1}=\varepsilon^{-1} K_{2}K_{1},
\quad 
\|A_{11}^{-1}\|_{X_{r}\to X_{r}} \lesssim \varepsilon^{-1}. 
\end{equation}

\noindent 
{\bf Step 2.}~We consider the operators 
$A_{12} \colon PL^{r}(\mathbb{R}^{3})
\to X_{r}$ and $A_{21} \colon X_{r} \to 
\Pi L^{r}(\mathbb{R}^{3})$. 
It follows from \eqref{19/01/27/18:15}, \eqref{19/01/28/09:53}, \eqref{19/01/28/15:30} and Lemma \ref{19/01/31/17:09} 
that for any $f \in L^{r}(\mathbb{R}^{3})$, 
\begin{equation}\label{19/01/31/16:28}
\begin{split} 
\|
(1-P)G(\alpha) P f \|_{L^{r}}
&= 
\alpha \|(1-P)
(-\Delta+\alpha)^{-1} G_{0} P f \|_{L^{r}}
\\[6pt]
&\lesssim
\alpha \|
(-\Delta+\alpha)^{-1} G_{0} P f \|_{L^{r}}
\lesssim 
\alpha^{\frac{1}{2}-\frac{3}{2r}} \|f\|_{L^{r}},
\end{split}
\end{equation}
so that 
\begin{equation}\label{19/01/31/16:57}
\| A_{12} \|_{P L^{r}\to X_{r}} 
\lesssim 
\alpha^{\frac{1}{2}-\frac{3}{2r}}.
\end{equation}
Similarly, one can verify that 
\begin{equation}\label{19/02/02/09:18}
\| A_{21} \|_{X_{r}\to \Pi L^{r}} 
\lesssim 
\varepsilon \alpha^{\frac{1}{2}}.
\end{equation}

\noindent 
{\bf Step 3.}~We consider 
$A_{22} \colon PL^{r}(\mathbb{R}^{3})
\to \Pi L^{r}(\mathbb{R}^{3})$. 
By \eqref{19/01/27/18:15}
and \eqref{19/01/28/09:53}, 
we see that for any $f \in L^{r}(\mathbb{R}^{3})$, 
\begin{equation} \label{eq7-01}
A_{22} f = - \alpha \Pi (- \Delta + \alpha)^{-1}
G_{0}Pf. 
\end{equation}
Observe from \eqref{eq7-01}, 
\eqref{19/10/9/10:22}, 
Proposition \ref{18/12/02/11:25} 
and Lemma \ref{19/9/30/9:26} that 
\begin{equation}\label{19/01/27/15:37}
\begin{split}
& 
A_{22}f
= 
-\alpha \sum_{j=1}^{3} 
\frac{\langle f, \, V_{+} \partial_{j} W \rangle}{ 
\langle \partial_{j} W, \, V_{+} \partial_{j} W \rangle
\|V_{+}\partial_{j} W\|_{L^{2}}^{2}} 
\langle (-\Delta+\alpha)^{-1}G_{0}\partial_{j} W ,
\, V_{+}\partial_{j} W \rangle V_{+}\partial_{j} W
\\[6pt]
&\quad -\alpha
\frac{\langle f, V_{+}\Lambda W \rangle}
{ \langle \Lambda W,\, V_{+} \Lambda W \rangle
\|V_{+}\Lambda W\|_{L^{2}}^{2}} 
\langle (-\Delta+\alpha)^{-1} G_{0} 
\Lambda W ,\, V_{+}\Lambda W \rangle 
V_{+}\Lambda W
\\[6pt]
&\quad -\alpha
\frac{\langle f, i V_{-} W \rangle}{ \langle i W,\, i V_{-} W \rangle
\|V_{-} W\|_{L^{2}}^{2}} 
\langle (-\Delta+\alpha)^{-1} G_{0} (iW),\, i V_{-} W \rangle 
V_{-} (i W)
\\[6pt]
&= 
\alpha \sum_{j=1}^{3} 
\frac{\langle f, \, V_{+} \partial_{j} W \rangle}{ 
\langle \partial_{j} W, \, V_{+} \partial_{j} W \rangle
\|V_{+}\partial_{j} W\|_{L^{2}}^{2}} 
\langle (-\Delta+\alpha)^{-1}V_{+} \partial_{j} W ,\, \partial_{j} W \rangle V_{+}\partial_{j} W
\\[6pt]
&\quad +\alpha
\frac{\langle f, V_{+}\Lambda W \rangle}
{ \langle \Lambda W,\, V_{+}\Lambda W \rangle
\|V_{+}\Lambda W\|_{L^{2}}^{2}} 
\langle (-\Delta+\alpha)^{-1} V_{+} \Lambda W ,\, 
\Lambda W \rangle 
V_{+}\Lambda W
\\[6pt]
&\quad +
\alpha
\frac{\langle f, i V_{-} W \rangle}
{ \langle i W,\, i V_{-} W \rangle
\|V_{-} W\|_{L^{2}}^{2}} 
\langle (-\Delta+\alpha)^{-1} V_{-}(iW) ,\, i W \rangle 
V_{-} (i W)
\\[6pt]
&=
\alpha 
\sum_{j=1}^{3} c_{j}(\alpha) 
\langle f, \, V_{+} \partial_{j} W \rangle 
V_{+}\partial_{j} W
\\[6pt]
&\quad 
+\alpha \big\{ 
b_{+} \alpha^{-\frac{1}{2}} + c_{+}(\alpha)
\big\}
\langle f, \, V_{+} \Lambda W \rangle V_{+} \Lambda W 
\\[6pt]
&\quad 
+\alpha \big\{ 
b_{-} \alpha^{-\frac{1}{2}} + c_{-}(\alpha)
\big\}
\langle f, \, i V_{-} \Lambda W \rangle V_{-} (i W), 
\end{split} 
\end{equation}
where $b_{+}\neq 0$, $b_{-} \neq 0$ 
are some constants depending 
only on 
$c_{1}(\alpha), c_{2}(\alpha), c_{3}(\alpha), 
c_{+}(\alpha), c_{-}(\alpha)$ are some constants such that 
\begin{equation}\label{19/01/27/17:20}
\sum_{j=1}^{3}|c_{j}(\alpha)| +|c_{+}(\alpha)| + |c_{-}(\alpha)| 
\lesssim 1.
\end{equation} 
Here, using the above constants 
$c_{1}(\alpha), c_{2}(d), c_{3}(\alpha)$ and $b_{+},b_{-}$, we define an operator $Q \colon P L^{r}(\mathbb{R}^{3}) 
\to \Pi L^{r}(\mathbb{R}^{3})$ by 
\begin{equation}\label{19/01/27/17:10}
\begin{split}
Q f
&:=
\alpha \sum_{j=1}^{d} c_{j}(\alpha) 
\langle f, \, V_{+}\partial_{j} W \rangle V_{+} \partial_{j} W
\\[6pt]
&\quad + 
b_{+} \alpha^{\frac{1}{2}} 
\langle f, \, V_{+}\Lambda W \rangle V_{+} \Lambda W
+b_{-}\alpha^{\frac{1}{2}} 
\langle f,\, i V_{-} W \rangle V_{-}(i W ).
\end{split} 
\end{equation}
Furthermore, we define 
$S_{22} \colon P L^{r}(\mathbb{R}^{3}) 
\to \Pi L^{r}(\mathbb{R}^{3})$ by 
\begin{equation}\label{19/01/27/17:15}
S_{22} f
:=
\alpha c_{+}(\alpha) \langle f, \, V_{+}\Lambda W \rangle V_{+} \Lambda W 
+ \alpha c_{-}(\alpha) \langle f,\, i V_{-} W \rangle V_{-} (i W).
\end{equation}
Then, one sees that 
\begin{equation}\label{19/01/28/10:20}
A_{22} = Q + S_{22}.
\end{equation}
Observe that the inverse of 
$Q \colon PL^{r}(\mathbb{R}^{3}) \to \Pi L^{r}(\mathbb{R}^{3})$ 
is given by 
\begin{equation}\label{19/01/27/17:21}
\begin{split}
Q^{-1}~\cdot 
&= 
\alpha^{-1} \sum_{j=1}^{3} 
\frac{ \partial_{j} W }{c_{j}(\alpha) 
\langle \partial_{j}W, 
\, V_{+} \partial_{j} W \rangle^{2} } 
\langle \, \cdot \, , \ \partial_{j} W \rangle
\\[6pt]
&\quad +
\alpha^{-\frac{1}{2}} 
\frac{ \Lambda W }{b_{+} 
\langle \Lambda W,\, V_{+} \Lambda W \rangle^{2} } 
\langle \, \cdot \, , \ \Lambda W \rangle
\\[6pt]
&\quad +
\alpha^{-\frac{1}{2}} 
\frac{ i W }{b_{-} 
\langle i W,\ , i V_{-} W \rangle^{2} } 
\langle \, \cdot \, ,\ i W \rangle.
\end{split} 
\end{equation}
Observe that
$ 
\|Q^{-1}\|_{PL^{r}\to \Pi L^{r}}
\lesssim
\alpha^{-\frac{1}{2}}
$. 
Rewrite $A_{22}$ as 
\begin{equation}\label{19/01/29/17:52}
A_{22} = Q (1+ Q^{-1}S_{22}) .
\end{equation}
Then, it follows from 
$\langle V_{+}\Lambda W, \, \partial_{j}W \rangle =
\langle V_{-} (iW), \, \partial_{j}W \rangle =0$ ($1\le j \le 3$), 
\eqref{19/01/27/17:20}
and \eqref{19/01/27/17:15} that for any 
$f \in P L^{r}(\mathbb{R}^{3})$,
\begin{equation}\label{19/01/31/18:07}
\begin{split}
\|Q^{-1}S_{22} f\|_{L^{r}}
&\lesssim 
\alpha^{-\frac{1}{2}}
| \langle S_{22}f, \Lambda W \rangle |
+
\alpha^{-\frac{1}{2}}
| \langle S_{22}f, i W \rangle |
\lesssim 
\alpha^{\frac{1}{2}}
\|f\|_{L^{r}}.
\end{split} 
\end{equation}
Furthermore, it follows from 
the theory of Neumann's series that 
for any sufficiently small $\alpha>0$, 
the operator $1+Q^{-1}S_{22}$ has the inverse, 
say $L$, such that 
\begin{equation}\label{19/01/31/18:17}
\|L \|_{P L^{r} \to PL^{r}} \le (1-\|Q^{-1}
S_{22}\|_{PL^{r}\to PL^{r}}) \lesssim 1. 
\end{equation}
We summarize: 
\begin{equation}\label{19/02/02/09:33}
A_{22}^{-1}=LQ^{-1}, 
\qquad 
\|A_{22}^{-1}\|_{\Pi L^{r} \to P L^{r}} 
\lesssim 
\alpha^{-\frac{1}{2}}. 
\end{equation}

\noindent 
{\bf Step 4.}~We shall finish the proof. Put 
\begin{equation}\label{19/02/02/09:39}
\widetilde{A}_{\varepsilon}
:=
\left( \begin{array}{cc} 
1 & A_{11}^{-1}A_{12} 
\\ 
A_{22}^{-1}A_{21} & 1 
\end{array} \right),
\end{equation}
so that 
\begin{equation}\label{19/02/02/09:45}
A_{\varepsilon}(\alpha)
=
\left( \begin{array}{cc} 
A_{11} & 0 
\\ 
0& A_{22} 
\end{array} \right)
\widetilde{A}_{\varepsilon}.
\end{equation}
We see from \eqref{19/01/31/16:10}, 
\eqref{19/02/02/09:33}, \eqref{19/01/31/16:57} and \eqref{19/02/02/09:18} that 
\begin{equation}\label{19/02/02/09:50}
\begin{split}
\| 1- \widetilde{A}_{\varepsilon}\|_{\mathbf{X}_{r} 
\to \mathbf{X}_{r}}
&\le 
\|A_{11}^{-1} A_{12}\|_{P L^{r} \to X_{r}}
+
\|A_{22}^{-1} A_{21}\|_{X_{r} \to P L^{r}}
\\[6pt]
&\lesssim
\varepsilon^{-1} \|A_{12}\|_{P L^{r} \to X_{r}}
+
\alpha^{-\frac{1}{2}}
\|A_{21}\|_{X_{r} \to \Pi L^{r}}
\\[6pt]
&\lesssim
\varepsilon^{-1} \alpha^{\frac{1}{2}-\frac{3}{2r}}
+\varepsilon.
\end{split} 
\end{equation}
Hence, we find that there exist 
$\alpha_{*}\in (0,e_{0})$ and $\varepsilon_{0}\in (0,1)$, 
both depending only on $r$, 
such that for any $0<\alpha < \alpha_{*}$, 
\begin{equation}\label{19/02/02/10:08}
\| 1- \widetilde{A}_{\varepsilon_{0}}
\|_{\mathbf{X}_{r} \to \mathbf{X}_{r}} 
\le \frac{1}{2}. 
\end{equation}
Furthermore, the theory of Neumann's series 
shows that 
$\widetilde{A}_{\varepsilon_{0}}$ 
has the inverse 
$D:=\widetilde{A}_{\varepsilon_{0}}^{-1}$ satisfying 
\begin{equation}\label{19/02/02/10:12}
\|D\|_{\mathbf{X}_{r} \to \mathbf{X}_{r}} \le 2.
\end{equation}
Then, for any $0<\alpha < \alpha_{*}$, 
the inverse of $A_{\varepsilon_{0}}(\alpha)$ exists and is given by 
\begin{equation}\label{19/02/02/10:37}
A_{\varepsilon_{0}}(\alpha)^{-1}
=
\left(\begin{array}{cc} D_{11} & D_{12} \\ D_{21} & D_{22} \end{array} \right)
\left(\begin{array}{cc} A_{11}^{-1} & 0 \\ 0 & A_{22}^{-1} \end{array} \right)
,
\end{equation}
where $D_{jk}$ is the $(j,k)$-entry of $D$. 
Here, we may assume that 
$\varepsilon_{0}^{-1} \le \alpha_{*}^{-\frac{1}{2}}$. 
Then, we see from \eqref{19/01/31/16:10}, 
\eqref{19/02/02/09:33} and 
\eqref{19/02/02/10:12} that for any $0<\alpha < \alpha_{*}$, 
\begin{equation}\label{19/02/02/10:55}
\|A_{\varepsilon_{0}}(\alpha)^{-1} 
\|_{\mathbf{Y}_{r} \to \mathbf{X}_{r}} 
\lesssim 
\|A_{11}^{-1} \|_{X_{r} \to X_{r}}
+
\|A_{22}^{-1} \|_{\Pi L^{r} \to P L^{r}}
\lesssim \alpha^{-\frac{1}{2}}.
\end{equation}
Thus, as mentioned in \eqref{19/01/30/17:18}, 
for any $0< \alpha < \alpha_{*}$, 
the inverse of $G(\alpha)$ exists 
as an operator from $L^{r}(\mathbb{R}^{3})$ to itself and 
\begin{equation}\label{19/02/02/11:07}
G(\alpha)^{-1} 
= B_{\varepsilon_{0}} 
A_{\varepsilon_{0}}(\alpha)^{-1}\mathbf{C}.
\end{equation}
Observe rom \eqref{19/01/30/16:03}, 
\eqref{19/01/30/17:02}, 
\eqref{19/01/28/15:30} and $\varepsilon_{0}<1$ 
that 
\begin{equation}\label{19/02/02/11:10}
\|B_{\varepsilon_{0}}\|_{\mathbf{X}_{r}\to L^{r}}\lesssim 1, 
\qquad 
\|\mathbf{C}\|_{L^{r} \to \mathbf{Y}_{r}}\lesssim 1.
\end{equation}
Hence, we see from \eqref{19/02/02/10:55} 
through \eqref{19/02/02/11:10} that 
for any $0<\alpha <\alpha_{*}$, 
\begin{equation}\label{19/02/02/11:20}
\|
G(\alpha)^{-1} 
\|_{L^{r}\to L^{r}}
\lesssim \alpha^{-\frac{1}{2}},
\end{equation}
which proves the estimate \eqref{res-est-G}. In order to prove \eqref{res-est-G2}, we assume $f\in X_{r}$. Then $\Pi f =0$ and therefore 
\begin{equation}\label{19/02/02/11:21}
\begin{split}
G(\alpha)^{-1}f 
&= 
B_{\varepsilon_{0}} \left(\begin{array}{cc} D_{11} & D_{12} \\ D_{21} & D_{22} \end{array} \right)
\left(\begin{array}{cc} A_{11}^{-1} & 0 \\ 0 & A_{22}^{-1} \end{array} \right)
\left( \begin{array}{c}
(1-P)f
\\
0
\end{array} \right)
\\[6pt]
&=
D_{11}A_{11}^{-1}(1-P)f 
+D_{21}A_{11}^{-1}(1-P)f.
\end{split} 
\end{equation}
Hence, by \eqref{19/02/02/11:21}, \eqref{19/01/31/16:10} with $\varepsilon=\varepsilon_{0}$, and \eqref{19/02/02/10:12}, we can obtain the desired estimate \eqref{res-est-G2}. Thus, we have completed the proof. 
\end{proof}

\subsection{Proof of Theorem \ref{thm-imp}}\label{19/10/9/11:28}

In this section, we prove Theorem 
\ref{thm-imp}. 
For $\alpha>0$, we introduce the notation
\begin{equation}
\label{19/9/21/17:01}
Z_{\alpha, j}:=
\left\{ \begin{array}{ccc}
G(\alpha)^{*} V_{+}\partial_{j} W 
&\mbox{if} & 1\le j\le 3,
\\[6pt]
G(\alpha)^{*} V_{+}\Lambda W 
&\mbox{if} & j= 4,
\\[6pt]
G(\alpha)^{*} (i V_{-} W )
&\mbox{if} & j= 5.
\end{array} \right. 
\end{equation}
It follows from the definition 
of $G(\alpha)^{*}$ (see \eqref{def-G*}), 
$\Im [V_{+}\nabla W ]=0$ and 
\eqref{19/10/9/10:22} that if $1\le j\le 3$, then 
\begin{equation}\label{19/9/27/13:55}
\begin{split}
Z_{\alpha,j}
&=
(1+V_{+}(-\Delta + \alpha)^{-1}) V_{+}\partial_{j} W
\\[6pt]
&=
V_{+}\partial_{j} W+ V_{+}(-\Delta +\alpha)^{-1} 
\{-(-\Delta+\alpha)+\alpha \}\partial_{j}W
\\[6pt]
&=
\alpha V_{+}(-\Delta + \alpha)^{-1} \partial_{j}W
.
\end{split}
\end{equation}
Similarly, one can verify that 
\begin{align}
\label{19/9/27/14:00}
Z_{\alpha, 4}&:=\alpha V_{+} (-\Delta +\alpha)^{-1} \Lambda W, 
\\[6pt]
\label{19/9/30/18:18}
Z_{\alpha, 5}&:= - \alpha V_{-}(-\Delta + \alpha)^{-1} (i W).
\end{align}

For the proof of Theorem 
\ref{thm-imp}, we will use the following fact:
\begin{lemma}\label{19/9/26/11:51}
Let $\alpha>0$. Then, we have 
\begin{align}
\label{19/9/26/10:33}
\langle \partial_{j} W, \, Z_{\alpha, j} \rangle
&=O(\alpha) 
\quad
\mbox{for any $1\le j \le 3$}, 
\\[6pt]
\label{19/9/26/11:55}
\langle \Lambda W, \, Z_{\alpha, 4} \rangle
&=
O(\alpha^{\frac{1}{2}}),
\\[6pt]
\label{19/9/26/11:56}
\langle i W, \, Z_{\alpha, 5} \rangle 
&= 
O(\alpha^{\frac{1}{2}}).
\end{align}
\end{lemma}

\begin{proof}[Proof of Lemma \ref{19/9/26/11:51}]
We shall prove the claim 
\eqref{19/9/26/10:33}. 
By \eqref{19/9/27/13:55}, 
one can see that if $1\le j \le 3$, then 
\begin{equation}\label{19/9/28/17:8}
\begin{split}
\langle \partial_{j} W, \, Z_{\alpha, j} \rangle
&=
\langle \partial_{j} W, \, \alpha V_{+} (-\Delta + \alpha)^{-1} \partial_{j}W 
\rangle
=
\alpha \langle (-\Delta +\alpha)^{-1} V_{+}\partial_{j} W, \, \partial_{j} W 
\rangle.
\end{split} 
\end{equation}
This together Lemma \ref{19/9/30/9:26} 
shows \eqref{19/9/26/10:33}. 

Next, we shall prove \eqref{19/9/26/11:55}. 
By \eqref{19/9/27/14:00}, one sees that 
\begin{equation}\label{19/9/26/11:57}
\langle \Lambda W, \, Z_{\alpha, 4} \rangle
=
\alpha \langle (-\Delta + \alpha)^{-1} V_{+}\Lambda W, \, \Lambda W 
\rangle.
\end{equation}
From this together with Corollary \ref{18/12/02/11:25cor}, 
we obtain the desired result \eqref{19/9/26/11:55}. 
Similarly, we can prove \eqref{19/9/26/11:56}. 
\end{proof}

Now, we give a proof of Theorem \ref{thm-imp}:
\begin{proof}[Proof of Theorem \ref{thm-imp}]
The proof basically follows the standard proof 
of the implicit function theorem. 
The main difference is to show that the radii 
$\delta_{\rm geo}$ and $r_{\rm geo}$ are 
independent of the parameter $\alpha$. 
\par 
First, we define the functions $\mathscr{H}_{\alpha,1}, \ldots, \mathscr{H}_{\alpha, 5}$ from $\dot{H}^{1}(\mathbb{R}^{3}) \times \mathbb{R}^{3} \times (0,\infty) \times (-\pi, \pi)$ to $\mathbb{R}$ by
\begin{equation}\label{18/03/29/14:05}
\mathscr{H}_{\alpha,j}(u, y, \nu, \theta)
=
\langle 
T_{(y,\nu, \theta)}[u]-W,\, Z_{\alpha, j} 
\rangle
\quad 
\mbox{for $1\le j \le 5$}. 
\end{equation} 
To make the representation simple, we introduce the following notation: 
\begin{align}
\label{19/9/26/7:2}
B&:=\mathbb{R}^{3}\times (0,\infty) 
\times (-\frac{\pi}{2},\frac{\pi}{2}),
\qquad 
\mathbf{b}:=(y,\nu,\theta), 
\qquad 
\mathbf{b}_{0}:=(0,1,0).
\end{align}
Then, define the function 
$\mathscr{H}_{\alpha} \colon \dot{H}^{1}(\mathbb{R}^{3}) \times B \to \mathbb{R}^{5}$ by
\begin{equation}\label{18/02/24/16:47}
\mathscr{H}_{\alpha}(u, \mathbf{b})
:=\bigm( 
\mathscr{H}_{\alpha, 1}(u, \mathbf{b}),\ \ldots, \ 
\mathscr{H}_{\alpha, 5}(u,\mathbf{b})
\bigm).
\end{equation}
Observe that for any $(u,\mathbf{b})=(u,y,\nu, \theta)\in \dot{H}^{1}(\mathbb{R}^{3}) \times 
B$ with $y=(y_{1}, y_{2}, y_{3})$, and any $1\le j \le 5$: 
\begin{align}
\label{18/02/24/17:13}
\partial_{y_{k}} \mathscr{H}_{\alpha, j}(u, \mathbf{b})
&=
\langle T_{\mathbf{b}}[\partial_{k}u], \, Z_{\alpha, j} \rangle 
\qquad 
\mbox{for $1\le k \le 3$},
\\[6pt]
\label{19/9/25/23:35}
\partial_{\nu} \mathscr{H}_{\alpha, j}(u, \mathbf{b})
&= 
-\nu^{-1} 
\langle 
T_{\mathbf{b}}[u], Z_{\alpha,j}
\rangle
-2\nu^{-1} 
\langle
\nu^{-1} e^{i\theta} \nu^{-2}x \cdot
\nabla u (\nu^{-2}x+y), Z_{\alpha, j}
\rangle,
\\[6pt]
\label{18/02/24/17:12}
\partial_{\theta} \mathscr{H}_{\alpha, j}(u, \mathbf{b})
&=\langle T_{\mathbf{b}}[i u],\, Z_{\alpha, j} \rangle, 
\end{align}
In particular, by \eqref{18/02/24/17:13}, \eqref{19/9/25/23:35}, \eqref{18/02/24/17:12} and Lemma \ref{19/9/26/11:51}, one sees that 
\begin{align}
\label{19/9/25/11:21}
\partial_{y_{j}} \mathscr{H}_{\alpha, j}(W,\mathbf{b}_{0})
&=
\langle \partial_{j}W, \, Z_{\alpha, j} \rangle 
=
O(\alpha)
\qquad 
\mbox{for $1\le j \le 3$},
\\[6pt]
\label{19/9/25/20:2}
\partial_{\nu} \mathscr{H}_{\alpha, 4}(W,\mathbf{b}_{0})
&=
-2\langle \Lambda W,\, Z_{\alpha, 4} \rangle
=
O(\alpha^{\frac{1}{2}})
,
\\[6pt]
\label{19/9/25/11:24}
\partial_{\theta} \mathscr{H}_{\alpha, 5}(W,\mathbf{b}_{0})
&
=
\langle i W,\, Z_{\alpha, 5} \rangle
=
O(\alpha^{\frac{1}{2}}).
\end{align}
Moreover, by \eqref{19/9/27/13:55} through 
\eqref{19/9/30/18:18}, 
we can verify that the off-diagonal components of $\partial_{\mathbf{b}}\mathscr{H}_{\alpha}(W,\mathbf{b}_{0})$ vanish, 
namely, for any $1\le j \le 5$, 
\begin{align}
\label{19/9/27/11:16}
\partial_{y_{k}} \mathscr{H}_{\alpha, j}(W,\mathbf{b}_{0})
=
\langle \partial_{k}W, \, Z_{\alpha, j} \rangle 
=0
\quad 
&\mbox{if $j \neq k$ \ ($1\le k \le 3$)},
\\[6pt]
\label{19/9/25/20:1}
\partial_{\nu} \mathscr{H}_{\alpha, j}(W,\mathbf{b}_{0})
=
-2\langle \Lambda W,\, Z_{\alpha, j} \rangle
=0
\quad 
&\mbox{if $j \neq 4$},
\\[6pt]
\label{19/9/25/11:23}
\partial_{\theta} \mathscr{H}_{\alpha, j}(W,\mathbf{b}_{0})
=\langle i W,\, Z_{\alpha, j} \rangle
=0
\quad 
&\mbox{if $j \neq 5$}.
\end{align}

Put 
\begin{equation}\label{19/9/27/11:40} 
A_{\alpha}:=\partial_{\mathbf{b}} \mathscr{H}_{\alpha}(W,\mathbf{b}_{0}). 
\end{equation}
Then, it follows from \eqref{19/9/25/11:21} through \eqref{19/9/25/11:23} that the differential mapping 
$A_{\alpha} \colon \mathbb{R}^{5}\to \mathbb{R}^{5}$ 
is bijective, and the inverse $A_{\alpha}^{-1}$ is of the form 
\begin{equation}\label{19/9/30/11:21}
A_{\alpha}^{-1}
=
\left( \begin{array}{cccccc}
O(\alpha^{-1}) 
& 0
& 0
& 0 
& 0 
\\[6pt]
0
& O(\alpha^{-1}) 
& 0 
&0 
&0
\\[6pt]
0
& 0
& O(\alpha^{-1}) 
& 0
& 0
\\[6pt]
0
& 0
&0
& O(\alpha^{-\frac{1}{2}}) 
& 0
\\[6pt]
0
& 0
&0
& 0
& O(\alpha^{-\frac{1}{2}})
\end{array} \right).
\end{equation}
As well as the standard proof of the implicit function theorem, we consider the mapping $\mathscr{G}_{\alpha}\colon 
H^{1}(\mathbb{R}^{3}) \times B\to \mathbb{R}^{5}$ defined by 
\begin{equation}\label{19/9/25/14:41}
\mathscr{G}_{\alpha}(u, \mathbf{b})
:=
\mathbf{b} -A_{\alpha}^{-1} \mathscr{H}_{\alpha}(u,\mathbf{b}).
\end{equation}
Note that 
$\mathbf{b}=\mathscr{G}_{\alpha}(u,\mathbf{b})$ implies that 
$H_{\alpha}(u,\mathbf{b})=0$. 
Hence, for the desired result, 
it suffices to prove that: 
there exist $0< \delta_{\rm geo}<1$ 
and $0< r_{\rm geo}<\frac{1}{2}$ 
such that for any $0< \alpha <1$, 
there exists a unique continuous mapping 
$\mathbf{b}_{\alpha} \colon B_{\dot{H}^{1}}(W,\delta_{\rm geo}) \to B(\mathbf{b}_{0},r_{\rm geo})$ such that 
\begin{align}
\label{19/10/5/21:11}
\mathbf{b}_{\alpha}(W)&=\mathbf{b}_{0},
\\[6pt]
\label{19/10/5/21:12}
\mathbf{b}_{\alpha}(u)& 
=\mathscr{G}_{\alpha}(u, \mathbf{b}_{\alpha}(u)).
\end{align}

We shall prove this claim in several steps. 
Throughout the proof, 
all of the implicit constants depend 
only on 
the fixed dimension $d=3$. 
\\
\noindent 
{\bf Step 1.}~Let $0< \delta <1$ and $0< r <\frac{1}{2}$. Furthermore, let $u\in B_{\dot{H}^{1}}(W,\delta)$ and $\mathbf{b}=(y,\nu,\theta) \in B(\mathbf{b}_{0},r)$. Then, we shall prove the following estimates: 
\begin{align}
\label{19/10/2/1:1}
\sup_{1\le k \le 3}
\| T_{\mathbf{b}}[\partial_{k}u]-\partial_{k}W \|_{L^{2}}
&\lesssim 
\delta + r, 
\\[6pt]
\label{19/10/2/1:3}
\| T_{\mathbf{b}}[u]-W \|_{L^{6}}
&\lesssim 
\delta +r ,
\\[6pt]
\label{19/10/2/1:2}
\| (1+|x|)^{-1} 
\big\{ 
\nu^{-1} e^{i\theta} 
\nu^{-2}x \cdot \nabla u(\nu^{-2}x+y)
- x\cdot \nabla W \big\} \|_{L^{2}}
&\lesssim 
\delta + r.
\end{align}

Let us begin by proving \eqref{19/10/2/1:1}. 
It follows from elementary computations, 
the fundamental theorem of calculus, 
$|y|+|\nu-1|+|\theta|<r<\frac{1}{2}$
and $0<\delta <1$ that 
\begin{equation}\label{19/10/1/15:31}
\begin{split}
\| T_{\mathbf{b}}[\partial_{k}u]-\partial_{k}W \|_{L^{2}}
&\le 
\| e^{i\theta} \nu^{-1}\partial_{k}u(\nu^{-2}\cdot + y)
-
e^{i\theta} \nu^{-1}\partial_{k}W(\nu^{-2}\cdot+y)
\|_{L^{2}}
\\[6pt]
&\quad +
\|
e^{i\theta} \nu^{-1}\partial_{k}W(\nu^{-2}\cdot+y)
-
e^{i\theta} \nu^{-1}\partial_{k}W(\nu^{-2}\cdot)
\|_{L^{2}}
\\[6pt]
&\quad +
\|
e^{i\theta} \nu^{-1}\partial_{k}W(\nu^{-2}\cdot)
-
e^{i\theta} \partial_{k}W 
\|_{L^{2}}
\\[6pt]
&\quad +
\|e^{i\theta} \partial_{k}W -\partial_{k}W \|_{L^{2}}
\\[6pt]
&\le 
\nu^{2} \| \partial_{k}u -\partial_{k}W\|_{L^{2}}
+
\nu^{2} \| \int_{0}^{1} y\cdot \nabla \partial_{k}
W(\cdot + \kappa y)\,d\kappa \|_{L^{2}}
\\[6pt]
&\quad +
\|\int_{1}^{\nu} 2\lambda^{-1}
T_{\lambda}[\Lambda \partial_{k}W]\,d\lambda \|_{L^{2}}
+ \|\int_{0}^{\theta}e^{it}\,dt \ \partial_{k}W\|_{L^{2}}
\\[6pt]
&\le 4 \delta + 4 r \| \nabla \partial_{k}W\|_{L^{2}}
+ 4 r \|\Lambda \partial_{k} W \|_{L^{2}}
+r \|\partial_{k}W \|_{L^{2}}
\lesssim \delta + r.
\end{split} 
\end{equation}
Thus, we have obtained the 
estimate \eqref{19/10/2/1:1}. 
Similarly, we can prove \eqref{19/10/2/1:3}. 
It remains to prove \eqref{19/10/2/1:2}. 
By the triangle inequality, one has 
\begin{equation}\label{19/10/2/9:10}
\begin{split}
&
\|(1+|x|)^{-1}\big\{ 
\nu^{-1} e^{i\theta} 
\nu^{-2}x\cdot \nabla u (\nu^{-2}x+y) 
- x\cdot \nabla W \big\} \|_{L^{2}}
\\[6pt]
&\le 
|\nu^{-1}-1| \|(1+|x|)^{-1} \nu^{-2}x \cdot \nabla u (\nu^{-2}x+y)
\|_{L^{2}}
\\[6pt]
& \quad +
\|(1+|x|)^{-1}\big\{e^{i\theta} \nu^{-2}x 
\cdot \nabla u (\nu^{-2}x+y) - x\cdot \nabla W \big\} \|_{L^{2}}.
\end{split} 
\end{equation}
We consider the first term on the right-hand side of \eqref{19/10/2/9:10}.
By $|\nu-1|<r <\frac{1}{2}$, one can verify that 
for any $x\in \mathbb{R}^{3}$, 
\begin{equation}\label{19/10/2/9:45}
\begin{split}
&
|\nu^{-1}-1| 
\|(1+|x|)^{-1} \nu^{-2}x \cdot \nabla u (\nu^{-2}x+y)
\|_{L^{2}}
\\[6pt]
&=
|\nu^{-1}-1| \nu^{3} 
\|(1+|\nu^{2} x|)^{-1} (x\cdot \nabla u (x+y))
\|_{L^{2}}
\\[6pt]
&\lesssim |\nu -1| \|\nabla u\|_{L^{2}}
\lesssim 
r \|\nabla u -\nabla W\|_{L^{2}}
+ r \|\nabla W\|_{L^{2}} 
\le r \delta + r 
\le \delta+r. 
\end{split}
\end{equation}
Next, we consider the second term on the right-hand side of \eqref{19/10/2/9:10}. 
By computations similar to \eqref{19/10/2/9:45}, 
\eqref{19/10/1/15:31} and $|\nu^{-2}-1| \lesssim |\nu -1|<r$, 
one can see that 
\begin{equation}\label{19/10/4/8:40}
\begin{split}
&
\|(1+|x|)^{-1} \big\{e^{i\theta} 
\nu^{-2}x \cdot 
\nabla u (\nu^{-2}x+y) - x\cdot \nabla W 
\big\} \|_{L^{2}}
\\[6pt]
&\lesssim 
\|(1+|x|)^{-1}e^{i\theta}
\nu^{-2}x\cdot
\big\{ 
\nabla u (\nu^{-2}x+y)
- \nabla W(\nu^{-2}x+y)
\big\} \|_{L^{2}}
\\[6pt]
&\quad + 
\| (1+|x|)^{-1}
\nu^{-2}x \cdot \{ \nabla W(\nu^{-2}x+y) -
\nabla W(\nu^{-2}x) \} \|_{L^{2}}
\\[6pt]
&\quad + 
\|
1+|x|)^{-1}
\big\{\nu^{-2}x\cdot \nabla W(\nu^{-2}x)
- x\cdot \nabla W\big\}
\|_{L^{2}}
\\[6pt]
&\quad + 
\|(1+|x|)^{-1}
(e^{i\theta}-1)
x\cdot \nabla W \|_{L^{2}}
\\[6pt]
&\lesssim 
\| \nabla u -\nabla W\|_{L^{2}}
\\[6pt]
&\quad 
+
\nu^{2} \| \int_{0}^{1} 
\sum_{j=1}^{3}y_{j} \partial_{j} 
\nabla W(\nu^{-2}x+\kappa y)
\,d\kappa
\|_{L^{2}}
\\[6pt]
&\quad +
|\nu^{-2}-1| \nu^{3} \| \nabla W \|_{L^{2}}
+ \|\int_{1}^{\nu^{-2}} 
\sum_{j=1}^{3} x_{j}\cdot 
\partial_{j}\nabla W(\lambda x) 
\,d\lambda
\|_{L^{2}} 
\\[6pt]
& \quad 
+ \|(1+|x|)^{-1}\int_{0}^{\theta}e^{it}\,dt 
\ x\cdot \nabla W\|_{L^{2}}
\\[6pt]
&\lesssim 
\delta + r.
\end{split} 
\end{equation} 
Putting \eqref{19/10/2/9:10}, \eqref{19/10/2/9:45} and \eqref{19/10/4/8:40} together, we obtain the desired estimate \eqref{19/10/2/1:2}.

\noindent 
{\bf Step 2.}~We shall show that there exist 
$0< \delta_{\rm geo} <1$ and 
$0< r_{\rm geo}<\frac{1}{2}$ 
such that for any $0<\alpha < 1$, 
any $u \in B_{\dot{H}^{1}}(W,\delta_{\rm geo})$, and any $\mathbf{b},\mathbf{c} \in B(\mathbf{b}_{0},r_{\rm geo})$, 
\begin{align}
\label{19/9/25/14:39}
\big| 
\mathscr{G}_{\alpha}(u,\mathbf{b}) 
- \mathscr{G}_{\alpha}(u, \mathbf{c}) 
\big|
&\le 
\frac{1}{4} 
| \mathbf{b}-\mathbf{c} |, 
\\[6pt]
\label{19/10/5/11:15}
\big| 
\mathscr{G}_{\alpha}(u,\mathbf{b}) -\mathbf{b}_{0}
\big|
&\le \frac{r_{\rm geo}}{2}.
\end{align}
Let $\alpha>0$, $u \in \dot{H}^{1}(\mathbb{R}^{3})$ 
and $\mathbf{b}, \mathbf{c} \in 
B:=\mathbb{R}^{3}\times (0,\infty) 
\times (-\frac{\pi}{2},\frac{\pi}{2})$. 
Then, observe from \eqref{19/9/27/11:40} that 
\begin{equation}\label{19/9/25/17:02}
\begin{split}
&
\mathscr{G}_{\alpha}(u, \mathbf{b}) 
- \mathscr{G}_{\alpha}(u, \mathbf{c}) 
\\[6pt]
&=
\mathbf{b} 
- \mathbf{c} 
- A_{\alpha}^{-1}
\big\{
\mathscr{H}_{\alpha}(u,\mathbf{b})
- \mathscr{H}_{\alpha}(u,\mathbf{c})
\big\}
\\[6pt]
&=A_{\alpha}^{-1}A_{\alpha}(\mathbf{b}-\mathbf{c})
-A_{\alpha}^{-1} 
\int_{0}^{1} \partial_{\mathbf{b}} \mathscr{H}_{\alpha}
(u, \mathbf{c} + t (\mathbf{b}-\mathbf{c})) (\mathbf{b}-\mathbf{c}) 
\,dt 
\\[6pt]
&=
A_{\alpha}^{-1} 
\int_{0}^{1}
\Bigm\{
\partial_{\mathbf{b}} \mathscr{H}_{\alpha}(W,\mathbf{b}_{0})
- \partial_{\mathbf{b}} \mathscr{H}_{\alpha}
(u, \mathbf{c} + t (\mathbf{b}-\mathbf{c}))
\Bigm\}(\mathbf{b}-\mathbf{c}) 
\,dt.
\end{split} 
\end{equation}
Let $0<\delta <1$ and $0<r <\frac{1}{2}$
be constants to be specified later. 
Then, observe from \eqref{19/9/30/11:21} that 
for any $u \in B_{\dot{H}^{1}}(W,\delta)$ 
and $\mathbf{b}=(y,\nu,\theta)
\in B(\mathbf{b}_{0},r)$, we have 
\begin{equation}\label{19/9/25/17:07}
\begin{split}
&
\bigm\| 
A_{\alpha}^{-1}\bigm\{ 
\partial_{\mathbf{b}} \mathscr{H}_{\alpha}(W,\mathbf{b}_{0})
- \partial_{\mathbf{b}} \mathscr{H}_{\alpha}
(u, \mathbf{b})
\bigm\} 
\bigm\|_{\mathbb{R}^{5}\to \mathbb{R}^{5}}
\\[6pt]
&\lesssim 
\sum_{1\le j\le 3} \sum_{1\le k \le 3}
\alpha^{-1}
\big|
\partial_{y_{k}} \mathscr{H}_{\alpha,j}(W,\mathbf{b}_{0})
-
\partial_{y_{k}} \mathscr{H}_{\alpha,j}
(u, \mathbf{b})
\big|
\\[6pt]
&\quad + \sum_{1\le j\le 3} 
\alpha^{-1}
\big|
\partial_{\nu} \mathscr{H}_{\alpha,j}(W,\mathbf{b}_{0})
- \partial_{\nu} \mathscr{H}_{\alpha,j}
(u, \mathbf{b})
\big|
\\[6pt]
&\quad + \sum_{1\le j\le 3} 
\alpha^{-1}
\big|
\partial_{\theta} \mathscr{H}_{\alpha,j}(W,\mathbf{b}_{0})
- \partial_{\theta} \mathscr{H}_{\alpha,j}
(u, \mathbf{b}) \big|
\\[6pt]
&\quad +
\sum_{j= 4, 5} 
\sum_{1\le k \le 3}
\alpha^{-\frac{1}{2}}
\big| \partial_{y_{k}} \mathscr{H}_{\alpha,j}(W,\mathbf{b}_{0})
- \partial_{y_{k}} \mathscr{H}_{\alpha,j}
(u, \mathbf{b})
\big|
\\[6pt]
&\quad 
+
\sum_{j=4,5}
\alpha^{-\frac{1}{2}}
\big|
\partial_{\nu} \mathscr{H}_{\alpha,j}(W,\mathbf{b}_{0})
- \partial_{\nu} \mathscr{H}_{\alpha,j}
(u, \mathbf{b}) \big|
\\[6pt]
&\quad 
+
\sum_{j=4,5}
\alpha^{-\frac{1}{2}}
\big|
\partial_{\theta} \mathscr{H}_{\alpha,j}(W,\mathbf{b}_{0})
-
\partial_{\theta} \mathscr{H}_{\alpha,j}
(u, \mathbf{b})
\big|.
\end{split} 
\end{equation}

\noindent 
{\bf Step 2.1.}~Consider the first term on the right-hand side of \eqref{19/9/25/17:07}. 
By \eqref{18/02/24/17:13}, \eqref{19/9/25/11:21}, 
\eqref{19/9/27/11:16}, \eqref{19/9/27/13:55}, 
Lemma \ref{19/9/26/15:4} 
and \eqref{19/10/2/1:1}, 
one can see that 
\begin{equation}\label{19/9/27/13:45}
\begin{split} 
&
\sum_{1\le j \le 3} \sum_{1\le k\le 3}
\alpha^{-1} \big|
\partial_{y_{k}} \mathscr{H}_{\alpha, j}(u, \mathbf{b}_{0}) 
-
\partial_{y_{k}} \mathscr{H}_{\alpha, j}(W, \mathbf{b})
\big|
\\[6pt]
&=
\sum_{1\le j\le 3}\sum_{1\le k \le 3}
\alpha^{-1} \big|
\langle T_{\mathbf{b}}[\partial_{k}u]-\partial_{k}W, \, Z_{\alpha, j} 
\rangle
\big|
\\[6pt]
&=
\sum_{1\le j\le 3}
\sum_{1\le k \le 3}
\big|
\langle 
(-\Delta +\alpha)^{-1} V_{+}
\{ T_{\mathbf{b}}[\partial_{k}u]-\partial_{k}W\}, \, \partial_{j}W 
\rangle
\big|
\\[6pt]
&\lesssim 
\sup_{1\le k \le 3}\| T_{\mathbf{b}}
[\partial_{k}u]-\partial_{k}W \|_{L^{2}}
\lesssim 
\delta +r.
\end{split} 
\end{equation}

\noindent 
{\bf Step 2.2.}~Consider the second term on the right-hand side of \eqref{19/9/25/17:07}. 
Observe from \eqref{19/9/27/13:55} that 
\begin{equation}\label{Mar5-1}
\langle W, Z_{\alpha,j} \rangle=0
\quad 
\mbox{for all $1\le j \le 3$}.
\end{equation}
By \eqref{19/9/25/23:35}, \eqref{Mar5-1}, 
\eqref{19/9/25/20:1}, 
\eqref{19/9/27/13:55}, Lemma \ref{19/9/26/15:4}, 
$|\nu^{-1}-1| \lesssim |\nu-1| <r$, 
\eqref{19/10/2/1:3} and \eqref{19/10/2/1:2}, 
one can see that 
\begin{equation}\label{19/10/1/16:21}
\begin{split} 
&
\sum_{1 \le j\le 3}\alpha^{-1} 
\big|
\partial_{\nu} \mathscr{H}_{\alpha,j}(W,\mathbf{b}_{0})
-
\partial_{\nu} \mathscr{H}_{\alpha,j}
(u, \mathbf{b})
\big|
\\[6pt]
&\le 
2\sum_{1\le j\le 3} \alpha^{-1} 
\big|
\langle \nu^{-1} T_{\mathbf{b}}[u] - W, \, Z_{\alpha, j} 
\rangle
\big|
\\[6pt]
&\quad +
2 \sum_{1\le j \le 3}\alpha^{-1}
|\nu^{-1} - 1| 
\big| \langle 
\nu^{-1} e^{i\theta} \nu^{-2}x 
\cdot \nabla u (\nu^{-2}x+y), 
Z_{\alpha,j}
\rangle \rangle
\\[6pt]
&\quad +
2\sum_{1\le j \le 3} \alpha^{-1}
\big| \langle \nu^{-1} e^{i\theta} \nu^{-2}x\cdot \nabla u (\nu^{-2}x+y)
-
\Lambda W, Z_{\alpha,j} 
\rangle \big|
\\[6pt]
&\lesssim 
\| \nu^{-1} \big\{ T_{\mathbf{b}}[u]- W\big\} \|_{L^{6}}
+
\| (\nu^{-1} -1) W \|_{L^{6}}
\\[6pt]
&\quad +
|\nu^{-1}-1|
\| (1+|x|)^{-1} \nu^{-2}|x| |\nabla u (\nu^{-2}x+y)|\|_{L^{2}}
\\[6pt]
&\quad +
\|(1+|x|)^{-1}\big\{\nu^{-1}e^{i\theta}
\nu^{-2}x \cdot \nabla u (\nu^{-2}x+y) - 
x\cdot \nabla W \big\}\|_{L^{2}}
\\[6pt]
&\lesssim \delta +r.
\end{split} 
\end{equation}

\noindent 
{\bf Step 2.3.}~Consider the third term on the right-hand side of \eqref{19/9/25/17:07}. 
By \eqref{18/02/24/17:12}, \eqref{19/9/25/11:23}, 
\eqref{19/9/27/13:55}, 
Proposition \ref{19/9/26/15:4} 
and \eqref{19/10/2/1:3} that 
\begin{equation}\label{19/10/1/16:55}
\begin{split} 
&
\sum_{1 \le j\le 3}
\alpha^{-1} \big|
\partial_{\theta} \mathscr{H}_{\alpha,j}(W,\mathbf{b}_{0})
-
\partial_{\theta} \mathscr{H}_{\alpha,j}
(u, \mathbf{b})
\big|
=
\sum_{1\le j\le 3}
\alpha^{-1} \big|
\langle T_{\mathbf{b}}[i u]- i W, \, Z_{\alpha, j} 
\rangle
\big|
\\[6pt]
&=
\sum_{1\le j\le 3}
\big|
\langle 
(-\Delta +\alpha)^{-1} V_{+}
\big\{ T_{\mathbf{b}}[u]- W \big\}, \, \partial_{j}W 
\rangle
\big|
\lesssim 
\| T_{\mathbf{b}}[u]- W \|_{L^{6}}
\lesssim 
\delta +r.
\end{split} 
\end{equation}

\noindent 
{\bf Step 2.4.}~Consider the fourth term on the right-hand side of \eqref{19/9/25/17:07}. By \eqref{18/02/24/17:13}, 
\eqref{19/9/27/11:16}, \eqref{19/9/27/14:00}, 
\eqref{19/9/30/18:18}, Corollary \ref{18/12/02/11:25cor} and \eqref{19/10/2/1:1}, 
one can see that 
\begin{equation}\label{19/9/30/18:5}
\begin{split} 
&
\sum_{j=4,5} \sum_{1\le k \le 3}
\alpha^{-\frac{1}{2}} \big|
\partial_{y_{k}} \mathscr{H}_{\alpha, j}(u, \mathbf{b}_{0}) 
-
\partial_{y_{k}} \mathscr{H}_{\alpha, j}(W, \mathbf{b})
\big|
\\[6pt]
&\lesssim 
\sup_{1\le k \le 3}\| T_{\mathbf{b}}
[\partial_{k}u]-\partial_{k}W \|_{L^{2}}
\lesssim \delta +r.
\end{split} 
\end{equation}

\noindent 
{\bf Step 2.5.}~Consider the fifth term on the right-hand side of \eqref{19/9/25/17:07}. 
By \eqref{19/9/25/23:35}, $\Lambda W =\frac{1}{2}W+x\cdot \nabla W$, \eqref{19/9/25/20:2}, \eqref{19/9/25/20:1}, \eqref{19/9/27/14:00}, \eqref{19/9/30/18:18}, Corollary \ref{18/12/02/11:25cor}, 
$|\nu^{-1}-1|\lesssim r$, \eqref{19/10/2/1:3} and \eqref{19/10/2/1:2}, 
\begin{equation}\label{19/10/1/16:22}
\begin{split} 
&
\sum_{j=4,5} 
\alpha^{-\frac{1}{2}}
\big|
\partial_{\nu} \mathscr{H}_{\alpha,j}(W,\mathbf{b}_{0})
-
\partial_{\nu} \mathscr{H}_{\alpha,j}
(u, \mathbf{b})
\big|
\\[6pt]
&\lesssim
\sum_{j=4,5}
\alpha^{-\frac{1}{2}}\big|
\langle \nu^{-1} T_{\mathbf{b}}[u]-W, \ Z_{\alpha, j} 
\rangle
\big|
+ 
\sum_{j=4,5}
\alpha^{-\frac{1}{2}}\big|
\langle \nu^{-1}e^{i\theta}\nu^{-2}x \cdot
\nabla u (\nu^{-2}+y) - x\cdot \nabla W, 
\ Z_{\alpha, j} 
\rangle
\big|
\\[6pt]
& \quad + 
\sum_{j=4,5}\alpha^{-\frac{1}{2}}
|\nu^{-1}-1| 
\big| \langle
\nu^{-1}e^{i\theta} \nu^{-2}x \cdot
\nabla u(\nu^{-2} x+y), Z_{\alpha,j} 
\rangle \big|
\\[6pt]
&\lesssim 
\| \nu^{-1} T_{\mathbf{b}}[ u]- W \|_{L^{6}}
+
\| (1+|x|)^{-1} 
\big\{ \nu^{-1} T_{\mathbf{b}}[x\cdot \nabla u]- x\cdot \nabla W \big\} \|_{L^{2}}
\\[6pt]
&\quad +
|\nu^{-1}-1| 
\|(1+|x|)^{-1} \nu^{-2}|x| 
\nabla u (\nu^{-2}x+y ) \|_{L^{2}}
\\[6pt]
&
\lesssim \delta +r.
\end{split} 
\end{equation}

\noindent 
{\bf Step 2.6.}~Consider the last term on the right-hand side of \eqref{19/9/25/17:07}. By \eqref{18/02/24/17:12}, 
\eqref{19/9/25/11:24}, \eqref{19/9/25/11:23}, 
\eqref{19/9/27/14:00}, \eqref{19/9/30/18:18}, Corollary \ref{18/12/02/11:25cor} and \eqref{19/10/2/1:3}, 
\begin{equation}\label{19/10/1/16:56}
\begin{split} 
&
\sum_{j=4,5} 
\alpha^{-\frac{1}{2}}
\big|
\partial_{\theta} \mathscr{H}_{\alpha,j}(W,\mathbf{b}_{0})
-
\partial_{\theta} \mathscr{H}_{\alpha,j}
(u, \mathbf{b})
\big|
\\[6pt]
&=\sum_{j=4,5}
\alpha^{-\frac{1}{2}}\big|
\langle T_{\mathbf{b}}[i u]- iW, \, Z_{\alpha, j} 
\rangle
\big|
\lesssim 
\| T_{\mathbf{b}}[ u]- W \|_{L^{6}}
\lesssim \delta +r.
\end{split} 
\end{equation}

\noindent 
{\bf Step 2.7.}~We shall derive the estimates 
\eqref{19/9/25/14:39} and \eqref{19/10/5/11:15}. 
Plugging 
the estimates \eqref{19/9/27/13:45} though \eqref{19/10/1/16:56} into \eqref{19/9/25/17:07}, one see that for any $\mathbf{b}\in B(\mathbf{b}_{0},r)$, 
\begin{equation}\label{19/10/4/10:38}
\bigm\| 
A_{\alpha}^{-1}\bigm\{ 
\partial_{\mathbf{b}} \mathscr{H}_{\alpha}(W,\mathbf{b}_{0})
-
\partial_{\mathbf{b}} \mathscr{H}_{\alpha}
(u, \mathbf{b})
\bigm\} 
\bigm\|_{\mathbb{R}^{5}\to \mathbb{R}^{5}}
\lesssim 
\delta +r.
\end{equation}
Moreover, it is easy to verify that if $\mathbf{b}, \mathbf{c} \in B(\mathbf{b}_{0},r)$, then $\mathbf{c} + t (\mathbf{b}-\mathbf{c}) 
\in B(\mathbf{b}_{0},r)$. 
Hence, it follows from \eqref{19/9/25/17:02} 
and \eqref{19/10/4/10:38} that for any 
$u\in B_{\dot{H}^{1}}(W,\delta)$ and 
$\mathbf{b}, \mathbf{c} \in B(\mathbf{b}_{0},r)$, 
\begin{equation}
\begin{split}
&
\big|
\mathscr{G}_{\alpha}(u,\mathbf{b}) 
- \mathscr{G}_{\alpha}(u,\mathbf{c}) 
\big| 
\\[6pt]
&\le 
\sup_{0\le t\le 1}
\big\| 
A_{\alpha}^{-1} 
\bigm\{
\partial_{\mathbf{b}} \mathscr{H}_{\alpha}(W,\mathbf{b}_{0})
-
\partial_{\mathbf{b}} \mathscr{H}_{\alpha}
(u, \mathbf{c} + t (\mathbf{b}-\mathbf{c}))
\Bigm\}
\big\|_{\mathbb{R}^{5}\to \mathbb{R}^{5}}
|\mathbf{b}-\mathbf{c}| 
\\[6pt]
&\lesssim 
(\delta+r) |\mathbf{b}-\mathbf{c}|,
\end{split} 
\end{equation}
so that if $\delta \ll 1$ and $r \ll 1$, then 
\begin{equation}\label{19/10/4/10:43}
\big|
\mathscr{G}_{\alpha}(u,\mathbf{b}) 
- 
\mathscr{G}_{\alpha}(u,\mathbf{c}) 
\big| 
\le \frac{1}{4}|\mathbf{b} - \mathbf{c}| 
.
\end{equation}
This proves \eqref{19/9/25/14:39}. 

On the other hand, by \eqref{18/03/29/14:05}, 
\eqref{19/9/27/13:55} 
through \eqref{19/9/30/18:18}, 
Proposition \ref{19/9/26/15:4}, Corollary \ref{18/12/02/11:25cor} and Sobolev's inequality, one can see that for any 
$u \in B_{\dot{H}^{1}}(W,\delta)$, 
\begin{equation}
\begin{split}
&\big|
\mathscr{G}_{\alpha}(u,\mathbf{b}_{0}) -\mathbf{b}_{0} 
\big| 
=
\big|A_{\alpha}^{-1} \mathscr{H}_{\alpha}(u,\mathbf{b}_{0}) \big|
\\[6pt]
&\le 
\sum_{1\le j\le 3} \alpha^{-1} 
\big| 
\langle 
u- W, \, \alpha V_{+}(-\Delta +\alpha)^{-1} \partial_{j}W 
\rangle
\big| 
\\[6pt]
&\quad +
\alpha^{-\frac{1}{2}}
\big| 
\langle 
u- W, \, \alpha V_{+}(-\Delta +\alpha)^{-1} \Lambda W 
\rangle
\big| 
+\alpha^{-\frac{1}{2}}
\big| 
\langle 
u- W, \, \alpha V_{-}(-\Delta + \alpha)^{-1} (iW) 
\rangle
\big| 
\\[6pt]
&\lesssim
\|u-W\|_{L^{6}}\lesssim 
\|\nabla \{u-W\}\|_{L^{2}} \le \delta,
\end{split} 
\end{equation}
so that if $\delta \ll r$, then 
\begin{equation}\label{19/10/5/13:28}
\big|
\mathscr{G}_{\alpha}(u,\mathbf{b}_{0}) -\mathbf{b}_{0}
\big| \le \frac{r}{4}. 
\end{equation}
Furthermore, it follows from \eqref{19/10/4/10:43} 
and \eqref{19/10/5/13:28} 
that if $\delta \ll r \ll 1$, 
then for any $u \in B_{\dot{H}^{1}}(W,\delta)$ 
and $\mathbf{b} \in B(\mathbf{b}_{0},r)$, 
\begin{equation}\label{19/10/5/12:7}
\begin{split}
\big|
\mathscr{G}_{\alpha}(u,\mathbf{b}) -\mathbf{b}_{0}
\big| 
&\le 
\big|
\mathscr{G}_{\alpha}(u,\mathbf{b}) 
- 
\mathscr{G}_{\alpha}(u,\mathbf{b}_{0}) 
\big| 
+
\big|
\mathscr{G}_{\alpha}(u,\mathbf{b}_{0}) -\mathbf{b}_{0}
\big| 
\\[6pt]
&\le 
\frac{1}{4}|\mathbf{b}-\mathbf{b}_{0}|
+
\frac{r}{4} 
\le 
\frac{r}{2}.
\end{split} 
\end{equation}
Thus, we have proved \eqref{19/10/5/11:15}.

\noindent 
{\bf Step 3.}~We shall finish the proof of the proposition. 
Let $0< \delta_{\rm geo}<1$ and 
$0< r_{\rm geo}<\frac{1}{2}$ be constants given in the previous step, 
and let $0<\alpha <1$. 
Recall that it suffices to prove that for any 
$0< \alpha <1$, there exists a unique continuous 
mapping $\mathbf{b}_{\alpha} 
\colon B_{\dot{H}^{1}}(W,\delta_{\rm geo}) 
\to B(\mathbf{b}_{0},r_{\rm geo})$ satisfying 
\eqref{19/10/5/21:11} and \eqref{19/10/5/21:12}.
To this end, we define the sequence 
$\{ \mathbf{b}_{\alpha, n}\}$ of mappings from 
$B_{\dot{H}^{1}}(W,\delta_{\rm geo})$ to 
$\mathbb{R}^{5}$ to be that for any 
$u \in B_{\dot{H}^{1}}(W,\delta_{\rm geo})$, 
\begin{align}
\label{19/10/5/16:46}
\mathbf{b}_{\alpha, 1}(u)&:=
\mathscr{G}_{\alpha}(u,\mathbf{b}_{0}),
\\[6pt]
\label{19/10/5/16:3}
\mathbf{b}_{\alpha, n+1}(u)&:=
\mathscr{G}_{\alpha}(u,\mathbf{b}_{\alpha, n}(u))
\quad 
\mbox{for $n \ge 1$}.
\end{align}

\noindent 
{\bf Step 3.1.}~We shall show that 
\begin{equation}\label{19/10/5/17:6}
\big| \mathbf{b}_{\alpha, n}(u) -\mathbf{b}_{0} \big| 
\le \frac{r_{\rm geo}}{2}
\quad
\mbox{for any $n\ge 1$ and $u\in 
B_{\dot{H}^{1}}(W,\delta_{\rm geo})$}.
\end{equation}
It follows from \eqref{19/10/5/11:15} that 
\eqref{19/10/5/17:6} is true for $n=1$. 
Furthermore, if \eqref{19/10/5/17:6} holds 
for some $n\ge 2$ 
(hence $\mathbf{b}_{\alpha,n}(u) 
\in B(\mathbf{b}_{0}, r_{\rm geo})$), 
then by \eqref{19/10/5/11:15}, 
\begin{equation}\label{19/10/5/17:12}
\big| \mathbf{b}_{\alpha, n+1}(u) -\mathbf{b}_{0} \big| 
\le 
\big| \mathscr{G}_{\alpha}
(u, \mathbf{b}_{\alpha,n}(u)) - 
\mathbf{b}_{0} \big|
\le \frac{r_{\rm geo}}{2}.
\end{equation}
Thus, by induction, \eqref{19/10/5/17:6} must hold. 

\noindent 
{\bf Step 3.2.}~We shall show that 
\begin{equation}\label{19/10/5/16:20}
\mathbf{b}_{\alpha, n}(W)=\mathbf{b}_{0}
\quad 
\mbox{for any $n\ge 1$}.
\end{equation}
We prove this by induction. It is easy to verify that $\mathbf{b}_{\alpha,1}(W)=\mathbf{b}_{0}$. Suppose that $\mathbf{b}_{\alpha, n}(W)=\mathbf{b}_{0}$ for some $n\ge 2$. Then, one can verify that 
\begin{equation}\label{19/10/5/17:2}
\mathbf{b}_{\alpha, n+1}(W)
=
\mathscr{G}_{\alpha}(u,\mathbf{b}_{\alpha, n}(W)) 
=
\mathscr{G}_{\alpha}(W,\mathbf{b}_{0})=\mathbf{b}_{0}
.
\end{equation}
Thus, we have proved \eqref{19/10/5/16:20}.

\noindent 
{\bf Step 3.3.}~We prove that for any $0< \alpha <1$, there exists a unique continuous mapping $\mathbf{b}_{\alpha} \colon B_{\dot{H}^{1}}(W,\delta_{\rm geo}) \to B(\mathbf{b}_{0},r_{\rm geo})$ satisfying \eqref{19/10/5/21:11} and \eqref{19/10/5/21:12}, which completes the proof of the lemma. 

Let $n\ge 1$ and $u\in B_{\dot{H}^{1}}(W,\delta_{\rm geo})$. Notice that 
by \eqref{19/10/5/17:6}, $\mathbf{b}_{\alpha,n}(u) \in B(\mathbf{b}_{0},r_{\rm geo})$. Hence, it follows from \eqref{19/9/25/14:39} that 
\begin{equation}\label{19/10/6/8:22}
\begin{split}
&\big| \mathbf{b}_{\alpha,n+1}(u)- \mathbf{b}_{\alpha,n}(u)
\big|
=
\big| \mathscr{G}_{\alpha}(u,\mathbf{b}_{\alpha,n}(u) )
- \mathscr{G}_{\alpha}(u,\mathbf{b}_{\alpha,n-1}(u))
\big|
\\[6pt]
&\le \frac{1}{2}
\big| \mathbf{b}_{\alpha,n}(u)- \mathbf{b}_{\alpha,n-1}(u)
\big|
=
\frac{1}{2}\big| \mathscr{G}_{\alpha}(u,\mathbf{b}_{\alpha,n-1}(u) )
- \mathscr{G}_{\alpha}(u,\mathbf{b}_{\alpha,n-2}(u))
\big|
\\[6pt]
&\le \Big( \frac{1}{2} \Big)^{2}
\big| \mathbf{b}_{\alpha,n-1}(u)- \mathbf{b}_{\alpha,n-2}(u)
\big|
\\[6pt]
&\le \cdots \le 
\Big( \frac{1}{2} \Big)^{n}
\big| \mathbf{b}_{\alpha,2}(u)- \mathbf{b}_{\alpha,1}(u)
\big|
=
\Big( \frac{1}{2}\Big)^{n}
\big| \mathscr{G}_{\alpha}(u,\mathbf{b}_{\alpha,1}(u) )
- \mathscr{G}_{\alpha}(u,\mathbf{b}_{0})
\big|
\\[6pt]
&\le 
\Big( \frac{1}{2}\Big)^{n+1}
\big| \mathbf{b}_{\alpha,1}(u) 
- \mathbf{b}_{0}
\big|
\le 
\Big( \frac{1}{2}\Big)^{n+2} r_{\rm geo}
\le 
\Big( \frac{1}{2}\Big)^{n+2}.
\end{split} 
\end{equation}
Furthermore, by the triangle inequality and \eqref{19/10/6/8:22}, one can verify that 
\begin{equation}\label{19/10/5/21:56}
\begin{split} 
\big| \mathbf{b}_{\alpha,n}(u)- \mathbf{b}_{\alpha,m}(u)
\big|
&\le 
\sum_{l=m+1}^{n}
\big| \mathbf{b}_{\alpha,l}(u)- \mathbf{b}_{\alpha,l-1}(u)
\big|
\le 
\sum_{l=m+1}^{n} \Big( \frac{1}{2}\Big)^{l+2}
\\[6pt]
&=
\Big( \frac{1}{2}\Big)^{m+3}
-
\Big( \frac{1}{2}\Big)^{n+2}
\to 0 \quad 
\mbox{as $m,n\to \infty$},
\end{split}
\end{equation}
so that $\{ \mathbf{b}_{\alpha,n}(u) \}$ is Cauchy sequence in $\mathbb{R}^{5}$. Thus, for any $u\in B_{\dot{H}^{1}}(W,\delta_{\rm geo})$, the limit of $\{ \mathbf{b}_{\alpha,n}(u) \}$ exists. Let $u \in B_{\dot{H}^{1}}(W,\delta_{\rm geo})$, and put $\mathbf{b}_{\alpha}(u) :=\lim\limits_{n\to \infty}\mathbf{b}_{\alpha,n}(u)$. Then, using \eqref{19/10/5/17:6}, one can verify that 
\begin{equation}\label{19/10/6/8:47}
\mathbf{b}_{\alpha}(u) \in B(\mathbf{b}_{0},r_{\rm geo})
.
\end{equation}
Moreover, it follows from \eqref{19/10/5/16:20} that 
\begin{equation}\label{19/10/6/9:3}
\mathbf{b}_{\alpha}(W)
=
\lim_{n\to \infty} \mathbf{b}_{\alpha,n}(W)
=
\mathbf{b}_{0}. 
\end{equation} 
Furthermore, by the continuity of the mapping $\mathbf{b}\in B(\mathbf{b}_{0},r_{\rm geo}) \mapsto \mathscr{G}_{\alpha}(u,\mathbf{b}) \in \mathbb{R}^{5}$ (see \eqref{19/9/25/14:39}), one sees that 
\begin{equation}\label{19/10/5/22:38}
\mathbf{b}_{\alpha}(u)
=
\lim_{n\to \infty} \mathbf{b}_{\alpha,n+1}(u)
=
\lim_{n\to \infty}\mathscr{G}_{\alpha}(u,\mathbf{b}_{\alpha,n}(u))
=
\mathscr{G}_{\alpha}(u,\mathbf{b}_{\alpha}(u)). 
\end{equation} 
Thus, we have obtained a mapping $\mathbf{b}_{\alpha}(\cdot) \colon B_{\dot{H}^{1}}(W,\delta_{\rm geo}) \to B(\mathbf{b}_{0},r_{\rm geo})$ satisfying \eqref{19/10/5/21:11} and \eqref{19/10/5/21:12}. 

We shall prove the continuity of the mapping $u\in B_{\dot{H}^{1}}(W,\delta_{\rm geo}) \mapsto \mathbf{b}_{\alpha}(u)$. It follows from \eqref{19/10/5/22:38} and \eqref{19/9/25/14:39} that for any $u,v \in B_{\dot{H}^{1}}(W,\delta_{\rm geo})$,
\begin{equation}\label{19/10/6/7:47}
\begin{split}
&\big| 
\mathbf{b}_{\alpha}(u)- \mathbf{b}_{\alpha}(v)
\big|
\\[6pt]
&=
\big| 
\mathscr{G}_{\alpha}(u,\mathbf{b}_{\alpha}(u))
- 
\mathscr{G}_{\alpha}(v,\mathbf{b}_{\alpha}(v))
\big|
\\[6pt]
&\le 
\big| 
\mathscr{G}_{\alpha}(u,\mathbf{b}_{\alpha}(u))
- 
\mathscr{G}_{\alpha}(u,\mathbf{b}_{\alpha}(v))
\big|
+
\big| 
\mathscr{G}_{\alpha}(u,\mathbf{b}_{\alpha}(v))
- 
\mathscr{G}_{\alpha}(v,\mathbf{b}_{\alpha}(v))
\big|
\\[6pt]
&\le 
\frac{1}{4}|\mathbf{b}_{\alpha}(u)-\mathbf{b}_{\alpha}(v)|
+
\big| A_{\alpha}^{-1}
\big\{
\mathscr{H}_{\alpha}(u,\mathbf{b}_{\alpha}(v))
- 
\mathscr{H}_{\alpha}(v,\mathbf{b}_{\alpha}(v))
\big\} \big| 
.
\end{split}
\end{equation}
Furthermore, by \eqref{19/10/6/7:47}, \eqref{19/9/27/13:55} through \eqref{19/9/30/18:18}, Proposition \ref{19/9/26/15:4}, Corollary \ref{18/12/02/11:25cor} and Sobolev's inequality, one can verify that for any $u,v \in B_{\dot{H}^{1}}(W,\delta_{\rm geo})$,
\begin{equation}\label{19/10/6/7:53}
\begin{split} 
\big| 
\mathbf{b}_{\alpha}(u)- \mathbf{b}_{\alpha}(v)
\big|
&\lesssim 
\big| A_{\alpha}^{-1}
\big\{
\mathscr{H}_{\alpha}(u,\mathbf{b}_{\alpha}(v))
- 
\mathscr{H}_{\alpha}(v,\mathbf{b}_{\alpha}(v))
\big\} \big| 
\\[6pt]
&\lesssim 
\sum_{j=1}^{3} \alpha^{-1}
\big| \langle T_{\mathbf{b}_{\alpha}(v)}[u-v],\ 
\alpha V_{+}(-\Delta +\alpha)^{-1} \partial_{j}W \rangle \big| 
\\[6pt]
&\quad +
\alpha^{-\frac{1}{2}}
\big| \langle T_{\mathbf{b}_{\alpha}(v)}[u-v],\ \alpha V_{+}(-\Delta +\alpha)^{-1} \Lambda W \rangle \big| 
\\[6pt]
&\quad +
\alpha^{-\frac{1}{2}}
\big| \langle T_{\mathbf{b}_{\alpha}(v)}[u-v],\ \alpha V_{-}(-\Delta +\alpha)^{-1} (iW) \rangle \big| 
\\[6pt]
&\lesssim 
\|T_{\mathbf{b}_{\alpha}(v)}[u-v]\|_{L^{6}}
=
\|u-v\|_{L^{6}}
\lesssim
\|u-v\|_{\dot{H}^{1}}
.
\end{split}
\end{equation}
Thus, we have proved the continuity of $\mathbf{b}_{\alpha}(\cdot)$. 

It remains to prove the uniqueness. Suppose that there exists a mapping $\mathbf{c}_{\alpha}\colon B_{\dot{H}^{1}}(W,\delta_{\rm geo}) \to B(\mathbf{b}_{0},r_{\rm geo})$ such that $\mathbf{c}_{\alpha}(W)=\mathbf{b}_{0}$ and $\mathbf{c}_{\alpha}(u)=\mathscr{G}_{\alpha}(u,\mathbf{c}_{\alpha}(u))$ for any $u\in B_{\dot{H}^{1}}(W,\delta_{\rm geo})$. Then, it follows from \eqref{19/9/25/14:39} that for any $u\in B_{\dot{H}^{1}}(W,\delta_{\rm geo})$, 
\begin{equation}\label{19/10/6/10:20}
\big| 
\mathbf{b}_{\alpha}(u)- \mathbf{c}_{\alpha}(u)
\big|
= 
\big| \mathscr{G}_{\alpha}(u,\mathbf{b}_{\alpha}(u))
- 
\mathscr{G}_{\alpha}(u,\mathbf{c}_{\alpha}(u))
\big| 
\le 
\frac{1}{4}\big| 
\mathbf{b}_{\alpha}(u)- \mathbf{c}_{\alpha}(u)
\big|
.
\end{equation}
This implies that $\mathbf{b}_{\alpha}=\mathbf{c}_{\alpha}$. Thus, we have completed the proof. 
\end{proof}

\appendix


\section{Proof of \eqref{sob-bc-rela}} \label{app-A} 

In this section, we give a proof of \eqref{sob-bc-rela}. 

Let $u \in \dot{H}^{1}(\mathbb{R}^d)$ satisfy 
$\mathcal{N}_{\omega}(u) = 0$. 
It follows from \eqref{19/9/7/15:31} and 
$\mathcal{N}_{\omega}(u) = 0$ that 
\[
\|\nabla u\|_{L^{2}}^{2}
= \|u\|_{L^{\frac{2d}{d-2}}}^{\frac{2d}{d-2}} 
\leq \sigma^{- \frac{d}{d-2}} 
\|\nabla u\|_{L^{2}}^{\frac{2d}{d-2}}. 
\]
This yields that 
\[
\sigma^{\frac{d}{2}} \leq 
\|\nabla u\|_{L^{2}}^{2}. 
\]
For any $u \in \dot{H}^{1}(\mathbb{R}^{d})$ 
with $\mathcal{N}_{\omega}(u) = 0$, 
we obtain 
\[
\mathcal{H}^{\ddagger}(u) = 
\frac{1}{d} \|\nabla u\|_{L^{2}}^{2} 
\geq \frac{1}{d} \sigma^{\frac{d}{2}}. 
\]
Taking an infimum on $u \in 
\dot{H}^{1}(\mathbb{R}^{d})$ 
with $\mathcal{N}_{\omega}(u) = 0$, 
we see that 
\begin{equation} \label{m-infty-1}
m_{\infty} \geq \frac{1}{d} \sigma^{\frac{d}{2}}. 
\end{equation}
Let $u \in \dot{H}^{1}(\mathbb{R}^{d})$ 
satisfy $\|u\|_{L^{\frac{2d}{d-2}}} = 1$. 
Putting $w = \|\nabla u\|_{L^{2}}^{\frac{d-2}{2}}u$, 
one has 
\[
\mathcal{N}_{\omega}(w) 
= \|\nabla u\|_{L^{2}}^{d} 
- \|\nabla u\|_{L^{2}}^{d} \|u\|_{L^{\frac{2d}{d-2}}}
^{\frac{2d}{d-2}} 
= \|\nabla u\|_{L^{2}}^{d} 
- \|\nabla u\|_{L^{2}}^{d}
= 0. 
\] 
This implies that 
\[
m_{\infty} \leq 
\mathcal{H}^{\ddagger}(w) = 
\frac{1}{2} \|\nabla u\|_{L^{2}}^{\frac{d}{2}} 
- \frac{d-2}{2d} \|\nabla u\|_{L^{2}}^{d} 
= \frac{1}{d} \|\nabla v\|_{L^{2}}^{d}. 
\]
Taking an infimum on $u 
\in \dot{H}^{1}(\mathbb{R}^{d})$ 
with $\|u\|_{L^{\frac{2d}{d-2}}} = 1$, 
we obtain 
\begin{equation} \label{m-infty-2}
m_{\infty} \leq \frac{1}{d} \sigma^{\frac{d}{2}}. 
\end{equation}
From \eqref{m-infty-1} and \eqref{m-infty-2}, 
we infer that \eqref{sob-bc-rela} holds.


\section{Proof of Proposition \ref{thm-0-1}}
\label{exisistence-groundstate}
In this appendix, 
following \cite[Proposition 1.1]{AIIKN}, 
we shall give a proof of Proposition \ref{thm-0-1}. 

\begin{proof}[Proof of Proposition \ref{thm-0-1}]
From Lemma \ref{variational-character4}, 
it suffices to prove the existence of minimizer for
\begin{equation} \label{mini-1}
\inf\left\{ \mathcal{J}(u) \colon u \in H^{1}(\mathbb{R}^{3}) 
\setminus \{0\}, \ \mathcal{N}_{\omega}(u) \leq 0 \right\}. 
\end{equation}
To this end, we consider a minimizing sequence $\{u_{n}\}$ for \eqref{mini-1}. 
We denote the Schwarz symmetrization of $u_{n}$ by $u_{n}^{*}$. 
Note that $\| \nabla u_n^\ast 
\|_{L^2} \leq \| \nabla u_n \|_{L^2}$ and 
$\| u_n^\ast \|_{L^q} = 
\| u_n \|_{L^q}$ hold for each $q \in [2,2^\ast]$ 
(see e.g. \cite{Lieb-Loss}). 
From these properties 
and Lemma \ref{variational-character4}, we have
\begin{align}
\label{eq:A.6}
&\mathcal{N}_{\omega}(u_{n}^{*})\le 0
\quad 
\mbox{for any $n \in \mathbb{N}$},
\\[6pt]
\label{eq:A.7}
&\lim_{n\to \infty}\mathcal{J}(u_{n}^{*}) = m_{\omega}^{\mathcal{N}}. 
\end{align}
It follows from \eqref{eq:A.6} and \eqref{eq:A.7} that 
\begin{equation}
\|\nabla u_{n}^{*}\|_{L^{2}}^{2} + \omega \|u_{n}^{*}\|_{L^{2}}^{2} 
\leq \|u_{n}^{*}\|_{L^{p+1}}^{p+1} + \|u_{n}^{*}\|_{L^{6}}^{6} \lesssim 1. 
\end{equation}
Thus, we see that 
$\{u_{n}^{*}\}$ is bounded in $H^{1}(\mathbb{R}^{d})$. 
Since $\{u_{n}^{*}\}$ is radially symmetric 
and bounded in $H^{1}(\mathbb{R}^{d})$, 
there exists a radially symmetric function 
$Q_{\omega} \in H^{1}(\mathbb{R}^{d})$ 
such that, passing to some subsequence, 
\begin{equation}
\label{eq:A.8}
\lim_{n\to \infty}u_{n}^{*}=Q_{\omega} 
\quad 
\mbox{weakly in $H^{1}(\mathbb{R}^{d})$ and 
strongly in $L^{p+1}(\mathbb{R}^{d})$}. 
\end{equation}
We shall show that $Q_{\omega}$ 
becomes a minimizer for $m_{\omega}^{\mathcal{N}}$.

We first claim that $Q_{\omega} \not\equiv 0$. 
Suppose the contrary that $Q_{\omega} \equiv 0$. 
Then, it follows from (\ref{eq:A.6}) and (\ref{eq:A.8}) that, passing to some subsequence, 
\begin{equation}\label{eq:A.9}
0\ge \lim_{n\to \infty}\mathcal{N}_{\omega}(u_{n}^{*})
\ge \lim_{n\to \infty}\left\{ 
\left\| \nabla u_{n}^{*} \right\|_{L^{2}}^{2}
-\left\| u_{n}^{*} \right\|_{L^{6}}^{6}
\right\}.
\end{equation}
If $\| \nabla u_n^* \|_{L^2} \to 0$, then 
it follows from the boundedness of $\{u_{n}^{*}\}$ in $H^{1}(\mathbb{R}^{3})$
and H\"{o}lder's inequality that 
$\| u_n^* \|_{L^q} \to 0$ for all $2<q \leq 6$. 
This together with \eqref{eq:A.7} yields that 
\begin{equation}
m_{\omega}^{\mathcal{N}} = \lim_{n \to \infty} \mathcal{J}(u_{n}^{*}) 
= 0.
\end{equation} 
However, this contradicts \eqref{m-lower}. 
Therefore, by taking a subsequence, 
we may assume $\lim_{n\to\infty} \| \nabla u_n^* \|_{L^2} > 0$.

Now, \eqref{eq:A.9} with the definition of $\sigma$ gives us
\begin{equation}\label{19/10/9/14:38}
\lim_{n\to \infty}\left\| \nabla u_{n}^{*} \right\|_{L^{2}}^{2}
\ge 
\sigma 
\lim_{n\to \infty}\left\| u_{n}^{*} \right\|_{L^{6}}^{2}
\ge 
\sigma \lim_{n\to \infty}
\left\| \nabla u_{n}^{*} \right\|_{L^{2}}^{\frac{2}{3}}.
\end{equation}
From this together with 
$\lim_{n\to\infty} \| \nabla u_n^* \|_{L^2} >0$ 
and \eqref{eq:A.9}, we have 
\begin{equation}\label{eq:A.10}
\sigma^{\frac{3}{2}}
\le \lim_{n\to \infty}\left\| \nabla u_{n}^{*} \right\|_{L^{2}}^2 \leq 
\lim_{n \to \infty} \left\|u_{n}^{*} \right\|_{L^{6}}^{6}.
\end{equation}
Hence, we see from \eqref{eq-J}, 
\eqref{eq:A.7} and \eqref{19/10/9/14:38} that 
\[
\begin{split}
m_{\omega}^{\mathcal{N}}
&=
\lim_{n\to \infty}\mathcal{J}(u_{n}^{*})
\\[6pt]
&\ge 
\lim_{n\to \infty}
\left\{ 
\frac{p-1}{2(p+1)}\left\|u_{n}^{*} \right\|_{L^{p+1}}^{p+1}
+
\frac{1}{3}\left\| u_{n}^{*} \right\|_{L^{6}}^{6}
\right\}
\\[6pt]
&\ge \frac{1}{3}\lim_{n\to \infty}
\left\|u_{n}^{*} \right\|_{L^{6}}^{6}
\ge m_{\infty}.
\end{split}
\]
However, this contradicts the assumption. 
Thus, $Q_{\omega} \not\equiv 0$.

Using Lemma \ref{thm-bl}, we have
\begin{align}
\label{eq:A.11}
\mathcal{J}(u_{n}^{*})
-
\mathcal{J}(u_{n}^{*}-Q_{\omega})
-
\mathcal{J}(Q_{\omega})
&=o_{n}(1),
\\[6pt]
\label{eq:A.12}
\mathcal{N}_{\omega}(u_{n}^{*}) 
- \mathcal{N}_{\omega}(u_{n}^{*}-Q_{\omega})
-\mathcal{N}_{\omega}(Q_{\omega})
&=o_{n}(1)
.
\end{align} 
Furthermore, \eqref{eq:A.11} together with \eqref{eq:A.7} and the positivity of $\mathcal{J}$ implies that
\begin{equation}\label{eq:A.13}
\mathcal{J}(Q_{\omega}) \le m_{\omega}^{\mathcal{N}}.
\end{equation}

Next, we shall show that 
$\mathcal{N}_{\omega}(Q_{\omega}) \leq 0$. 
Suppose that $\mathcal{N}_{\omega}(Q_{\omega})>0$. 
Then, it follows from \eqref{eq:A.6} and \eqref{eq:A.12} that $\mathcal{N}_{\omega}(u_{n}^{*}-Q_{\omega})<0$ 
for sufficiently large $n$. 
Hence, we can take $\lambda_{n}\in 
(0,1)$ such that 
$\mathcal{N}_{\omega}
(\lambda_{n}(u_{n}^{*}-Q_{\omega}))=0$. 
Furthermore, we see from $0<\lambda_{n}<1$, \eqref{eq:A.7}, \eqref{eq:A.11} 
and $Q_{\omega} \neq 0$ that 
\begin{equation}
\begin{split}
m_{\omega}^{\mathcal{N}}
&\le 
\mathcal{J}( \lambda_{n} (u_{n}^{*}-Q_{\omega}))
< 
\mathcal{J}(u_{n}^{*}-Q_{\omega})
=\mathcal{J}(u_{n}^{*})
-
\mathcal{J}(Q_{\omega})+o_{n}(1)
\\[6pt]
&= m_{\omega}^{\mathcal{N}} -
\mathcal{J}(Q_{\omega})+o_{n}(1) < m_{\omega}^{\mathcal{N}} 
\end{split}
\end{equation}
for sufficiently large $n \in \mathbb{N}$, 
which is a contradiction. 
Thus, $\mathcal{N}_{\omega}(Q_{\omega}) 
\leq 0$.

Since $Q_{\omega} \not\equiv 0$ 
and $\mathcal{N}_{\omega}(Q_{\omega}) \leq 0$, 
it follows from Lemma \ref{variational-character4} 
that 
\begin{equation}\label{eq:A.15}
m_{\omega}^{\mathcal{N}}
\le \mathcal{J}(Q_{\omega}). 
\end{equation}
Moreover, it follows from 
the weak lower semicontinuity that 
\begin{equation}\label{eq:A.16}
\mathcal{J}(Q_{\omega}) 
\le 
\liminf_{n\to \infty}\mathcal{J}(u_{n}^{*})
\le m_{\omega}^{\mathcal{N}}.
\end{equation} 
Combining (\ref{eq:A.15}) and (\ref{eq:A.16}), we obtain $\mathcal{J}(Q_{\omega})=m_{\omega}^{\mathcal{N}}$. 
Thus, we have proved that $Q_{\omega}$ 
is a minimizer for $m_{\omega}^{\mathcal{N}}$. 
\end{proof}

\section{Proofs of Proposition \ref{thm4-1}} 
\label{section-4}
In this section, we give a proof of Proposition \ref{thm4-1}.

\begin{proof}[Proof of Proposition \ref{thm4-1}]
Let $\omega^{\prime}>0$. 
We see from Lemma 
\ref{variational-character4} that for any $\varepsilon>0$, 
there exists a nontrivial function $u_{\varepsilon} \in H^{1}
(\mathbb{R}^{3})$ such that $\mathcal{N}_{\omega^{\prime}}(u_{\varepsilon})\le 0$ and $\mathcal{J}(u_{\varepsilon})
\le m_{\omega^{\prime}}^{\mathcal{N}} +\varepsilon$. 
Observe that $\mathcal{N}_{\omega}(u_{\varepsilon}) 
< \mathcal{N}_{\omega^{\prime}}(u_{\varepsilon})$ 
for any $0< \omega < \omega^{\prime}$. 
Hence, Lemma \ref{variational-character4} together with the properties of $u_{\varepsilon}$ shows that for any 
$0< \omega < \omega^{\prime}$, 
\begin{equation}\label{19/9/8/9:25}
m_{\omega}^{\mathcal{N}} 
\le 
\mathcal{J}(u_{\varepsilon})
\le 
m_{\omega^{\prime}}^{\mathcal{N}} +\varepsilon.
\end{equation}
Taking $\varepsilon \to 0$ in \eqref{19/9/8/9:25} yields $m_{\omega}^{\mathcal{N}} \le m_{\omega^{\prime}}^{\mathcal{N}}$. 
This proves the proposition. 
\end{proof}


\section{Compactness in $\boldsymbol{L^{p}}$}\label{18/11/11/15:30}

We record a compactness theorem in $L^{p}$ (see Proposition A.1 of \cite{Killip-Visan-book}):
\begin{lemma}\label{18/11/10/13:29} 
Let $d\ge 1$ and $1\le p < \infty$. A sequence $\{f_{n}\}$ in $L^{p}(\mathbb{R}^{d})$ has a convergent subsequence if and only if is satisfies the following properties:
\\
\noindent 
{\rm (i)}~There exists $A>0$ such that 
\begin{equation}\label{18/11/07/10:16}
\|f_{n} \|_{L^{p}}\le A
\end{equation}
for all $n\ge 1$. 
\\
\noindent 
{\rm (ii)}~For any $\varepsilon>0$ there exists $\delta>0$ such that 
\begin{equation}\label{18/11/07/10:17}
\int_{\mathbb{R}^{d}} |f_{n}(x)-f_{n}(x+y) |\,dx < \varepsilon
\end{equation}
for all $n\ge 1$ and any $|y|<\delta$.
\\
\noindent 
{\rm (iii)}~For any $\varepsilon>0$ there exists $R>0$ such that 
\begin{equation}\label{18/11/07/10:18}
\int_{|x|\ge R} |f_{n}(x)|^{p} \,dx < \varepsilon
\end{equation}
for all $n\ge 1$. 
\end{lemma}

\section*{Acknowledgements}
The work of S.I. was supported by NSERC grant (371637-2014).
The work of H.K. was supported by JSPS KAKENHI Grant Number 
JP17K14223.

\bibliographystyle{plain}

\begin{thebibliography}{99}

\bibitem{AIIKN}
Akahori T., Ibrahim S., Norihisa Ikoma, Kikuchi H. and Nawa H.,
Uniqueness of ground states to nonlinear scalar field equations involving the Sobolev critical exponent in their nonlinearities for high frequencies, 
Calc. Var. and PDE (2019).

\bibitem{AIKN}
Akahori, T., Ibrahim, S., Kikuchi, H. and Nawa, H.,
Existence of a ground state and blow-up for a nonlinear Schr\"odinger equation with critical growth. 
Differential and Integral Equations {\bf 25} (2012), 383--402.

\bibitem{AIKN3}
Akahori, T., Ibrahim, S., Kikuchi, H. and Nawa, H.,
Global dynamics for a nonlinear Schr\"odinger equation above a ground state with small frequency. To appear in Memoir of AMS. (arxiv:1510.08034).



\bibitem{AKN} 
Akahori, T., Kikuchi, H. and Nawa, H.,
Scattering and blowup problems for a class of nonlinear Schr\"{o}dinger equations. Differential Integral Equations {\bf 25} (2012), 1075--1118.

\bibitem{Aubin}
Aubin, T., 
\'{E}quations diff\'{e}rentielles non 
lin\'{e}aires et probl\`{e}me de 
Yamabe concernant la courbure scalaire.
J. Math. Pures Appl. (9) 
{\bf 55} (1976), 269--296.


\bibitem{Berestycki-Lions}
Berestycki, H. and Lions, P.-L., 
Nonlinear scalar field equations. I. Existence of a ground state. 
Arch. Rational Mech. Anal. {\bf 82} (1983), 313--345.

\bibitem{Brezis-Lieb}
Brezis, H. and Lieb, E. H., 
A relation between pointwise convergence of functions and convergence of functionals. 
Proc. Amer. Math. Soc. {\bf 88} (1983), 486--490. 


\bibitem{Brezis-Nirenberg}
Brezis, H. and Nirenberg, L., 
Positive solutions of nonlinear elliptic equations involving critical Sobolev 
exponents. Comm. Pure Appl. Math. {\bf 36} (1983), 437--477.
\bibitem{CG} 
Coles, M. and Gustafson, S., 
Solitary Waves and Dynamics for 
Subcritical Perturbations of Energy 
Critical NLS.
Pub. RIMS Kyoto U. {\bf 56} (2020), 
647--699. 

\bibitem{CMR}
Collot, C., Merle F., Rapha\"ol P., 
Dynamics near the ground state for the energy critical nonlinear heat equation in large dimensions.
Comm. Math. Phys. {\bf 352} (2017). 

\bibitem{DM}
Duyckaerts, T., Merle, F., 
Dynamics of threshold solution for energy-critical NLS. 
Geom. Funct. Anal. 
\textbf{18} (2009), 178--1840. 

\bibitem{FLL}
Fr\"ohlich, J., Lieb, E. H. and Loss, M., 
Stability of coulomb systems 
with magnetic fields I, The one-electron atom.
Commun. Math. Phys. {\bf 78} (1989), 160--190. 



\bibitem{GNN}
Gidas, B., Ni, W. M. and Nirenberg, L., 
Symmetry of positive solutions of nonlinear elliptic equations in $\mathbf{R}^n$. 
Mathematical analysis and applications, Part A, pp. 369--402, 
Adv. in Math. Suppl. Stud., 7a, Academic Press, New York-London, 1981.


\bibitem{GT}
Gilbarg, D. and Trudinger, N. S., 
Elliptic partial differential equations of second order. 
Second edition. Grundlehren der Mathematischen Wissenschaften, {\bf 224}. 
Springer-Verlag, Berlin, 1983.





\bibitem{Jensen-Kato}
Jensen, A. and Kato, T.,
Spectral properties of Schr\"{o}dinger operators and time-decay of the wave functions.
Duke Math. J. {\bf 46} (1979), 583--611.


\bibitem{Killip-Visan}
Killip, R. and Visan, M., 
The focusing energy-critical nonlinear Schr\"odinger equation in dimensions five and higher. 
Amer. J. Math. {\bf 132} (2010), 361--424.

\bibitem{Killip-Visan-book}
Killip, R. and Visan, M.,
Nonlinear Schr\"odinger equations at critical regularity, 
Clay Lecture Notes. 


\bibitem{Lieb}
Lieb E.H., 
On the lowest eigenvalue of Laplacian for the intersection of two domains. 
Invent. Math. {\bf 74} (1983), 441--448. 



\bibitem{Lieb-Loss} 
Lieb, E. A. and Loss, M., 
Analysis.
Second edition. Graduate Studies in Mathematics, 
14. American Mathematical Society, Providence, RI, 2001. 

\bibitem{Linares-Ponce}
Linares, F. and Ponce, G. 
Introduction to Nonlinear Dispersive Equations, Second edition, Springer-Verlag, New York, 2015.





\bibitem{RS}
Reed, M., Simon, B., 
Methods of modern mathematical physics. I. Functional analysis, 
Academic press, New York--London, (1978). 

\bibitem{Struwe}
Struwe, M., 
Variational Methods. 
Applications to Nonlinear Partial Differential Equations and Hamiltonian Systems, 
4th ed., Springer, Berlin, 2008. 



\bibitem{Talenti}
Talenti, G., 
Best constant in Sobolev inequality.
Ann. Mat. Pura Appl. (4) {\bf 110} (1976), 353--372.


\bibitem{Zhang-Zou}
Zhang, J. and Zou, W.,
The critical case for a Berestycki-Lions theorem, 
Science China Math. {\bf 57} (2014), 541--554.

\end{thebibliography}

\vspace{24pt}

\end{document}